\documentclass[11pt]{amsart}

\usepackage{amsmath}
\usepackage{amsfonts}
\usepackage{amssymb}
\usepackage{amscd}
\usepackage{amsthm}
\usepackage{framed}
\usepackage{fullpage}
\usepackage{graphicx}
\usepackage{latexsym}
\usepackage[numbers]{natbib}

\usepackage{hyperref}
\hypersetup{colorlinks,linkcolor={red},citecolor={blue},urlcolor={blue}}

\newtheorem{theorem}{Theorem}[section]
\newtheorem{lemma}[theorem]{Lemma}

\newtheorem{corollary}[theorem]{Corollary}
\newtheorem{conjecture}[theorem]{Conjecture}

\theoremstyle{definition}
\newtheorem{definition}[theorem]{Definition}
\newtheorem{example}[theorem]{Example}
\newtheorem{remark}[theorem]{Remark}
\newtheorem{condition}[theorem]{Condition}
\newtheorem{assumption}[theorem]{Assumption}
\newtheorem{Notation}[theorem]{Notation}

\numberwithin{equation}{section}

\newcommand{\CC}{\mathbb C}
\newcommand{\HH}{\mathbb H}

\newcommand{\NN}{\mathbb N}
\newcommand{\cD}{\mathcal D}
\newcommand{\cA}{\mathcal A}
\newcommand{\cH}{\mathcal H}

\newcommand{\cE}{\mathcal E}
\newcommand{\PP}{\mathbb P}
\newcommand{\QQ}{\mathbb Q}
\newcommand{\RR}{\mathbb R}
\newcommand{\ZZ}{\mathbb Z}
\newcommand{\F}{\mathbb F}
\newcommand{\GL}{\mathop{\mathrm {GL}}\nolimits}
\newcommand{\SL}{\mathop{\mathrm {SL}}\nolimits}

\newcommand{\Sp}{\mathop{\mathrm {Sp}}\nolimits}
\newcommand{\Orth}{\mathop{\null\mathrm {O}}\nolimits}
\newcommand{\U}{\mathop{\null\mathrm {U}}\nolimits}
\newcommand{\im}{\mathop{\mathrm {Im}}\nolimits}
\newcommand{\rk}{\mathop{\mathrm {rk}}\nolimits}
\newcommand{\latt}[1]{{\langle{#1}\rangle}}
\newcommand{\ord}{\mathop{\mathrm {ord}}\nolimits}
\newcommand{\orb}{\mathop{\mathrm {orb}}\nolimits}
\def\Grit{\operatorname{Grit}}
\def\Borch{\operatorname{Borch}}
\def\Div{\operatorname{div}}
\def\dim{\operatorname{dim}}

\def\det{\operatorname{det}}
\def\w{\operatorname{w}}

\newenvironment{psmallmatrix}
  {\left(\begin{smallmatrix}}
{\end{smallmatrix}\right)}

\begin{document}

\title[Modular forms with poles on hyperplane arrangements]{Modular forms with poles on hyperplane arrangements}

\author{Haowu Wang}

\address{Center for Geometry and Physics, Institute for Basic Science (IBS), Pohang 37673, Korea}

\email{haowu.wangmath@gmail.com}

\author{Brandon Williams}

\address{Lehrstuhl A für Mathematik, RWTH Aachen, 52056 Aachen, Germany}

\email{brandon.williams@matha.rwth-aachen.de}

\subjclass[2020]{11F55, 32S22, 17B22, 11F50, 11F46}

\date{\today}

\keywords{Modular forms on symmetric domains, Orthogonal groups, Unitary groups, Hyperplane arrangements, Looijenga compactification, Root systems, Jacobi forms, Theta blocks conjecture}

\begin{abstract} 
We study algebras of meromorphic modular forms whose poles lie on Heegner divisors for orthogonal and unitary groups associated to root lattices. We give a uniform construction of $147$ hyperplane arrangements on type IV symmetric domains for which the algebras of modular forms with constrained poles are free and therefore the Looijenga compactifications of the arrangement complements are weighted projective spaces. We also construct $8$ free algebras of modular forms on complex balls with poles on hyperplane arrangements. The most striking example is the discriminant kernel of the $2U\oplus D_{11}$ lattice, which admits a free algebra on $14$ meromorphic generators. Along the way, we determine minimal systems of generators for non-free algebras of orthogonal modular forms for $26$ reducible root lattices and prove the modularity of formal Fourier--Jacobi series associated to them. By exploiting an identity between weight one singular additive and multiplicative lifts on $2U\oplus D_{11}$, we prove that the additive lift of any (possibly weak) theta block of positive weight and $q$-order one is a Borcherds product; the special case of holomorphic theta blocks of one elliptic variable is the theta block conjecture of Gritsenko, Poor and Yuen.  
\end{abstract}

\maketitle

\begin{small}
\tableofcontents
\end{small}

\addtocontents{toc}{\setcounter{tocdepth}{1}} 
%to hide subsections of introduction in the table contents.

\section{Introduction}
\subsection{Modular forms with poles on hyperplane arrangements}
Geometric invariant theory provides a means of constructing and compactifying many interesting moduli spaces. On the other hand, global Torelli theorems often yield identifications of moduli spaces with locally symmetric varieties under the Hodge-theoretic period map. It is natural to compare the GIT-compactifications of these moduli spaces with the Baily--Borel compactifications of the related arithmetic quotients. The moduli space of polarized $\mathrm{K3}$ surfaces of a fixed degree can be identified with the quotient of a symmetric domain of type IV and dimension $19$ by an arithmetic group, and this identification extends to the respective compactifications. Similar identifications also hold for the moduli spaces of Enriques surfaces and Del Pezzo surfaces. 
However, for many interesting geometric objects, the image of the moduli space under the period map is not the full quotient of a Hermitian symmetric domain but rather the complement of a hyperplane arrangement.

In 2003, Looijenga \cite{Loo03a, Loo03b} constructed a compactification for the complements of certain hyperplane arrangements in arithmetic quotients of complex balls and type IV symmetric domains. The Looijenga compactification is, roughly speaking, an interpolation between the Baily--Borel and toroidal compactifications, and it coincides with the Baily--Borel compactification if the hyperplane arrangement is empty. For a number of moduli spaces, it has been proved that the GIT compactifications are identified with the Looijenga compactifications via an extension of the period map; for example, the moduli spaces of quartic curves \cite{Kon00}, genus four curves \cite{Kon02}, rational elliptic surfaces \cite{HL02}, cubic surfaces \cite{ACT02}, cubic threefolds \cite{ACT11, LS07} and cubic fourfolds \cite{Loo09, Laz10}.

Let $\cD$ be a Hermitian symmetric domain and let $\Gamma$ be a congruence subgroup acting on $\cD$.  The Baily--Borel compactification of $\cD / \Gamma$ can be understood as the $\mathrm{Proj}$ of the algebra of holomorphic modular forms on $\cD$ for $\Gamma$ \cite{BB66}. Let $\mathcal{H}$ be a $\Gamma$-invariant arrangement of hyperplanes on $\cD$. Assume that $\mathcal{H}$ satisfies the \emph{Looijenga condition} which guarantees in particular that a generalized Koecher's principle holds; that is, the algebra $M_*^!(\Gamma)$ of modular forms on $\cD$ for $\Gamma$ with poles contained in $\mathcal{H}$ is generated by forms of positive weight.  Looijenga proved that $M_*^!(\Gamma)$ is finitely generated and its $\mathrm{Proj}$ characterizes the Looijenga compactification of the complement $(\cD - \mathcal{H}) / \Gamma$.

Holomorphic modular forms on symmetric domains have a very rich theory. The study of (holomorphic) modular forms on orthogonal groups has important applications to infinite-dimensional Lie algebras and birational geometry; for example, the classification of generalized Kac--Moody algebras \cite{Bor95, GN98, Sch06}, the proof that moduli spaces of polarized $\mathrm{K3}$ surfaces of degree larger than $122$ are of general type \cite{GHS07}, and the proof of the finiteness of orthogonal modular varieties not of general type \cite{Ma18}. Modular forms with singularities on hyperplane arrangements, by contrast, have attracted less attention. There are only a few explicit structure theorems for these algebras in the literature. In this paper we investigate the automorphic side of the Looijenga compactification, and especially the construction of free algebras of meromorphic modular forms for which these compactifications are simple weighted projective spaces.  

In \cite{Wan21a} the first named author found necessary and sufficient conditions for an algebra of modular forms on $\Orth(l, 2)$ to be free. These conditions rely on the Jacobian of a set of potential generators having simple zeros exactly on the mirrors of reflections in the modular group. By Bruinier's converse theorem \cite{Bru02, Bru14} such a Jacobian must be a Borcherds product \cite{Bor98}. The necessary condition yields an explicit classification of free algebras of orthogonal modular forms \cite{Wan21a}.  Using the sufficient part of the criterion, we constructed a number of free algebras of orthogonal modular forms \cite{Wan21RNT, WW20b, WW20c}. 
In \cite{WW21a} we extended this approach to modular forms on complex balls attached to unitary groups of signature $(l,1)$. In this paper we further extend this argument to modular forms with singularities on hyperplane arrangements. Our first main theorem is

\begin{theorem}\label{MTH1}
Let $\cD_l$ be a symmetric domain of type IV and dimension $l\geq 3$ or a complex ball of dimension $l\geq 2$. Let $\Gamma$ be a congruence subgroup of $\Orth(l,2)$ or $\U(l,1)$ acting on $\cD_l$. Let $\cH$ be a $\Gamma$-invariant arrangement of hyperplanes satisfying the Looijenga condition. Then the algebra of modular forms on $\cD$ for $\Gamma$ with poles supported on $\mathcal{H}$ is freely generated by $l+1$ forms if and only if the Jacobian $J$ of the $l+1$ potential generators vanishes with multiplicity $d_\sigma-1$ on mirrors of reflections $\sigma$ in $\Gamma$ which are not contained in $\mathcal{H}$, and the other zeros and poles of $J$ are contained in $\mathcal{H}$, where $d_\sigma$ is the order of  $\sigma$. 
\end{theorem}

By applying the Jacobian criterion to symmetric domains of type IV attached to root lattices, we obtain the structure of the following algebras:
\begin{theorem}\label{MTH2}
Let $L$ be any lattice in the following three families of root lattices:
\begin{align*}
&\text{$A$-type:}& & \left(\bigoplus_{j=1}^t A_{m_j}\right)\oplus A_m, \quad t\geq 0, \quad m\geq 1, \quad (m+1)+\sum_{j=1}^t (m_j+1) \leq 11; \\
&\text{$AD$-type:}& & \left(\bigoplus_{j=1}^t A_{m_j}\right)\oplus D_m,  \quad t\geq 0, \quad m\geq 4, \quad   m + \sum_{j=1}^t (m_j+1) \leq 11; \\
&\text{$AE$-type:}& & E_6,\quad A_1 \oplus E_6, \quad A_2\oplus E_6, \quad E_7,\quad  A_1\oplus E_7.
\end{align*}
Let $U$ be the even unimodular lattice of signature $(1,1)$. Let $\widetilde{\Orth}^+(2U\oplus L)$ denote the subgroup of $\Orth(2U\oplus L)$ which respects the symmetric domain and acts trivially on the discriminant group $L'/L$. Then there exists an arrangement of hyperplanes $\mathcal{H}$ such that the ring of modular forms for $\widetilde{\Orth}^+(2U\oplus L)$ with poles supported on $\mathcal{H}$ is a polynomial algebra. 
\end{theorem}

If we write $L = L_0 \oplus L_1$ with $L_1 = A_m, D_m, E_6, E_7$, then the hyperplane arrangement $\cH$ is a union $\cH_{L, 0} \cup \cH_{L, 1}$, where $\cH_{L, 0}$ consists of the orbits of hyperplanes $v^{\perp}$ for vectors $v\in L_0'$ of minimal norm in the dual of each component of $L_0$, and where $\cH_{L, 1}$ consists of the orbits of hyperplanes $(r+u)^{\perp}$, where $u\in L'$ are vectors of minimal norm in their cosets in $L'/L$ that satisfy $u^2 > 2$, and $r\in U$ with $(r,r)=-2$. This is made more precise in Theorem \ref{th:2precise}. There are $147$ hyperplane arrangements in total.

When $L = A_m$ for $1 \leq m \leq 7$ or $L = D_m$ for $4 \leq m \leq 8$ or $L = E_6, E_7$, the arrangement $\mathcal{H}$ above is empty. The corresponding free algebras of holomorphic modular forms have been constructed in \cite{Igu62, FH00, DK03, Kri05, Vin10, Vin18, WW20a}. By \cite{Wan21a}, these $14$ algebras and the algebra associated to $E_8$ computed in \cite{HU14} are the only free algebras of holomorphic modular forms for groups of type $\widetilde{\Orth}^+(2U\oplus M)$. The free algebra associated to $L=A_1\oplus A_2$ was determined in \cite{Nag21} using the period map between the moduli space of $U\oplus E_7\oplus E_6$-polarised $\mathrm{K3}$ surfaces and the arithmetic quotient attached to $2U\oplus A_1\oplus A_2$. The case $A_1\oplus A_1$ is related to Hermitian modular forms of degree two over $\QQ(\sqrt{-1})$. The remaining $131$ free algebras of modular forms with poles on (nonempty) arrangements seem to be new, and we expect them to also have interpretations of moduli spaces of lattice polarised $\mathrm{K3}$ surfaces.  We remark that Hermitian modular forms of degree two and Siegel modular forms of degree two have been treated by a similar method in our previous papers \cite{WW21c, WW21d}.

Vinberg and Shvartsman \cite{VS17} showed that the ring of holomorphic modular forms on symmetric domains of type IV and dimension $l>10$ is never free. However, Theorem \ref{MTH2} includes a number of free algebras of meromorphic modular forms in dimension $l > 10$. In the most extreme case, the algebra of meromorphic modular forms associated to $L=D_{11}$ will turn out to be freely generated in weights $1$, $4$, $4$, $6$, $6$, $8$, $8$, $10$, $10$, $12$, $12$, $14$, $16$ and $18$; the associated variety has dimension $13$.

Some of the lattices in Theorem \ref{MTH2} have complex multiplication over $\QQ(\sqrt{-1})$ or $\QQ(\sqrt{-3})$. If in addition the restriction $\cH_{\U}$ of $\cH$ to the complex ball satisfies the Looijenga condition, then the associated algebra of unitary modular forms whose poles are contained in $\cH_{\U}$ is also free. This is described in Theorem \ref{th:algebras-unitary}. For example, $2U\oplus D_{10}$ has complex multiplication over $\QQ(\sqrt{-1})$, and the ring of meromorphic modular forms on the $6$-dimensional complex ball attached to $D_{10}$ is freely generated in weights $2$, $4$, $8$, $8$, $12$, $12$, $16$. This modular group is an example of a finite-covolume reflection group acting on complex hyperbolic space of dimension $6$ with quotient birational to $\mathbb{CP}^6$.

Let us sketch the idea of the proof of Theorem \ref{MTH2}. To apply Theorem \ref{MTH1}, we need to construct the potential generators, verify that their Jacobian is nonzero, and that it has the zero divisor prescribed by Theorem \ref{MTH1}. In general, there seems to be no rule to characterize the weights of the generators; however, when the underlying lattice is related to a root system, we find that the weights of the generators can be predicted in terms of invariants of the root system, using the theory of Jacobi forms. Jacobi forms are connected to modular forms on orthogonal groups by the Fourier--Jacobi expansion (cf. \cite{EZ85, Gri94}). Wirthm\"uller \cite{Wir92} proved that for any irreducible root system $R$ other than $E_8$, the ring of weak Jacobi forms associated to its root lattice that are invariant under the Weyl group is a polynomial algebra, and the weights and indices of its generators are natural invariants of $R$. We found in \cite{WW20a} that the weights of the generators of a free algebra of holomorphic modular forms related to a root system are simply $k_i + 12t_i$, where $(k_i, t_i)$ are the weights and indices of the generators of the ring of Weyl-invariant weak Jacobi forms. The weights of the generators of the rings of meromorphic modular forms in Theorem \ref{MTH2} are more complicated (cf. Theorem \ref{th:2precise}) and are controlled by exceptional generators related to theta blocks (cf. \cite{GSZ19}). The first part $\cH_{L, 0}$ of the hyperplane arrangement corresponds to the divisors of these theta blocks. The generators are constructed in terms of Borcherds' additive and multiplicative singular theta lifts \cite{Bor95, Bor98} and we will be able to read their algebraic independence off of their leading Fourier--Jacobi coefficients. To compute the divisor of their Jacobian, we prove the existence of a preimage of the Jacobian under Borcherds' lift. We do not construct their preimages directly; instead, we prove their existence in a uniform way by an argument similar to that used by the first named author \cite{Wan19} in the classification of reflective modular forms.

It was proved in \cite{HU14} that the ring of holomorphic modular forms for $\widetilde{\Orth}^+(2U\oplus E_8)$ is also free. However, we know from \cite{Wan21b} that the ring of Weyl-invariant weak Jacobi forms for $E_8$ is not a polynomial algebra. For this reason, the proof of Theorem \ref{MTH2} does not apply to this case. For $9\leq m\leq 11$ we have $2U\oplus D_m \cong 2U\oplus E_8 \oplus D_{m-8}$, but it is not clear to us how the root lattice structure of the latter model is related to the algebra structure. The restrictions in Theorem \ref{MTH2} are natural; if the condition ‘‘$\leq 11$" in Theorem \ref{MTH2} fails to hold, then the Looijenga condition never holds; and if we define the arrangement in other ways then it seems that either the Looijenga condition fails or the algebra of meromorphic modular forms is non-free.

\subsection{The theta block conjecture}
The generators of ``abelian type" (cf. Theorem \ref{th:2precise}) of algebras of meromorphic modular forms are constructed as rational expressions in the Gritsenko lifts of basic Jacobi forms. In order to control the divisor that results, we will prove and then use at several points that fact that the (possibly singular) Gritsenko lift of any lattice-valued pure theta block of $q$-order one is a Borcherds product. This result is of independent interest: the case of holomorphic theta blocks associated to lattices of rank one is the \emph{Theta Block Conjecture} of Gritsenko--Poor--Yuen \cite{GPY15}, and it is closely related to the question of which modular forms are simultaneously Borcherds products and Gritsenko lifts.

The theory of theta blocks is developed in the paper \cite{GSZ19} of Gritsenko--Skoruppa--Zagier. The (pure) theta block associated to a function $f: \NN \to \NN$ with finite support is defined as
\begin{equation}
\Theta_f(\tau,z)=\eta(\tau)^{f(0)} \prod_{a=1}^\infty (\vartheta(\tau,a z)/\eta(\tau))^{f(a)}, \quad (\tau, z) \in \HH \times \CC,
\end{equation}
where $\eta(\tau)=q^{1/24}\prod_{n=1}^\infty(1-q^n)$ is the Dedekind eta function and $\vartheta$ is the odd Jacobi theta function
\begin{equation}
\vartheta(\tau,z)=q^{1/8}(e^{\pi iz}-e^{-\pi i z})\prod_{n=1}^\infty (1-q^ne^{2\pi i z})(1-q^ne^{-2\pi iz})(1-q^n), \quad q=e^{2\pi i\tau}.
\end{equation}

The function $\Theta_f$ is a weak Jacobi form (with character) of weight $f(0)/2$ in the sense of Eichler--Zagier \cite{EZ85}. Gritsenko--Poor--Yuen conjectured that if $\Theta_f$ is a holomorphic Jacobi form of integral weight and has $q$-order one, i.e. $$f(0)+2\sum_{a=1}^\infty f(a)=24,$$ then its Gritsenko lift \cite{Gri94} is a Borcherds product. This has been proved in \cite{GPY15, Gri18, GW20, DW20} for the eight infinite families of theta blocks of $q$-order one (built in \cite{GSZ19}) related to root systems, and in particular for all theta blocks of weight at least $4$. However, there are many theta blocks of $q$-order one in weights two and three which do not belong to these families. Surprisingly, the proof of the theta block conjecture in general becomes much easier when we extend it to the (singular) Gritsenko lifts of theta blocks which are merely weak Jacobi forms. We will prove that all such Gritsenko lifts are (meromorphic) Borcherds products. This follows from identities for singular theta lifts associated to the root lattices $D_m$:

\begin{theorem}\label{MTH3}
As a meromorphic modular form of weight $12-m$ on $\widetilde{\Orth}^+(2U\oplus D_m)$, the Borcherds additive lift of the generalized theta block
\begin{equation}
\vartheta_{D_m}(\tau, \mathfrak{z}) =\eta(\tau)^{24-3m} \prod_{j=1}^m \vartheta(\tau, z_j), \quad 1\leq m \leq 11
\end{equation}
is the Borcherds multiplicative lift of the weight zero weak Jacobi form  $-(\vartheta_{D_m}| T_{-}(2) ) / \vartheta_{D_m}$, where $T_{-}(2)$ is the index-raising Hecke operator on Jacobi forms. As a result, the singular additive lift of any generalized theta block of positive weight and $q$-order one is a Borcherds product. In particular, the Gritsenko--Yuen--Poor conjecture is true.
\end{theorem}

The inputs into Borcherds' singular additive lifts \cite{Bor95, Bor98} are Jacobi forms of positive weight. When the inputs are holomorphic Jacobi forms, the singular additive lift coincides with the Gritsenko lift. When $m\leq 8$, the form $\vartheta_{D_m}$ above is a holomorphic Jacobi form of index $D_m$ so its Gritsenko lift is a holomorphic modular form, and in these cases Theorem \ref{MTH3} was proved in \cite{GPY15}. When $m>8$, $\vartheta_{D_m}$ is no longer a holomorphic Jacobi form and its additive lift will have singularities along $\cH_{D_m,1}$ defined in Theorem \ref{MTH2} (note $\cH_{D_m,0}=\emptyset$).

The Borcherds product in Theorem \ref{MTH3} has only simple zeros outside of the hyperplane arrangement $\cH_{D_m,1}$, and these simple zeros are determined by the zeros of the theta block $\vartheta_{D_m}$. This allows us to reduce the proof of the theorem to the Koecher principle on the Looijenga compactification.

\subsection{Non-free algebras of holomorphic modular forms}
Theorem \ref{MTH2} also includes a class of generators of ``Jacobi type",  whose leading Fourier--Jacobi coefficients are $\Delta^{t}\phi_{k,t}$ where $\Delta=\eta^{24}$ is the cusp form of weight $12$ on $\SL_2(\ZZ)$, and $\phi_{k,t}$ are the generators of weight $k$ and index $t$ of the ring of Weyl invariant weak Jacobi forms. This motivates us to study the rings of holomorphic modular forms for $\widetilde{\Orth}^+(2U\oplus L)$ satisfying the following property: 

\vspace{3mm}

\textit{Let $J_{k,L,t}^{\w,\widetilde{\Orth}(L)}$ denote the space of weak Jacobi forms of weight $k$ and index $t$ associated to $L$ which are invariant under $\widetilde{\Orth}(L)$. For any $\phi_{k,t}\in J_{k,L,t}^{\w,\widetilde{\Orth}(L)}$, there exists a holomorphic modular form of weight $12t+k$ for $\widetilde{\Orth}^+(2U\oplus L)$ whose leading Fourier--Jacobi coefficient is  $\Delta^{t}\phi_{k,t}$.}

\vspace{3mm}

Following \cite{Aok00} we can estimate the dimensions of spaces of orthogonal modular forms in terms of Jacobi forms:
$$
\dim M_k(\widetilde{\Orth}^+(2U\oplus L)) \leq \sum_{t=0}^\infty \dim J_{k-12t,L,t}^{\w,\widetilde{\Orth}(L)}. 
$$
The above property means that this inequality is actually an equality in every weight. This property also implies the modularity of formal Fourier--Jacobi series, which has applications to the modularity of certain generating series which appear in arithmetic geometry (see \cite{BR15}). We classify all root lattices satisfying this property and we also determine the corresponding rings of modular forms in the last main theorem:

\begin{theorem}\label{MTH4}
Let $L$ be a direct sum of irreducible root lattices not of type $E_8$. Assume that $L$ satisfies the $\mathrm{Norm}_2$ condition, i.e. $\delta_L \leq 2$, where
$$
\delta_L := \max\{ \min\{\latt{y,y}: y\in L + x \} : x \in L' \}.
$$
Then the algebra of holomorphic modular forms for $\widetilde{\Orth}^+(2U\oplus L)$ is minimally generated by forms of weights $4$, $6$ and $12t+k$, where the pairs $(k, t)$ are the weights and indices of the generators of the ring of $\widetilde{\Orth}(L)$-invariant weak Jacobi forms. 
\end{theorem}

There are $40$ root lattices that satisfy the conditions of Theorem \ref{MTH4}. Fourteen of them are irreducible and they correspond to the free algebras of holomorphic modular forms which were included in Theorem \ref{MTH2}. The $26$ reducible root lattices in Theorem \ref{MTH4} yield non-free algebras of modular forms. Some of these non-free algebras have a lot of generators, and a complicated structure which seems hard to determine by other methods. For example, the algebra of holomorphic modular forms for $\widetilde{\Orth}^+(2U\oplus A_2\oplus E_6)$ has Krull dimension $11$ but is minimally generated by $33$ generators of weights 4, 4, 5, 6, 7, 7, 8, 9, 10, 10, 10, 11, 12, 12, 12, 13, 13, 14, 14, 15, 15, 16, 16, 16, 17, 18, 18, 18, 19, 20, 21, 22, 24. The construction of generators of Jacobi type for lattices of $AE$-type in Theorem \ref{MTH2} is completed by Theorem \ref{MTH4}. 

\subsection{The structure of this paper} In \S \ref{sec:Looijenga} we review the theory of meromorphic modular forms on type IV symmetric domains related to the Looijenga compactification, and we establish the modular Jacobian approach (i.e. Theorem \ref{MTH1}). In \S \ref{sec:Jacobiforms} we review Jacobi forms of lattice index and the Gritsenko and Borcherds lifts. \S \ref{sec:theta block conjecture} contains the proof of Theorem \ref{MTH3} and therefore the theta block conjecture. The heart of the paper is \S \ref{sec:type IV} where we prove Theorem \ref{MTH2}, with the exception of some lattices related to $E_6$ and $E_7$. In \S \ref{sec:non-free} we prove Theorem \ref{MTH4} and complete the proof of Theorem \ref{MTH2}. In \S \ref{sec:ball quotients} we review Looijenga's compactifications for unitary groups and we use Theorem \ref{MTH2} to construct eight free algebras of meromorphic unitary modular forms, including one such algebra on a six-dimensional complex ball.  In the appendix (cf. \S \ref{appendix}), we formulate the weights of generators of the algebras of modular forms described in Theorems \ref{MTH2} and \ref{MTH4}.

\addtocontents{toc}{\setcounter{tocdepth}{2}}
%to hide subsections of introduction in the table contents.

\section{Meromorphic modular forms on orthogonal groups}\label{sec:Looijenga}
In this section we review the Looijenga compactification on symmetric domains of type IV and establish the modular Jacobian approach for meromorphic modular forms whose poles are contained in an arrangement of hyperplanes.

\subsection{The Looijenga compactification}\label{subsec:Looijenga}
Let $M$ be an even lattice of signature $(l,2)$ for some $l\geq 3$ with bilinear form $(-,-)$. We equip the complex vector space $V(M):=M\otimes \CC$ with the $\QQ$-structure defined by $M$ and the non-degenerate symmetric bilinear form induced by $(-,-)$. The corresponding orthogonal group $\Orth(V(M))\cong\Orth(l,2)$ is therefore an algebraic group defined over $\QQ$. The set of complex lines $\mathcal{Z} \in \PP(V(M))$ satisfying $( \mathcal{Z}, \mathcal{Z} ) = 0$ and $( \mathcal{Z}, \overline{\mathcal{Z}} ) < 0$ has two conjugate connected components. We fix one component once and for all; it is a symmetric domain of Cartan type IV, denoted $\cD(M)$. 
Let $\Orth^+(M)$ be the subgroup of $\Orth(l,2)$ that preserves $M$ and $\mathcal{D}(M)$. The subgroup of $\Orth^+(M)$ which acts trivially on the discriminant group $M'/M$ is the \emph{discriminant kernel}, $$\widetilde{\Orth}^+(M) = \{\gamma \in \Orth^+(M): \; \gamma x - x \in M \; \text{for all} \; x \in M'\}.$$

The symmetric domain $\mathcal{D}(M)$ also determines a set of oriented real planes in $V(M)$ after identifying $\mathcal{Z} = X+iY$ with $\frac{i}{2} \mathcal{Z} \wedge \mathcal{\overline{Z}} = X \wedge Y$. For $\mu \in M$ let $[\mu]$ denote the space $\RR\mu$. For any rational isotropic line $I \in \PP(V_{\mathbb{Q}}(M))$, define the negative cone $$C_I = \{I \wedge [\mu]: \; \mu \in I^{\perp}, \; (\mu, \mu) < 0\} \cap \overline{\mathcal{D}(M)} \subseteq \mathbb{P}(\wedge^2 V_{\mathbb{R}}(M)),$$ and for any rational isotropic plane $J \subseteq \PP(V_{\mathbb{Q}}(M))$, define the isotropic half-line $$C_J = [\wedge^2 J] \cap \overline{\mathcal{D}(M)} \subseteq \mathbb{P}(\wedge^2 V_{\mathbb{R}}(M)).$$ Note that the sets $C_I$, $C_J$ are pairwise disjoint and that $I \subseteq J$ if and only if $C_J \subseteq \overline{C_I}$. We define $$C_0 = \{0\}, \quad C_{J, +} = C_J \cup \{0\}, \quad C_{I, +} = C_I \cup \bigcup_{I \subseteq J} C_J \cup \{0\}.$$ Then $C_0$, $C_{J, +}$ and $C_{I, +}$ are the $0$-, $1$- and $2$-dimensional faces of a rational polyhedral fan $\mathcal{C}(M) \subseteq \wedge^2 V_{\mathbb{R}}(M)$ called the \emph{conical locus} $\mathcal{C}(M)$ of $M$ (see \cite[Definition 2.1]{Loo03b}).

For any vector $v$ of positive norm in $M'$, the associated hyperplane (or rational quadratic divisor)
$$
\cD_v(M) := v^\perp \cap \cD(M) = \{ [\mathcal{Z}] \in \cD(M): (\mathcal{Z}, v) = 0 \} 
$$
is the type IV symmetric domain attached to the signature $(l-1,2)$ lattice $M_v:=v^\perp \cap M$. For any finite-index subgroup $\Gamma$ of $\Orth^+(M)$, a ($\Gamma$-invariant) hyperplane arrangement $\cH$ is defined to be a finite union of $\Gamma$-orbits of hyperplanes $\cD_v(M)$. An arrangement $\cH$ is said to satisfy \textit{the Looijenga condition} (cf. Corollary 7.5 of \cite{Loo03b}) if every one-dimensional intersection of hyperplanes in $\cH$ (taken in the ambient space $V(M)$) is positive definite.

Let $\cD^\circ$ denote the arrangement complement $\cD(M) - \cH$.  Looijenga \cite{Loo03b} defined a compactification $\widehat{X^{\circ}}$ of $X^\circ:=\cD^\circ / \Gamma$ by successively blowing up nonempty intersections of members of $\cH$ modulo $\Gamma$ and then blowing down the final blow-up. On the level of sets, $\widehat{X^\circ}:=\widehat{\cD^\circ} / \Gamma$, where $\widehat{\cD^\circ}$ is the disjoint union 
\begin{align*}
\widehat{\cD^\circ} = \cD^\circ \sqcup \bigsqcup_{L\in \mathrm{PO}(\cH)} \pi_L \cD^\circ \sqcup \bigsqcup_{\sigma \in \sum(\cH)} \pi_{V_\sigma} \cD^\circ. 
\end{align*}
The notation above is as follows. The first union is taken over the poset $\mathrm{PO}(\cH)$ of nonempty intersections inside of hyperplanes in $\cH$ within $\cD(M)$. In the second union, note that any hyperplane $H \in \cH$ containing $I$ defines a hyperplane $\{I \wedge [\mu]: \; \mu \in H\}$ in the facet $C_{I, +}$; that as $H$ runs through $\cH$ this family is locally finite; and therefore $\cH$ determines a decomposition of the conical locus $\mathcal{C}(M)$ into locally rational cones. The collection of these cones is labelled $\sum(\cH)$. For any $\sigma \in \sum(\cH)$ contained in the facet $C_{I, +}$ but not itself a face of the conical locus, its \emph{support space} is the span $V_{\sigma} \subseteq I^{\perp}$ of vectors $\mu$ for which $I \wedge [\mu] \subseteq \sigma$; and for any one-dimensional face $\sigma = C_{J, +} \in \sum(\cH)$, $$V_{\sigma} = J^{\perp} \cap \bigcap_{\substack{H \in \cH \\ J \subseteq H}} H.$$ Finally, for any vector subspace $W \subseteq V(M)$, $\pi_W$ is the canonical projection $$\pi_W : \Big( \PP(V(M)) - \PP(W) \Big) \longrightarrow \PP(V(M) / W).$$ This union carries a $\Gamma$-equivariant topology defined in \cite{Loo03b}. When $\mathcal{H}$ is empty, $\widehat{\cD^{\circ}} = \widehat{\cD}$ is the union of $\cD(M)$ with the entire conical locus and it carries the usual Satake--Baily--Borel topology, and therefore $\widehat{X^{\circ}}$ is exactly the Baily--Borel compactification.

The set $\cD^\circ$ is open and dense in $\widehat{\cD^\circ}$. When we pass to the $\Gamma$-orbit space, the boundary of the compactification $\widehat{X^\circ}$ of $X^\circ$ is decomposed into finitely many components. The Looijenga condition on $\cH$ guarantees that these boundary components have codimension at least two.

It is convenient to rephrase Looijenga's condition on $\cH$ in terms of intersections of rational quadratic divisors (within $\cD(M)$). Note that an intersection $\cD_{v_1}(M)\cap \cdots \cap \cD_{v_t}(M)$ is nonempty and of dimension $m\geq 0$ if and only if $\mathbb{Z}v_1 + ... + \mathbb{Z}v_t$ is a positive-definite lattice of rank $l-m$. It is not hard to show that the Looijenga condition is equivalent to the following condition.
\begin{condition}\label{condition}
For every nonempty intersection $\mathcal{I} = \cD(M) \cap \cD_{v_1}(M)\cap \cdots \cap \cD_{v_t}(M)$ of hyperplanes in $\cH$, the following hold:
\begin{enumerate}
    \item $\mathcal{I}$ is not discrete.
    \item If $\mathcal{I}$ is of dimension $2$, then the lattice $M\cap v_1^\perp \cap \cdots \cap v_t^\perp$ does not contain an isotropic plane. 
    \item If $\mathcal{I}$ is of dimension $1$, then the lattice $M\cap v_1^\perp \cap \cdots \cap v_t^\perp$ does not contain an isotropic vector. 
\end{enumerate}
\end{condition}

\subsection{Orthogonal modular forms}

The Baily--Borel compactification can be realized as the $\mathrm{Proj}$ of the algebra of holomorphic modular forms, and the Looijenga compactification has a similar characterization. In the latter case, the corresponding modular forms are defined as meromorphic sections of $\widehat{\mathcal{L}}$ which are regular on $X^\circ$, where $\widehat{\mathcal{L}}$ is the ample line bundle on $\widehat{X^\circ}$ which has the same restriction to $X^\circ$ as the automorphic line bundle on $X$, and the Looijenga compactification is again the Proj of the graded algebra of modular forms. We will explain this more explicitly below.

Let us define the affine cone over $\cD(M)$ as
$$\cA(M)=\{ \mathcal{Z} \in V(M): [\mathcal{Z}] \in \cD(M) \}.$$

\begin{definition}
Let $k$ be an integer. A \emph{modular form} of weight $k$ and character $\chi\in \mathrm{Char}(\Gamma)$ with poles on $\cH$ is a meromorphic function $F: \mathcal{A}(M)\to \CC$ which is holomorphic away from $\cH$ and satisfies
\begin{align*}
F(t\mathcal{Z})&=t^{-k}F(\mathcal{Z}), \quad \forall t \in \CC^\times,\\
F(g\mathcal{Z})&=\chi(g)F(\mathcal{Z}), \quad \forall g\in \Gamma.
\end{align*}
\end{definition}
All modular forms with poles on $\cH$ of integral weight and trivial character form a graded algebra over $\CC$ and we denote it by
$$
M_*^!(\Gamma)=\bigoplus_{k = -\infty}^{\infty} M_k^!(\Gamma).
$$
When $\cH$ is empty, this algebra reduces to the usual algebra of holomorphic modular forms
$$
M_*(\Gamma)=\bigoplus_{k = 0}^\infty M_k(\Gamma).
$$

It is possible that $\cH$ is defined by the zero of a holomorphic modular form (see \cite[\S 3.2 and \S 5.2]{Loo03b} for criteria). However, this will never happen when $\cH$ satisfies the Looijenga condition. In fact, the Looijenga condition on $\cH$ implies that every modular form with poles on $\cH$ defines a regular section of a non-negative tensor power of $\widehat{\mathcal{L}}$. Equivalently, a modular form of non-positive weight with poles on $\cH$ must be constant. Furthermore, Looijenga proved in \cite[Corollary 7.5]{Loo03b} that the algebra $M_*^!(\Gamma)$ is finitely generated by forms of positive weights and its $\mathrm{Proj}$ gives the Looijenga compactification $\widehat{X^\circ}$. 

As mentioned above, the Looijenga condition implies a form of Koecher's principle for modular forms with poles on $\cH$. Due to its importance, we give a new proof using the pullback trick.  

\begin{lemma}\label{lem:Koecher}
Assume that $\cH$ satisfies Condition \ref{condition}. Let $F$ be a nonzero modular form of weight $k$ with poles on $\cH$. Let $c_F$ denote the maximal multiplicity of poles of $F$. Then $k\geq c_F$. In particular, if $k=0$ then $F$ is constant. 
\end{lemma}
\begin{proof}
Assume that $F$ has poles of multiplicity $c_F$ on the hyperplane  $\cD_v(M)$ in $\cH$.  The restrictions of members of $\cH$ to $\cD_v(M)$ define a hyperplane arrangement $\cH_v$ which also satisfies Condition \ref{condition}. The quasi pullback of $F$ to $\cD_v(M)$ is a modular form of weight $k-c_F$ with poles supported on $\cH_v$. This allows us to prove the lemma by induction of the dimension of the symmetric domain $\cD(M)$. When $k=0$, $c_F$ must be zero, which yields that $F$ is  holomorphic and then must be constant by the usual Koecher's principle. 
\end{proof}

In this paper the hyperplane arrangements will always be finite unions of Heegner divisors. Let $\gamma \in M'/M$ and $a$ be a positive rational number. The (primitive) Heegner divisor of discriminant $(a,\gamma)$ is
$$
H(a,\gamma)= \bigcup_{\substack{v \in M + \gamma \\ v \; \text{primitive in} \; M' \\ (v,v) = 2a}} \mathcal{D}_v(M),
$$
and it is locally finite and $\widetilde{\Orth}^+(M)$-invariant and therefore descends to an (irreducible) analytic divisor on the modular variety $\cD(M) / \widetilde{\Orth}^+(M)$. Note that $\mathcal{D}_v(M) = \mathcal{D}_{-v}(M)$ implies $H(a, \gamma) = H(a, -\gamma)$. 

The lemma below describes the intersections in $\cD(M)$ of members of $\widehat{H}(a, \gamma)$ and therefore is useful to verify the Looijenga condition \ref{condition} for arrangements of hyperplanes, where $\widehat{H}(a, \gamma)$ is the non-primitive Heegner divisor defined as
$$
\widehat{H}(a,\gamma)= \bigcup_{\substack{v \in M + \gamma \\ (v,v) = 2a}} \mathcal{D}_v(M),
$$
which is a finite union of some primitive Heegner divisors defined above.

\begin{lemma}\label{lem:intersection}
Let $u, v \in M+\gamma$ such that $\cD_u(M), \cD_v(M) \in \widehat{H}(a, \gamma)$. 
\begin{enumerate}
    \item If $a\leq \frac{1}{4}$, then $\cD_u(M)\cap\cD_v(M)=\emptyset $. 
    \item Suppose that $\cD_u(M)\cap\cD_v(M) \neq \emptyset $. If $\frac{1}{4}< a \leq \frac{1}{2}$, then $\widehat{H}(a, \gamma) \cap \cD_v(M)$, viewed as a non-primitive Heegner divisor in $\cD(M_v) = \cD_v(M)$, is 
    \begin{equation*}
    \widehat{H}(1 - 1/(4a), u') = \bigcup_{\substack{x \in M_v + u' \\ (x,x)/2 = 1-1/(4a) }} \cD_x(M_v).
    \end{equation*}
\end{enumerate}
Here $M_v$ is the orthogonal complement of  $v$ in $M$, and $u'$ is the orthogonal projection of $u$ to the dual lattice $M_v'$. 
\end{lemma}
\begin{proof}
(1) $\cD_u(M)\cap\cD_v(M)$ is nonempty if and only if the sublattice generated by $u$ and $v$ is positive definite. Since $(u, u) = (v, v) = 2a$, this holds if and only if $|(u, v)| < 2a$. On the other hand, $u - v \in M$ implies
$$
(u,v) = (v,v) + (u-v, v) \in 2a + \ZZ,
$$
which forces $|(u,v)| \geq 2a$ if $a \le 1/4$. 

(2) Let $w\in M+\gamma$ with $(w,w)=2a$. As in (1), the intersection $\cD(M_v)\cap w^\perp $ is nonempty if and only if $|(v, w)| < 2a$, which occurs if and only if $(v, w)=2a-1$ under the assumption $1/4 < a \le 1/2$. When $(v, w)=2a-1$, the projection of $w$ to $M_v'$ is
$$
w'= w -\frac{(w,v)}{(v,v)}v = w + \left(\frac{1}{2a}-1\right)v,
$$
and it satisfies $(w',w')=2(1-\frac{1}{4a})$ and $w'-u'\in M_v$.  This proves our claim. 
\end{proof}

\begin{corollary}\label{cor:intersection} Define $a_0 = 0$ and $a_k = 1 / (4 - 4a_{k-1})$ for $k \ge 1$. Suppose that $\cH$ is a finite collection of distinct non-primitive Heegner divisors $\widehat{H}(a, \gamma)$ with $a<1/2$, of which $b_k$ have norm $a_{k-1} < a \le a_k$. If $$\sum_k k \cdot b_k < l - 2,$$ then $\cH$ satisfies the Looijenga condition \ref{condition}.
\end{corollary}
Recall that $M$ has signature $(l, 2)$. Only the values $a_1 = 1/4$, $a_2 = 1/3$, $a_3 = 3/8$, $a_4 = 2/5$ and $a_5 = 5/12$ will be needed in the paper.

\begin{proof} 
Let $\widehat{H}(a,\gamma)$ be a Heegner divisor contained in $\cH$ satisfying $a_{k-1}<a\leq a_k$. By Lemma \ref{lem:intersection}, the intersection between $\widehat{H}(a,\gamma)$ and one of its hyperplanes is a Heegner divisor $\widehat{H}(a',-)$ satisfying $a_{k-2}<a'\leq a_{k-1}$. By repeating this process, we find that all nonempty intersections of hyperplanes in $\widehat{H}(a,\gamma)$ have codimension at most $k$. It follows that the nonempty intersections of hyperplanes in $\cH$ have codimension at most $\sum_k kb_k$, and therefore dimension at least $3$. Thus $\cH$ satisfies the Looijenga condition \ref{condition}. 
\end{proof}

\subsection{Reflections and the Jacobian} 
In \cite{Wan21a} the first named author introduced the modular Jacobian criterion to decide when the ring $M_*(\Gamma)$ of holomorphic modular forms is a polynomial algebra. This states that a set of modular forms freely generates $M_*(\Gamma)$ if and only if their Jacobian vanishes exactly with multiplicity one on hyperplanes fixed by reflections in $\Gamma$. In this section we extend this approach to the ring $M_*^!(\Gamma)$ of modular forms with poles on a hyperplane arrangement.

Let $M$ be an even lattice of signature $(l, 2)$ with $l \ge 3$ and let $\Gamma \le \Orth^+(M)$ be a finite-index subgroup.
Let $r$ be a non-isotropic vector of $M$. The \textit{reflection} associated to $r$ is defined as
$$
\sigma_r:  v \mapsto \frac{2(v,r)}{(r,r)}r \in \Orth(M\otimes \QQ).
$$
A reflection $\sigma_r$ preserves $\cD(M)$ if and only if its spinor norm is $+1$, i.e. $(r,r)>0$. Therefore, we need only consider reflections associated to vectors of positive norm. The set of fixed points of $\sigma_r$ on $\cD(M)$ is the hyperplane $\cD_r(M)$, and is called the \emph{mirror} of $\sigma_r$. A vector $r\in M'$ of positive norm is called \textit{reflective} if its reflection $\sigma_r$ preserves the lattice $M$. By definition, a primitive vector $r\in M$ of norm $(r,r)=2d$ is reflective if and only if $\Div(r)=d$ or $2d$, where $\Div(r)$ denotes the positive generator of the ideal $\{ (r, v): v\in M \}$. It follows that a primitive vector $x \in M'$ is reflective if and only if there exists a positive integer $d$ such that $(x,x)=2/d$ and $\ord(x)=d$ or $d/2$, where $\ord(x)$ is the order of $x$ in $M'/M$.  Every vector of norm $2$ in $M$ is reflective, and the corresponding reflection acts trivially on $M'/M$; conversely, the discriminant kernel only contains this type of reflection. 

The Jacobian of modular forms is a generalization of the Rankin--Cohen--Ibukiyama differential operator for Siegel modular forms introduced in \cite{AI05}. For $0\leq j \leq l$, let $F_j$ be a modular form of weight $k_j$ and character $\chi_j$ for $\Gamma$ with poles on an arrangement $\cH$ of hyperplanes. View these $F_j$ as modular forms on the tube domain $c^\perp / c$, where $c$ is a primitive isotropic vector of $M$, and let $z_1, ..., z_l$ be coordinates on the tube domain. The Jacobian of these $l+1$ meromorphic modular forms in these coordinates is
$$
J(F_0, ..., F_{l+1}) = \mathrm{det} \begin{pmatrix} k_0 F_0 & k_1 F_1 & \cdots & k_l F_l \\ \partial_{z_1} F_0 & \partial_{z_1} F_1 & \cdots & \partial_{z_1} F_l \\  \vdots & \vdots & \ddots & \vdots \\ \partial_{z_l} F_0 & \partial_{z_l} F_1 & \cdots & \partial_{z_l} F_l  \end{pmatrix},
$$
and it defines a meromorphic modular form of weight $k_J$ and character $\chi_J$ for $\Gamma$, holomorphic away from $\cH$, where
$$
k_J=l + \sum_{j=0}^l k_j, \quad \chi_J = \mathrm{det} \otimes  \bigotimes_{j=0}^l \chi_j.
$$
Moreover, the Jacobian is not identically zero if and only if these $l+1$ forms are algebraically independent over $\CC$. 

The Jacobian satisfies the product rule in each component and every modular form with poles on $\cH$ can be written as a quotient $f = g/h$, where $g, h$ are holomorphic. Using $$J(g/h, f_1,...,f_l) = \frac{1}{h^2} \Big( h \cdot J(g, f_1,...,f_l) - g \cdot J(h, f_1,...,f_l) \Big)$$ and the analogous equations in the other components, we can reduce the claims above to the Jacobians of holomorphic modular forms, where they are already known (cf. \cite[Theorem 2.5]{Wan21a}).

Assume that the characters of all $F_j$ are trivial, so their Jacobian has the determinant character. Let $\cD_r(M)$ be the mirror of a reflection $\sigma_r\in \Gamma$ which is not contained in $\cH$. From $\det(\sigma_r)=-1$, we see that the Jacobian must vanish on $\cD_r(M)$.

We will prove some necessary conditions for $M_*^!(\Gamma)$ to be a free algebra.  The following theorem is an analogue of \cite[Theorem 3.5]{Wan21a} and the proof is essentially the same. 

\begin{theorem}\label{th:freeJacobian}
Let $\cH$ be a hyperplane arrangement satisfying the Looijenga condition. Suppose that the ring of modular forms for $\Gamma$ with poles on $\cH$ is freely generated by $F_j$ for $0\leq j \leq l$. 
\begin{enumerate}
    \item The group $\Gamma$ is generated by reflections. 
    \item The Jacobian $J:=J(F_0,...,F_l)$ is a nonzero meromorphic modular form of character $\det$ for $\Gamma$ which satisfies the conditions:
    \begin{enumerate}
        \item $J$ vanishes with multiplicity one on all mirrors of reflections in $\Gamma$ which are not contained in $\cH$;
        \item all other zeros and poles of $J$ are contained in $\cH$. \end{enumerate}
    \item Let $\{ \Gamma \pi_1, ..., \Gamma \pi_s\}$ be the distinct $\Gamma$-equivalence classes of mirrors of reflections in $\Gamma$ which are not contained in $\cH$.  Let $P$ be the polynomial expression of $J^2$ in terms of the generators $F_j$. Then $P$ decomposes into a product of $s$ irreducible polynomials, and each irreducible factor is a modular form for $\Gamma$ which vanishes with multiplicity two on some $\Gamma\pi_i$ and whose other zeros and poles are contained in $\cH$.
\end{enumerate}
\end{theorem}

\begin{proof}
(1) We use the notations in \S \ref{subsec:Looijenga}. Let $\cD^\circ=\cD - \cH$ and $X^\circ = \cD^\circ / \Gamma$. Recall that the Looijenga compactification $\widehat{X^\circ}$ is defined as $\mathrm{Proj}(M_*^!(\Gamma))$.  Let $Y^\circ=\cA^\circ / \Gamma$ be the affine cone over $X^\circ$ and $\widehat{Y^\circ}$ be the affine span of $\widehat{X^\circ}$ defined by the maximal spectrum of $M_*^!(\Gamma)$.  The freeness of $M_*^!(\Gamma)$ implies that $\widehat{X^\circ}$ is a weighted projective space and that $\widehat{Y^\circ}$ is simply the affine space $\CC^{l+1}$. We known from \cite{Loo03b} that the boundary of $\widehat{X^\circ} - X^\circ$ has only finitely many components and each component has codimension at least two in $\widehat{X^\circ}$.  Therefore, a similar property holds for $\widehat{Y^\circ}-Y^\circ$.  It follows that $Y^\circ$ is smooth and simply connected.  Applying a theorem of Armstrong \cite{Arm68} (cf. \cite[Theorem 3.2]{Wan21a}) to $Y^\circ$, we find that $\Gamma$ is generated by elements having fixed points.  Then a theorem of Gottschling \cite{Got69a} (cf. \cite[Theorem 3.1]{Wan21a}) implies that $\Gamma$ is generated by reflections.

(2) Let $k_i$ denote the weight of $F_i$ for $0\leq i \leq l$.  Then $$[z] \mapsto [F_0(z),...,F_{l}(z)]$$ identifies $\widehat{X^\circ}$ with $\PP(k_0,...,k_l)$ and further induces a holomorphic map
$$
   \phi:  \cA^\circ \to Y^\circ = \cA^\circ / \Gamma \hookrightarrow \CC^{l+1}\setminus \{0\}, \quad z \mapsto (F_0(z),...,F_{l}(z)).
$$
If $x\in \cA^\circ$ has trivial stabilizer $\Gamma_x$, then $\phi$ is biholomorphic on a neighborhood of $x$ since the action of $\Gamma$ on $\cA^\circ$ is properly discontinuous, so the Jacobian $J$ (which can be regarded as the usual Jacobian determinant of $\phi$) is nonvanishing there. Then Gottschling's theorem implies that $J$ is nonzero on $\cD^\circ$ away from mirrors of reflections in $\Gamma$.   From the character of $J$, we conclude that the restriction of $J$ to the complement $\cD^\circ$ vanishes precisely on mirrors of reflections in $\Gamma$. The fact that the multiplicity of the zeros is one follows from the fact that the reflections have order two. 

(3) The proof of this is similar to that of \cite[Theorem 3.5 (3)]{Wan21a}. We only need to compare the decomposition of the zero divisor of $J$ on $\cD^\circ$ and the decomposition of the zero locus of $P$ in $\PP(k_0,...,k_l)$. 
\end{proof}

At the end of this section, we establish the modular Jacobian criterion,  which will be used to construct free algebras of meromorphic modular forms later.
\begin{theorem}\label{th:Jacobiancriterion}
Let $\cH$ be an arrangement of hyperplanes satisfying the Looijenga condition. Suppose that there exist $l+1$ algebraically independent modular forms with poles on $\cH$ such that the restriction of their Jacobian to the arrangement complement $\cD^\circ$ vanishes precisely on mirrors of reflections in $\Gamma$ with multiplicity one. Then the algebra of modular forms with poles on $\cH$ is freely generated by these $l+1$ forms. 
\end{theorem}
\begin{proof}
The proof is similar to that of \cite[Theorem 5.1]{Wan21a}. For the reader's convenience we give a short proof. Let $f_i$ ($1\leq i\leq l+1$) be $l+1$ modular forms of weight $k_i$ with poles on $\cH$ whose Jacobian $J$ satisfies the conditions of the theorem.
Suppose that $M_*^!(\Gamma)$ is not generated by $f_i$. Let $f_{l+2} \in M_{k_{l+2}}^!(\Gamma)$ be a modular form of minimal weight which does not lie in $\CC[f_1,...,f_{l+1}]$. For $1\leq t \leq l+2$ we define $$J_{t} = J(f_1,..., \hat f_t, ..., f_{l+2})$$ as the Jacobian of the $l+1$ modular forms $f_i$ omitting $f_t$, such that $J=J_{l+2}$. It is clear that $g_t := J_t/J \in M_*^!(\Gamma)$. 
The identity
$$
0 = \mathrm{det} \begin{psmallmatrix} k_1 f_1 & k_2 f_2 & \cdots & k_{l+2} f_{l+2} \\ k_1 f_1 & k_2 f_2 & \cdots & k_{l+2} f_{l+2} \\ \nabla f_1 & \nabla f_2 & ... & \nabla f_{l+2} \end{psmallmatrix} =  \sum_{t=1}^{l+2} (-1)^t k_t f_t J_t = J\cdot \Big( \sum_{t=1}^{l+2} (-1)^t k_t f_t g_t \Big) 
$$
and the equality $g_{l+2}=1$ yield
$$
(-1)^{l+1}k_{l+2}f_{l+2}= \sum_{t=1}^{l+1}(-1)^t k_t f_t g_t.
$$
Each $g_t$ has weight strictly less than that of $f_{l+2}$. The construction of $f_{l+2}$ implies that $g_t \in \CC[f_1,...,f_{l+1}]$ and then $f_{l+2} \in \CC[f_1,...,f_{l+1}]$, which contradicts our assumption.
\end{proof}

\section{Jacobi forms and singular theta lifts}\label{sec:Jacobiforms}
In this section we review some results from the theories of Jacobi forms of lattice index and Borcherds' singular theta lifts, which are necessary to prove the main theorems in the introduction.

\subsection{Jacobi forms of lattice index}\label{sec:Jacobi} In 1985 Eichler and Zagier introduced the theory of Jacobi forms in their monograph \cite{EZ85}. These are holomorphic functions in two variables $(\tau,z)\in \HH \times \CC$ which are modular in $\tau$ and quasi-periodic in $z$. Jacobi forms of lattice index were defined in \cite{Gri88} by replacing $z$ with a vector of variables associated with a positive definite lattice. As a bridge between different types of modular forms,  Jacobi forms have many applications in mathematics and physics. Let $L$ be an even integral positive definite lattice of rank $\rk(L)$ with bilinear form $\latt{-,-}$ and dual lattice $L'$.

\begin{definition}\label{def:JFs}
Let $k\in \ZZ$ and $t\in \NN$. A \emph{nearly holomorphic Jacobi form} of weight $k$ and index $t$ associated to $L$ is a holomorphic function $\varphi : \HH \times (L \otimes \CC) \rightarrow \CC$ which satisfies 
\begin{align*}
\varphi \left( \frac{a\tau +b}{c\tau + d},\frac{\mathfrak{z}}{c\tau + d} \right) &= (c\tau + d)^k \exp\left( t\pi i \frac{c\latt{\mathfrak{z},\mathfrak{z}}}{c \tau + d}\right) \varphi ( \tau, \mathfrak{z} ), \quad \left( \begin{array}{cc}
a & b \\ 
c & d
\end{array} \right)   \in \SL_2(\ZZ),\\
\varphi (\tau, \mathfrak{z}+ x \tau + y)&= \exp\left(-t\pi i ( \latt{x,x}\tau +2\latt{x,\mathfrak{z}} )\right) \varphi ( \tau, \mathfrak{z} ), \quad x,y\in L,
\end{align*}
and has a Fourier expansion of the form 
\begin{equation*}
\varphi ( \tau, \mathfrak{z} )= \sum_{ n \gg -\infty}\sum_{ \ell \in L'}f(n,\ell)q^n \zeta^\ell,  \quad q=e^{2\pi i\tau}, \; \zeta^\ell = e^{2\pi i \latt{\ell, \mathfrak{z}}}.
\end{equation*}
If $f(n,\ell) = 0$ whenever $n <0$ (resp. $2nt - \latt{\ell,\ell} <0$), then $\varphi$ is called a \textit{weak} (resp. \textit{holomorphic}) Jacobi form.   We denote the vector spaces of nearly holomorphic, weak and holomorphic  Jacobi forms of weight $k$ and index $t$ respectively by
$$
J_{k,L,t}^{!} \supset J_{k,L,t}^{\w} \supset J_{k,L,t}.
$$
\end{definition}
Jacobi forms of index $0$ are independent of the lattice variable $\mathfrak{z}$ and are therefore classical modular forms on $\SL_2(\ZZ)$.  In general, the quasi-periodicity and the transformation under $\begin{psmallmatrix} -1 & 0 \\ 0 & -1 \end{psmallmatrix} \in \SL_2(\ZZ)$ imply the following constraints on the coefficients.

\begin{lemma}\label{lem:periodic}
\noindent
\begin{enumerate}
\item Let $t\geq 1$. The Fourier coefficients of $\varphi \in J_{k,L,t}^!$ satisfy 
$$
f(n_1,\ell_1)=f(n_2,\ell_2) \quad \text{if} \quad 2n_1t - \latt{\ell_1,\ell_1} = 2n_2t - \latt{\ell_2,\ell_2} \; \text{and} \; \ell_1 - \ell_2 \in tL.
$$
\item The Fourier coefficients of the $q^0$-term of $\varphi \in J_{k,L,1}^{\w}$ satisfy
$$
f(0,\ell) \neq 0 \quad \Rightarrow \quad \latt{\ell,\ell} \leq \latt{\ell_1,\ell_1} \; \text{for all $\ell_1 \in \ell + L$}.
$$
\item The Fourier coefficients of $\varphi \in J_{k,L,t}^!$ satisfy
$$
f(n,\ell)=(-1)^k f(n,-\ell).
$$
\end{enumerate}
\end{lemma}

Let $J_{*,L,t}^{\w}$ (resp. $J_{*,L,t}$) denote the spaces of weak (resp. holomorphic) Jacobi forms of fixed index $t$ and arbitrary weight associated to $L$. Both are free graded modules of rank $|L'/L|$ over the graded ring of modular forms 
$$
M_*(\SL_2(\ZZ)):=\bigoplus_{k=0}^\infty M_k(\SL_2(\ZZ))=\CC[E_4,E_6].
$$

We will need the following raising index Hecke operators of Jacobi forms. 
\begin{lemma}[see Corollary 1 of \cite{Gri88}]\label{lem:Hecke}
Let $\varphi \in J_{k,L,t}^{!}$.  For any positive integer $m$, we have 
\begin{equation}\label{T(m)}
(\varphi \lvert T_{-}(m))(\tau,
\mathfrak{z}):=m^{-1}\sum_{\substack{ad=m,a>0\\ 0\leq b <d}}a^k \varphi
\left(\frac{a\tau+b}{d},a\mathfrak{z}\right) \in J_{k,L,mt}^{!},
\end{equation}
and the Fourier coefficients of 
$\varphi \lvert T_{-}(m)$ are given by the formula
$$
f_m(n,\ell)=\sum_{\substack{a\in \NN\\ a \mid (n,\ell,m)}}a^{k-1} f
\left( \frac{nm}{a^2},\frac{\ell}{a}\right),
$$
where $a\mid(n,\ell,m)$ means that $a\mid (n,m)$ 
and $a^{-1}\ell\in L'$. 
\end{lemma}

In \cite{Gri18} Gritsenko proved the following useful identities for Jacobi forms of weight $0$. These are variants of identities of Borcherds \cite[Theorem 10.5]{Bor98} that hold for vector-valued modular forms. 

\begin{lemma}\label{Lem:q^0-term}
Every nearly holomorphic Jacobi form of weight $0$ and index $1$ with Fourier expansion
$$
\phi(\tau,\mathfrak{z})=\sum_{n\in\ZZ}\sum_{\ell\in L'}f(n,\ell)q^n\zeta^\ell \in J_{0,L,1}^!
$$ 
satisfies the identity
\begin{equation}\label{eq:q^0-term}
C:= \frac{1}{24}\sum_{\ell\in L'}f(0,\ell)-\sum_{n<0}\sum_{\ell\in L'}f(n,\ell)\sigma_1(-n)=\frac{1}{2\rk(L)} \sum_{\ell\in L'}f(0,\ell)\latt{\ell,\ell}
\end{equation}
and the identity
\begin{equation}\label{eq:vectorsystem}
\sum_{\ell\in L'}f(0,\ell)\latt{\ell,\mathfrak{z}}^2=2C\latt{\mathfrak{z},\mathfrak{z}},
\end{equation}
where $\sigma_1(m)$ is the sum of the positive divisors of $m \in \mathbb{N}$.
\end{lemma}

\subsection{Additive lifts and Borcherds products}
In \cite{Bor95, Bor98} Borcherds introduced the singular theta lift to construct modular forms on type IV symmetric domains with specified divisors. We will quickly review this theory from the point of view of Jacobi forms.

\subsubsection{Modular forms for the Weil representation}
Let $\mathrm{Mp}_2(\mathbb{Z})$ be the metaplectic group, consisting of pairs $A = (A, \phi_A)$, where $A = \begin{psmallmatrix} a & b \\ c & d \end{psmallmatrix} \in \mathrm{SL}_2(\mathbb{Z})$ and $\phi_A$ is a holomorphic square root of $\tau \mapsto c \tau + d$ on $\mathbb{H}$, with the standard generators $T = (\begin{psmallmatrix} 1 & 1 \\ 0 & 1 \end{psmallmatrix}, 1)$ and $S = (\begin{psmallmatrix} 0 & -1 \\ 1 & 0 \end{psmallmatrix}, \sqrt{\tau})$. Let $M$ be an even lattice with bilinear form $(-,-)$ and quadratic form $Q(-)$. The \emph{Weil representation} $\rho_M: \mathrm{Mp}_2(\mathbb{Z})\to \GL (\mathbb{C}[M'/M])$ is defined by $$\rho_M(T) e_x = \mathbf{e}(-Q(x)) e_x \quad \text{and} \quad \rho_M(S) e_x = \frac{\mathbf{e}(\mathrm{sig}(M) / 8)}{\sqrt{|M'/M|}} \sum_{y \in M'/M} \mathbf{e}((x,y )) e_y,$$ 
where $e_\gamma$, $\gamma\in M'/M$ is the standard basis of the group ring $\CC[M'/M]$, and $\mathbf{e}(z)=e^{2\pi i z}$.

A \emph{nearly holomorphic modular form} of weight $k \in \frac{1}{2}\mathbb{Z}$ for the Weil representation $\rho_M$ is a holomorphic function $f : \mathbb{H} \rightarrow \mathbb{C}[M'/M]$ which satisfies $$f(A \cdot \tau) = \phi_A(\tau)^{2k} \rho_M(A) f(\tau), \quad \text{for all $(A, \phi_A) \in \mathrm{Mp}_2(\mathbb{Z})$}$$ and is meromorphic at the cusp $\infty$, that is, its Fourier expansion  $$f(\tau) = \sum_{x \in M'/M} \sum_{n \in \mathbb{Z} - Q(x)} c(n, x) q^n e_x$$
has only finitely many negative exponents. 
We denote the space of such forms by $M_k^!(\rho_M)$. 
We define the \textit{principal part} of $f$ by the sum
$$
\sum_{x \in M'/M} \sum_{n < 0} c(n, x) q^n e_x.
$$
We call $f$ a \textit{holomorphic} modular form if $c(n,x)=0$ whenever $n<0$, i.e. its principal part is zero. We denote the space of holomorphic modular forms of weight $k$ for $\rho_M$ by $M_k(\rho_M)$.

If $M$ splits two hyperbolic planes, then modular forms for the Weil representation can be identified with Jacobi forms. Let $U$ be a hyperbolic plane, i.e. the unique even unimodular lattice of signature $(1,1)$. We fix the basis of $U$ as
\begin{equation}\label{eq:basis of U}
    U=\ZZ e + \ZZ f, \quad (e,e)=(f,f)=0, \quad (e,f)=-1.
\end{equation}
Let $U_1$ be a copy of $U$ with a similar basis $e_1, f_1$. As in Section \ref{sec:Jacobi}, let $L$ be an even positive definite lattice with bilinear form $\latt{-,-}$. Then $M:=U_1\oplus U\oplus L$ is an even lattice of signature $(2+\rk(L),2)$ and $M'/M=L'/L$. Define the theta functions
$$
\Theta_{L, \gamma}(\tau,\mathfrak{z}) = \sum_{\ell \in L + \gamma} \exp\left(\pi i \latt{\ell,\ell}\tau + 2\pi i\latt{\ell, \mathfrak{z}}\right), \quad \gamma \in L'/L.
$$
The \emph{theta decomposition} yields an isomorphism between $M_{k-\rk(L)/2}^!(\rho_L)$ and $J_{k,L,1}^!$ :
$$
f(\tau)=\sum_{\gamma \in L'/L} f_\gamma(\tau) e_\gamma \mapsto \sum_{\gamma \in L'/L} f_\gamma(\tau) \Theta_{L,\gamma}(\tau, \mathfrak{z}). 
$$
For a nearly holomorphic Jacobi form $\varphi \in J_{k,L,1}^!$, the Fourier coefficients $f(n,\ell)q^n\zeta^\ell$ satisfying $2n-\latt{\ell,\ell}<0$ are called the \textit{singular Fourier coefficients}. The singular coefficients of $\varphi$ are precisely the coefficients which appear in the principal part at $\infty$ of the vector-valued modular form $f$. Therefore, the theta decomposition defines an isomorphism between the spaces of holomorphic forms $M_{k-\rk(L)/2}(\rho_L)$ and $J_{k,L,1}$. It follows in particular that the minimum possible weight of a non-constant holomorphic Jacobi form for $L$ is $\rk(L)/2$. This is called the \textit{singular weight}; since the associated vector-valued form $f$ must be constant, a singular-weight Jacobi form has its Fourier series supported on exponents $(n, \ell)$ of hyperbolic norm $2n - \latt{\ell, \ell} = 0$.

\subsubsection{Borcherds' singular additive lift}
When $M$ is of signature $(l,2)$ and $f \in M_{\kappa}^!(\rho_M)$ has weight $\kappa = k+1-l/2$ with integral $k\geq 1$, Borcherds' (singular) theta lift of $f$ is a meromorphic modular form of weight $k$ for $\widetilde{\Orth}^+(M)$ whose only singularities are poles of order $k$ along the hyperplanes $\cD_v(M)$ satisfying (see \cite[Theorem 14.3]{Bor98} for full details)
$$
\sum_{n=1}^\infty c(-n^2Q(v),nv)\neq 0, \quad \text{where $v\in M'$ is primitive.}
$$
This construction is called the \textit{Borcherds singular additive lift}. Note that the singular additive lift is holomorphic if and only if the principal part of the input is zero.

Suppose as before that $M = U_1 \oplus U \oplus L$ with $U = \ZZ e + \ZZ f$ and $U_1 = \ZZ e_1 + \ZZ f_1$. Around the one-dimensional cusp associated to the isotropic plane $\ZZ e_1 + \ZZ e$, $\cD(M)$ can be realized as the tube domain
$$
\cH(L):=  \{Z=(\tau,\mathfrak{z},\omega)\in \HH\times (L\otimes\CC)\times \HH: 
(\im Z,\im Z)<0\}, 
$$
where $(\im Z,\im Z)=-2\im \tau \im \omega +
\latt{\im \mathfrak{z},\im \mathfrak{z}}$. 
Let $F$ be a holomorphic modular form of weight $k$ and trivial character for $\widetilde{\Orth}^+(2U\oplus L)$. The Fourier expansion of $F$ on $\cH(L)$ has the shape
$$
F(\tau,\mathfrak{z},\omega) = \sum_{n=0}^\infty \sum_{m=0}^\infty \sum_{\ell \in L'} f(n,\ell,m)q^n\zeta^\ell\xi^m, \quad q=e^{2\pi i\tau}, \; \zeta^\ell =e^{2\pi i\latt{\ell,\mathfrak{z}}},\; \xi=e^{2\pi i\omega}.
$$
The holomorphy of $F$ on the boundary of $\cH(L)$ in $\cD(M)$ forces $f(n,\ell,m)=0$ if $2nm-\latt{\ell,\ell}<0$. This series can be reorganized as the \emph{Fourier--Jacobi expansion}
\begin{equation}\label{eq:FJdef}
    F(\tau, \mathfrak{z}, \omega) = \sum_{m = 0}^\infty f_m(\tau, \mathfrak{z})\xi^m, \quad \xi=e^{2\pi i \omega },
\end{equation}
where each $f_m$ is a holomorphic Jacobi form of weight $k$ and index $m$ associated to $L$. If $F$ is non-constant, then the weight must satisfy $k\geq \rk(L)/2$ since this is true for each $f_m$. In particular, the minimum possible weight of a non-constant modular form on a lattice of signature $(l,2)$ is $l/2-1$ (as before we assume $l\geq 3$), which is again called the \textit{singular weight}.

Through the theta decomposition, we can realize nearly-holomorphic Jacobi forms of weight $k$ and index $1$ for $L$ as the inputs into the weight $k$ singular additive lift. Combining \cite[Theorem 7.1, Theorem 9.3]{Bor95} and \cite[Theorem 14.3]{Bor98} yields a simple expression for the Fourier--Jacobi expansion of Borcherds' singular additive lifts in terms of the input Jacobi form.

\begin{theorem}[\cite{Bor95, Bor98}]\label{th:additive}
Suppose that $k\geq 1$ is integral and that $\varphi_k \in J_{k,L,1}^!$ has Fourier expansion
$$
\varphi_k(\tau,\mathfrak{z})=\sum_{ n \in \ZZ}\sum_{ \ell \in L'}f(n,\ell)q^n \zeta^\ell.
$$
Then the series
$$
\Grit(\varphi_k)(\tau,\mathfrak{z},\omega) = \sum_{m=0}^\infty (\varphi_k | T_{-}(m)) (\tau, \mathfrak{z}) \cdot \xi^m
$$
defines a meromorphic modular form of weight $k$ and trivial character for $\widetilde{\Orth}^+(2U\oplus L)$, where the Hecke operator $T_{-}(0)$ is
\begin{align*} 
(\varphi_k | T_{-}(0)) (\tau, \mathfrak{z}) &= -\frac{f(0,0)B_k}{2k} +  \sum_{(n,\ell)>0} \sum_{d | (n,\ell)} d^{k-1} f(0,\ell/d) q^n \zeta^\ell. 
\end{align*} 
Here, $B_n$ are the Bernoulli numbers defined by $\sum_{n=0}^\infty B_n t^k / n! = t/ (e^t -1)$, and $(n,\ell)>0$ means that either $n>0$ or $n=0$ and $\ell >0$. In particular, when the input has trivial $q^0$-term the zeroth Fourier--Jacobi coefficient $\varphi_k | T_{-}(0)$ is zero.
\end{theorem}

The image $\varphi_k | T_{-}(0)$ can be expressed in closed form in terms of the Weierstrass zeta function, $$\zeta(\tau, z) = \frac{1}{z} + \sum_{\substack{w \in \ZZ \tau + \ZZ \\ w \neq 0}} \left( \frac{1}{z - w} + \frac{1}{w} + \frac{z}{w^2} \right) = \frac{1}{z} + \frac{1}{2} \sum_{\substack{w \in \ZZ\tau + \ZZ \\ w \neq 0}} \left( \frac{1}{z - w} + \frac{1}{z + w} + \frac{2z}{w^2} \right).$$ Using the well-known partial fractions decomposition, $$\lim_{N \rightarrow \infty} \sum_{n = -N}^N \frac{1}{z + n} = \pi i - \frac{2\pi i}{1 - e^{2\pi i z}}, \quad z \notin \ZZ,$$ and applying Eisenstein summation, we find the Fourier series for $\zeta$: 
\begin{align*} \zeta(\tau, z) =& \frac{\pi^2}{3} E_2(\tau) z + \frac{1}{2} \lim_{M \rightarrow \infty}  \sum_{m = -M}^M \lim_{N \rightarrow \infty} \sum_{n=-N}^N \left( -\frac{1}{-z + m\tau + n} + \frac{1}{z + m \tau + n} \right) \\ 
=& \frac{\pi^2}{3} E_2(\tau) z + \pi i  \sum_{m = -\infty}^{\infty} \left( \frac{1}{1 - q^m \zeta^{-1}} - \frac{1}{1 - q^m \zeta} \right) \\ 
=& \frac{\pi^2}{3} E_2(\tau) z + \pi i \left( \frac{1}{1 - \zeta^{-1}} - \frac{1}{ 1 - \zeta} \right) \\
& + \pi i \sum_{m=1}^{\infty} \left( \frac{1}{1 - q^m \zeta^{-1}} - \frac{1}{1 - q^m \zeta} + \frac{q^m \zeta^{-1}}{1 - q^m \zeta^{-1}} - \frac{q^m \zeta}{1 - q^m \zeta} \right) \\
=& \frac{\pi^2}{3} E_2(\tau) z - \pi i - 2\pi i \sum_{n=1}^{\infty} \zeta^n - 2\pi i \sum_{m=1}^{\infty} \sum_{n=1}^{\infty} q^{mn} (\zeta^n - \zeta^{-n}), \; 0 < |\zeta| < |q|,
\end{align*}
where $\zeta$ also denotes $e^{2\pi i z}$. When $k = 1$, the nonexistence of (nonzero) modular forms of weight $2$ forces the identity $$\sum_{\ell > 0} f(0, \ell) \latt{\ell, \mathfrak{z}} \equiv 0,$$ which, together with $f(0, \ell) = -f(0, -\ell)$, implies
\begin{align*} 
\sum_{(n, \ell) > 0} \sum_{d | (n, \ell)} f(0, \ell / d) q^n \zeta^{\ell} &= \sum_{\ell > 0} f(0, \ell) \left( -\frac{1}{2\pi i} \zeta(\tau, \latt{\ell, \mathfrak{z}}) - \frac{\pi i}{6} E_2(\tau) \latt{\ell, \mathfrak{z}} \right) \\ 
&= -\frac{1}{2\pi i} \sum_{\ell > 0} f(0,\ell) \zeta(\tau, \latt{\ell, \mathfrak{z}}). 
\end{align*} 
For $k \ge 2$, a similar argument yields 
\begin{align*} 
\sum_{(n, \ell) > 0} \sum_{d | (n, \ell)} f(0, \ell/d) q^n \zeta^{\ell} &= \delta_k f(0, 0) \sum_{n = 1}^{\infty} \sigma_{k - 1}(n)q^n - \frac{1}{(2\pi i)^k} \sum_{\ell > 0} f(0, \ell) \zeta^{(k-1)}(\tau, \latt{\ell, \mathfrak{z}}) \\ 
&= -\frac{f(0, 0)B_k \delta_k}{2k} (E_k(\tau) - 1) + \frac{1}{(2\pi i)^k} \sum_{\ell > 0} f(0, \ell) \wp^{(k-2)}(\tau, \latt{\ell, \mathfrak{z}}), 
\end{align*} 
where $\delta_k = 1$ if $k$ is even and $\delta_k = 0$ if $k$ is odd; where $E_k(\tau)=1+O(q)$ is the normalized Eisenstein series of weight $k$ on $\SL_2(\ZZ)$; where $\wp(\tau, z) = -\partial_z \zeta(\tau, z)$ is the Weierstrass elliptic function; and where $\wp^{(k)} = \partial_z^k \wp(\tau, z)$.

Altogether, we have the formula 
\begin{equation}
 (\varphi_k | T_{-}(0))(\tau, \mathfrak{z}) = -\delta_k f(0, 0) \cdot \frac{B_k}{2k}E_k(\tau)- \frac{1}{(2\pi i)^k} \sum_{\ell > 0} f(0, \ell) \zeta^{(k-1)}(\tau, \latt{\ell, \mathfrak{z}}).    
\end{equation}

When $\varphi_k$ is a holomorphic Jacobi form, $\Grit(\varphi_k)$ is also holomorphic and is exactly the Gritsenko lift constructed in \cite{Gri88, Gri94}, which is a generalization of the Saito--Kurokawa lift or Maass lift.

\subsubsection{Borcherds products}
When the nearly holomorphic modular form $f$ for $\rho_M$ is of weight $1-l/2$, Borcherds discovered that the modified exponential of the singular theta lift of $f$ gives a meromorphic modular form of weight $c(0,0)/2$ and character for $\widetilde{\Orth}^+(M)$, which has an infinite product expansion and all of whose zeros or poles lie on hyperplanes $\cD_r(M)$, each with multiplicity 
$$
\sum_{n\in \NN} c(-n^2Q(r),n r), \quad \text{where $r\in M'$ is primitive.}
$$
This remarkable modular form is called a \textit{Borcherds product} denoted $\Borch(f)$. The following expression for Borcherds products in terms of Jacobi forms when $M = 2U \oplus L$ is due to Gritsenko and Nikulin \cite{GN98, Gri18}. 

\begin{theorem}[{\cite[Theorem 4.2]{Gri18}}]\label{th:product}
Let $\phi$ be a nearly holomorphic Jacobi form of weight $0$ and index $1$ for $L$ with Fourier expansion
\[
\phi(\tau,\mathfrak{z})=\sum_{n\in \ZZ}\sum_{\ell\in L'}f(n,\ell)q^n\zeta^\ell\in J_{0,L,1}^!
\]
satisfying $f(n,\ell)\in \ZZ$ for all $2n-\latt{\ell,\ell}\leq 0$. There is a meromorphic modular form of weight 
$f(0,0)/2$ and character $\chi$ with respect to  
$\widetilde{\Orth}^+(2U\oplus L)$, given by the series
\begin{equation}\label{eq:JacobiLift}
\Borch(\phi)(\tau,\mathfrak{z},\omega)=
\biggl(\Theta_{f(0,\ast)}
(\tau,\mathfrak{z})\cdot \xi^C \biggr)
\exp \left(-\sum_{m=1}^\infty (\phi | T_{-}(m)) (\tau, \mathfrak{z}) \cdot \xi^m \right),
\end{equation}
convergent in an open subset of $\cH(L)$, where $C$ is defined by \eqref{eq:q^0-term} and where
\begin{equation*}\label{FJtheta}
\Theta_{f(0,\ast)}(\tau,\mathfrak{z})
=\eta(\tau)^{f(0,0)}\prod_{\ell >0}
\biggl(\frac{\vartheta(\tau,\latt{\ell,\mathfrak{z}})}{\eta(\tau)} 
\biggr)^{f(0,\ell)}
\end{equation*}
is a generalized theta quotient.
The character $\chi$ is induced by the character of the above theta block
and by the relation $\chi(V)=(-1)^D$, where
$V\colon (\tau,\mathfrak{z}, \omega) \mapsto (\omega,\mathfrak{z},\tau)$, 
and $D=\sum_{n<0}\sigma_0(-n) f(n,0)$, here $\sigma_0(x)$ denotes the number of positive divisors of a positive integer $x$.
\end{theorem} 

\begin{remark}\label{rem:divisor}
We will rephrase some properties of the zeros and poles of additive lifts and Borcherds products in the context of Jacobi forms. Write vectors in $M = 2U\oplus L$ in the form $v = (a, b, \ell, c, d)$ with $\ell \in L$ and $a, b, c, d \in \ZZ$, such that $(v, v) = -(ad+bc)+\latt{\ell, \ell}$.

(1) The Eichler criterion (see e.g. \cite[Proposition 3.3]{GHS09}) states that if $v_1$ and $v_2$ are primitive vectors of $M'$ that have the same norm and satisfy $v_1-v_2\in M$, then there exists $g\in \widetilde{\Orth}^+(M)$ such that $g(v_1)=v_2$.
Let $v$ be a primitive vector of positive norm in $M'$. Then there exists a vector $(0,n,\ell,1,0)\in  2U
\oplus L'$ such that $(v,v)=-2n+\latt{\ell,\ell}$ and
$v- (0,n,\ell,1,0)\in M$. Therefore, the Heegner divisor $\cH(\frac{1}{2}(v,v),v)$ is exactly the $\widetilde{\Orth}^+(M)$-orbit of $\cD_{(0,n,\ell,1,0)}(M)$. Its restriction to the tube domain $\cH(L)$ is $$\cD_{(0, n, \ell, 1, 0)} \cap \cH(L) = \{Z = (\tau, \mathfrak{z}, \omega): \; \tau - \latt{\ell, \mathfrak{z}} + n\omega = 0\}.$$

(2) In Theorem \ref{th:additive}, the singular additive lift $\Grit(\varphi_k)$ has poles of order $k$ along $\cD_{(0,n,\ell,1,0)}(M)$ if and only if $2n-\latt{\ell,\ell}<0$ and 
$$
\delta(n,\ell):= \sum_{d=1}^\infty f(d^2 n, d\ell)\neq 0.
$$

(3) In Theorem \ref{th:product}, the zeros or poles of the Borcherds product $\Borch(\phi)$ lie on $\cD_{(0,n,\ell,1,0)}(M)$ with multiplicity $\delta(n,\ell)$.

(4) By Lemma \ref{lem:periodic}, the singular Fourier coefficients of 
$$
\varphi_k = \sum_{n\geq n_0}\sum_{\ell \in L'} f(n,\ell) q^n \zeta^\ell \in J_{k,L,1}^!
$$
are represented by
$$
f(n,\ell), \quad n_0 \leq  n \leq \widehat{\delta}_L, \; \ell \in L',\; 2n-\latt{\ell,\ell}<0,
$$
where $\widehat{\delta}_L$ is the largest integer less than $\delta_L/2$ and as in the introduction
$$
\delta_L := \max\{ \min\{\latt{y,y}: y\in L + x \} : x \in L' \}.
$$

(5) Let $\ell$ be a nonzero vector of $L'$. If $\varphi_k \in J_{k,L,1}^{\w}$ has trivial $q^0$-term (i.e. $\varphi_k = O(q)$) and vanishes on the divisor 
$$
\{(\tau, \mathfrak{z})\in \HH\times (L\otimes\CC): \latt{\ell,\mathfrak{z}} \in \ZZ\tau+\ZZ\},
$$ 
then $\Grit(\varphi_k)$ vanishes on $\cD_{(0,0,\ell,1,0)}(M)$.
\end{remark}

\subsection{Weyl-invariant weak Jacobi forms}\label{subsec:Weyl}
In this subsection we review some known results about the algebras of Jacobi forms, which will be used to prove the main theorems. 

Let $G$ be a subgroup of $\Orth(L)$. A Jacobi form $\varphi$ for $L$ is called \textit{$G$-invariant} if it satisfies
$$
\varphi(\tau, \sigma(\mathfrak{z})) = \varphi(\tau, \mathfrak{z}), \quad \text{for all $\sigma\in G$}.
$$
All $G$-invariant weak Jacobi forms of integral weight and integral index for $L$ form a bigraded algebra over $\CC$
$$
J_{*,L,*}^{\w, G} = \bigoplus_{t\in \NN} J_{*,L,t}^{\w,G}, \quad \text{where} \quad J_{*,L,t}^{\w,G}=\bigoplus_{k\in\ZZ} J_{k,L,t}^{\w,G}.
$$
It is conjectured that the algebra $J_{*,L,*}^{\w,G}$ is always finitely generated.  The most important result in this direction is due to Wirthm\"uller. 

Let $R$ be an irreducible root system (cf. \cite{Bou60}) of rank $\rk(R)$, with Weyl group $W(R)$. The root lattice $L_R$ is the lattice generated by the roots of $R$ together with the standard scalar product, rescaled by two if that lattice is odd. When $R$ is not of type $E_8$, Wirthm\"uller \cite{Wir92} showed that $J_{*,L_R,*}^{\w,W(R)}$ is a polynomial algebra over $\CC[E_4,E_6]$ generated by $\rk(R)+1$ weak Jacobi forms. The weights and indices of the generators are invariants of the root system $R$.  More precisely, we know the following. 
\begin{enumerate}
    \item There is always a generator of weight $0$ and index $1$.
    \item The other indices are the coefficients of the dual of the highest coroot of $R$, written as a linear combination of the simple roots of $R$. 
    \item The negative weights of the other generators are the degrees of the generators of the ring of $W(R)$-invariant polynomials, or equivalently, the exponents of $W(R)$ increased by $1$.
\end{enumerate}

It was proved in \cite{Wan21b} that $J_{*,E_8,*}^{\w,W(E_8)}$ is not a polynomial algebra, and we refer to \cite{KW21} for an explicit description of its structure. An automorphic proof of  Wirthm\"uller's theorem was given in \cite{Wan21c}.  

We provide some information about the generators of $J_{*,R,*}^{\w,W(R)}$. We formulate the weights $k_j$ and indices $m_j$ of generators in Table \ref{Tab:Jacobi} below. We note that $A_3=D_3$ and
$$
W(B_n) = \Orth(nA_1), \quad W(G_2) = \Orth(A_2), \quad W(F_4) = \Orth(D_4), \quad W(C_n) = \Orth(D_n) \; \text{if $n\neq 4$}.
$$
The $W(A_1)$-invariant weak Jacobi forms are actually classical weak Jacobi forms of even weight introduced by Eichler--Zagier \cite{EZ85}. 
Explicit constructions of generators of type $A_n$, $B_n$ and $D_4$ were first obtained in \cite{Ber00a, Ber00b}. The generators of type $C_n$ and $D_n$ were obtained in \cite{AG20}. The generators of type $E_6$ and $E_7$ were constructed in \cite{Sak19}, and the generators of type $F_4$ were constructed in \cite{Adl20}. 
We fix some notation for the generators to avoid confusion later.
\begin{Notation}\label{notation}
\noindent
\begin{enumerate}
    \item When $R=A_n, E_6, E_7$, there are no distinct generators of the same weight and index. Thus we use $\phi_{k,R,t}$ to stand for the generator of weight $k$ and index $t$ associated to $W(R)$. 
    \item Since $W(C_n)$ is generated by $W(D_n)$ and the reflection which changes the sign of any fixed coordinate, we can choose the generators of $J_{*,D_n,*}^{\w,W(D_n)}$ in the following way:
\begin{itemize}
    \item[(a)] Index one: $\phi_{0,D_n,1}$, $\phi_{-2,D_n,1}$, $\phi_{-4,D_n,1}$, which are invariant under $W(C_n)$.
    \item[(b)] Index two: $\phi_{-2k,D_n,2}$ for $3\leq k \leq n-1$, which are invariant under $W(C_n)$.
    \item[(c)] Index one: $\psi_{-n,D_n,1}$, which is invariant under $W(D_n)$, but anti-invariant under the above reflection. 
\end{itemize}
The forms $\phi_{-,D_n,-}$ together with $\psi_{-n,D_n,1}^2$ form a system of generators of $J_{*,D_n,*}^{\w,W(C_n)}$. 
\end{enumerate}
\end{Notation}

\begin{table}[ht]
\caption{Weights $k_j$ and indices $m_j$ of generators of $J_{*,L_R,*}^{\w,W(R)}$ ($A_n: n\geq 1$, $D_n: n\geq 4$, $B_n: n\geq 2$, $C_n: n\geq 3$)}\label{Tab:Jacobi}
\renewcommand\arraystretch{1.3}
\noindent\[
\begin{array}{|c|c|c|}
\hline
R & L_R & (k_j,m_j) \\ 
\hline 
A_n &  A_n& (0,1), (-s,1) : 2\leq s\leq n+1\\ 
\hline 
D_n & D_n &  (0,1), (-2,1), (-4,1), (-n,1), (-2s,2) : 3\leq s \leq n-1 \\ 
\hline
E_6 & E_6 & (0,1), (-2,1), (-5,1), (-6,2), (-8,2), (-9,2), (-12,
3)  \\ 
\hline
E_7 & E_7 & (0,1), (-2,1), (-6,2), (-8,2), (-10,2), (-12,
3), (-14,3), (-18,4)  \\ 
\hline
B_n & n A_1 & (-2s,1) : 0\leq s \leq n \\
\hline
C_n & D_n & (0,1), (-2,1), (-4,1), (-2s,1) : 3\leq s \leq n \\
\hline
G_2 & A_2 & (0,1), (-2,1), (-6,2) \\
\hline
F_4 & D_4 & (0,1), (-2,1), (-6,2), (-8,2), (-12,3)\\
\hline
\end{array} 
\]
\end{table}

\begin{remark}\label{rem:sum of Jacobi forms}
When $L$ is a direct sum of irreducible root lattices, we can determine the algebra of weak Jacobi forms for $L$ using \cite[Theorem 2.4]{WW21c}. Let $L=\bigoplus_{j=1}^n R_j$ and $G=\bigotimes_{j=1}^n W(R_j)$. As a free module over $M_*(\SL_2(\ZZ))$, the space of weak Jacobi forms $J_{*,L,t}^{\w, G}$ of given index $t$ is generated by the tensor products of generators of these $J_{*,R_j,t}^{\w,W(R_j)}$. In other words, there is an isomorphism of graded $\CC[E_4,E_6]$-modules
$$
J_{*,R_1\oplus \cdots \oplus R_n, t}^{\w, W(R_1)\otimes \cdots \otimes W(R_n)} \cong J_{*,R_1,t}^{\w, W(R_1)} \otimes \cdots \otimes J_{*,R_n,t}^{\w, W(R_n)}. 
$$
\end{remark}

\section{A proof of the theta block conjecture}\label{sec:theta block conjecture}
In this section we prove Theorem \ref{MTH3} and prove the theta block conjecture as a corollary. As we mentioned in the introduction, Gritsenko, Skoruppa and Zagier \cite{GSZ19} developed the theory of theta blocks to construct holomorphic Jacobi forms of low weight for $A_1$. We recall their construction and its generalization to Jacobi forms of lattice index.  The Dedekind eta function
$$
\eta(\tau)= q^{\frac{1}{24}}\prod_{n=1}^\infty (1-q^n), \quad q=e^{2\pi i\tau}
$$
is a modular form of weight $1/2$ on $\SL_2(\ZZ)$ with a multiplier system of order $24$ denoted $v_\eta$. The odd Jacobi theta function 
$$
\vartheta(\tau,z)=q^{\frac{1}{8}}(e^{\pi iz}-e^{-\pi i z})\prod_{n=1}^\infty (1-q^ne^{2\pi i z})(1-q^ne^{-2\pi iz})(1-q^n), \quad z\in \CC
$$
is a holomorphic Jacobi form of weight $1/2$ and index $1/2$ for $A_1$ with a multiplier system of order $8$ (see \cite{GN98}). More precisely, $\vartheta$ satisfies the transformation laws
\begin{align*}
\vartheta (\tau, z+ x \tau + y)&= (-1)^{x+y} \exp(- \pi i ( x^2 \tau +2xz )) \vartheta ( \tau, z ), \quad   x,y \in \ZZ,\\
\vartheta \left( \frac{a\tau +b}{c\tau + d},\frac{z}{c\tau + d} \right) &= \upsilon_{\eta}^3 (A) \sqrt{c\tau + d} \exp\left(\frac{\pi i c z^2}{c \tau + d} \right) \vartheta ( \tau, z ), \quad  A=\left( \begin{array}{cc}
a & b \\ 
c & d
\end{array} \right)   \in \SL_2(\ZZ). 
\end{align*}
A (pure) \textit{theta block} is a holomorphic function of the form
$$
\Theta_f(\tau,z)=\eta(\tau)^{f(0)} \prod_{a=1}^\infty (\vartheta(\tau,a z)/\eta(\tau))^{f(a)}, \quad (\tau, z) \in \HH \times \CC,
$$
where $f: \NN \to \NN$ is a function with finite support. From the modular properties of $\eta$ and $\vartheta$ we see that $\Theta_f$ defines a weak Jacobi form of weight $f(0)/2$ and index $m_f$ for $A_1$, with multiplier system $v_\eta^{d_f}$ and leading term $q^{d_f/24}$ in its Fourier expansion, where 
$$
m_f=\frac{1}{2}\sum_{a=1}^\infty a^2 f(a), \quad d_f=f(0)+2\sum_{a=1}^\infty f(a).
$$
The number $d_f/24$ is called the \textit{$q$-order} of $\Theta_f$. This is called a \emph{holomorphic theta block} if it is holomorphic as a Jacobi form, i.e. its singular Fourier coefficients vanish. For example, the theta block $\eta^{-6}\vartheta^4\vartheta_2^3\vartheta_3^2\vartheta_4$ gives an explicit construction of the Jacobi Eisenstein series of weight $2$ and index $25$, where $\vartheta_a:=\vartheta(\tau,a z)$. In \cite{GSZ19} Gritsenko, Skoruppa and Zagier associated to a root system $R$ an infinite family of holomorphic theta blocks of weight $\rk(R)/2$. These infinite families are closely related to the famous Macdonald identities (see \cite[Theorem 10.1, Theorem 10.6]{GSZ19}).  Some of them also appear as the leading Fourier--Jacobi coefficients of holomorphic Borcherds products of singular weight which vanish precisely on mirrors of reflections (see \cite{DW20}). 

Siegel paramodular forms of level $N$ are Siegel modular forms for the paramodular group of degree two and level $N$,
\[
K(N)=\begin{pmatrix}
\ast & N\ast & \ast & \ast \\ \ast & \ast & \ast & \ast/N \\ \ast & N\ast & \ast & \ast \\ N\ast & N\ast & N\ast &\ast
\end{pmatrix}\cap \Sp_4(\QQ), \quad \text{ all } \ast\in \ZZ. 
\]
Paramodular forms can be identified with modular forms for $\widetilde{\Orth}^+(2U\oplus A_1(N))$ (see \cite{GN98}), so we can construct Siegel paramodular forms using Gritsenko lifts and Borcherds products. In \cite{GPY15} Gritsenko, Poor and Yuen formulated the \textit{theta block conjecture} which characterizes Siegel paramodular forms which are simultaneously Borcherds products and Gritsenko lifts. 
\begin{conjecture}[Conjecture 8.1 in \cite{GPY15}]\label{conj:theta}
    Suppose that the $q$-order one theta block $\Theta_f$ is a holomorphic Jacobi form of integral weight $k$ and integral index $N$ for $A_1$. Then as a Siegel paramodular form of weight $k$ and level $N$ the Gritsenko lift of $\Theta_f$ is a Borcherds product. More precisely,
    \[
    \Grit(\Theta_f)=\Borch\left( - \frac{\Theta_f|T_{-}(2)}{\Theta_f} \right).
    \]
\end{conjecture}

This conjecture has been proved for all (eight) infinite families of holomorphic theta blocks of $q$-order one which are related to root systems (see \cite{GPY15, Gri18, GW20, DW20}). Every theta block of $q$-order one has weight $k\leq 11$ and can be written as 
\begin{equation}\label{eq:theta q-order 1}
   \Theta_{\mathbf{a}}(\tau, z) :=  \eta^{3k-12} \prod_{j=1}^{12-k}\vartheta_{a_j}, \quad \vartheta_{a_j}:=\vartheta(\tau,a_j z).
\end{equation}
When $k\geq 4$, every theta block of $q$-order one is a holomorphic Jacobi form, and it comes from the infinite family $\prod_{j=1}^8 \vartheta_{a_j}$ associated to the root system $8A_1$ and its quasi pullbacks, in which case Conjecture \ref{conj:theta} was proved. However, when $k=2$ or $3$, the theta block \eqref{eq:theta q-order 1} is usually not a holomorphic Jacobi form, and there do exist holomorphic theta blocks of $q$-order one which do not belong to any of the infinite families associated to root systems. In these cases, Conjecture \ref{conj:theta} has remained open. We remark that there are no holomorphic Jacobi forms of weight $1$ for $A_1$ (see \cite{Sko84}) so the theta block \eqref{eq:theta q-order 1} is never holomorphic when $k=1$.

It is in fact easier to prove the generalization of Conjecture \ref{conj:theta} to meromorphic modular forms on higher-dimensional type IV domains. We first define theta blocks associated to an even positive definite lattice $L$ with bilinear form $\latt{-,-}$. For any finite set $\mathbf{s}=\{s_j\in L' : 1\leq j \leq d \}$, the function 
\begin{equation}
    \Theta_{\mathbf{s}}(\tau,\mathfrak{z}) = \eta(\tau)^{24-3d} \prod_{j=1}^d \vartheta(\tau, \latt{s_j, \mathfrak{z}}), \quad \mathfrak{z}\in L\otimes\CC
\end{equation}
is called a \textit{theta block of $q$-order one associated to $L$} if it defines a weak Jacobi form of index $1$ for $L$. This holds if and only if $\mathbf{s}$ satisfies the identity
$$
\sum_{j=1}^d \latt{s_j, \mathfrak{z}}^2 = \latt{\mathfrak{z},\mathfrak{z}},
$$
or equivalently, the map 
$$
\iota_{\mathbf{s}}: \quad \ell \mapsto (\latt{s_1, \ell}, ..., \latt{s_d, \ell})
$$
defines an embedding of $L$ into the odd unimodular lattice $\ZZ^d$ and therefore into its maximal even sublattice $D_d$.

Note that there exist lattices which cannot be embedded into $\ZZ^d$ for any $d$; however, for a given lattice $L$ there exists a positive integer $m$ such that $L(m)$ can be embedded into some $\ZZ^d$ (see \cite{CS89}).  For example, there is no embedding from $E_6$ into any $\ZZ^d$, but $E_6(2)$ is a sublattice of $\ZZ^7$. Therefore, for any given lattice, one can construct theta blocks of $q$-order one and of sufficiently large index.  

We first show that the generalized theta block conjecture is true for the $q$-order one theta blocks associated to $D_n$. Let $(\varepsilon_1,...,\varepsilon_n)$ denote the standard orthogonal basis of $\RR^n$.  We fix the model 
\begin{equation}\label{model of D_n}
D_n=\left\{ (x_i)_{i=1}^n \in \ZZ^n : \sum_{i=1}^n x_i \in 2\ZZ \right\} 
\end{equation}
and write $\mathfrak{z} \in D_n \otimes \CC$ in coordinates $\mathfrak{z} = (z_1,...,z_n)$. 

\begin{theorem}\label{th:theta D_n}
For $1\leq n \leq 11$, as meromorphic modular forms of weight $12-n$ for $\widetilde{\Orth}^+(2U\oplus D_n)$,  the following identity holds:
\begin{equation}
    \Grit(\vartheta_{D_n}) = \Borch\left( - \frac{\vartheta_{D_n}|T_{-}(2)}{\vartheta_{D_n}} \right),
\end{equation}
where the theta block $\vartheta_{D_n}$ is
\begin{equation}\label{eq:theta block D_n}
\vartheta_{D_n}(\tau, \mathfrak{z}) = \eta(\tau)^{24-3n}\prod_{j=1}^n \vartheta(\tau, \latt{\varepsilon_j,\mathfrak{z}}) =\eta(\tau)^{24-3n} \prod_{j=1}^n \vartheta(\tau, z_j).
\end{equation}
\end{theorem}

When $n\leq 8$, the theta block $\vartheta_{D_n}$ is a holomorphic Jacobi form, and this was proved in \cite[Theorem 8.2]{GPY15}. We only need to consider the remaining cases $n=9$, $10$, $11$.

\begin{proof}
Let $\cH_n$ be the full Heegner divisor of discriminant $n/4-2$: $$\cH_n = \bigcup_{\substack{v \in M_n' \\ (v, v) = n/4 - 2}} \cD_v(M_n)$$ in $M_n = 2U \oplus D_n$.  Since
$$
\delta_{D_n}:=\max\{ \min\{ \latt{v,v} : v \in D_n + \gamma  \}: \gamma \in D_n' \} = n/4,
$$
Remark \ref{rem:divisor} (4) implies that representatives of all singular Fourier coefficients of $\vartheta_{D_n}$ appear in the $q^1$-term of its Fourier expansion. (Note that $\vartheta_{D_n}= O(q)$.) Therefore, the only possible singularities of $\Grit(\vartheta_{D_n})$ lie on the arrangement $\cH_n$.

Write $\Psi_{D_n} := -\vartheta_{D_n} | T_{-}(2) / \vartheta_{D_n}$. Since $\vartheta(\tau, z)$ vanishes, if $\tau \in \mathbb{H}$ is fixed, with simple zeros exactly when $z \in \ZZ \tau + \ZZ$, it follows that $\vartheta_{D_n}(\tau, \mathfrak{z})$ vanishes exactly when one of the coordinates $z_j$ is a lattice vector. This implies that $\Psi_{D_n}$ is holomorphic and therefore a weak Jacobi form of weight $0$ and index $1$ for $D_n$. By Lemma \ref{lem:Hecke} and the infinite product expansion of $\vartheta_{D_n}$,  the input $\Psi_{D_n}$ has integral Fourier expansion beginning $$\Psi_{D_n}(\tau, \mathfrak{z}) = \sum_{j=1}^n (e^{2\pi i z_j} + e^{-2\pi i z_j}) + 2 (12 - n) + O(q).$$ From this we see that the Borcherds product $\mathrm{Borch}(\Psi_{D_n})$ has weight $12-n$ and simple zeros on the irreducible Heegner divisor $H(1/2,\varepsilon_1)$, and that its other zeros and poles are contained in $\cH_n$.

Remark \ref{rem:divisor} (5) implies that $\Grit(\vartheta_{D_n})$ vanishes on the divisor $H(1/2,\varepsilon_1)$. It follows that the quotient $\mathrm{Grit}(\vartheta_{D_n}) / \mathrm{Borch}(\Psi_{D_n})$ has weight zero and is holomorphic away from $\cH_n$.

We will show in Lemma \ref{Lem:KoecherD_n} below that $\cH_n$ satisfies the Looijenga condition. By Koecher's principle, i.e. Lemma \ref{lem:Koecher}, $\mathrm{Grit}(\vartheta_{D_n}) / \mathrm{Borch}(\Psi_{D_n})$ is a constant. By Theorem \ref{th:additive} and Theorem \ref{th:product}, the leading Fourier--Jacobi coefficients of these forms are both $\vartheta_{D_n}\cdot \xi$, which implies that the constant is $1$. 
\end{proof}

\begin{lemma}\label{Lem:KoecherD_n}
\noindent
\begin{enumerate}
    \item Every intersection of two distinct hyperplanes in $\cH_9$ is empty.
    \item Every nonempty intersection of hyperplanes in $\cH_{10}$ has dimension $11$ or $10$.
    \item Every nonempty intersection of hyperplanes in $\cH_{11}$ has dimension $12$, $11$ or $10$.
\end{enumerate}
\end{lemma}
\begin{proof}
(1) Since $\cH_9 = \cH(1/8, \sum_{j=1}^9\varepsilon_j/2)$, the claim follows from Lemma \ref{lem:intersection} (1).

(2) Write $\cH_{10}=\cH(1/4,v_1)\cup \cH(1/4,v_2)$ where $v_1, v_2 \in D_{10}'/D_{10}$ are the cosets containing non-integral vectors. More precisely, $v_1=\sum_{j=1}^{10}\varepsilon_j/2$ and $v_2=v_1-\varepsilon_1$. By Lemma \ref{lem:intersection} (1), the intersection of two different hyperplanes in the same irreducible component is empty, so nonempty intersections only arise from the intersection of one hyperplane in $\cH(1/4,v_1)$ and one hyperplane in $\cH(1/4,v_2)$. This yields the claim. Since $2U\oplus D_{10} \cong 2U\oplus E_8\oplus 2A_1$ we see that the $10$-dimensional intersections are all symmetric spaces associated to the lattice $2U\oplus E_8$. 

(3) We have $\cH_{11}=H(3/8, v)$ with $v=\sum_{j=1}^{11}\varepsilon_j/2$. By Lemma \ref{lem:intersection} (2), the restriction of $\cH(3/8, v)$ to a hyperplane in $\cH(3/8, v)$ is a Heegner divisor of type $\cH(1/3,-)$. Similarly, the restriction of $\cH(1/3, -)$ to a hyperplane in $\cH(1/3, -)$ is a Heegner divisor of type $\cH(1/4,-)$. The intersection of two different hyperplanes in $\cH(1/4,-)$ is empty. In particular, all nonempty intersections have dimension at least $10$. From $2U\oplus D_{11} \cong 2U\oplus E_8\oplus A_3$ we see that the $11$-dimensional (resp. $10$-dimensional) intersections correspond to the lattice $2U\oplus E_8\oplus A_1$ (resp. $2U\oplus E_8$). 
\end{proof}

\begin{remark}
By calculating the $q^1$-term of $\Psi_{D_n}$, we find the explicit divisor:
$$
\Div(\Borch{\Psi_{D_n}}) = H(1/2,\varepsilon_1) - (12-n)\cH_n, \quad n=9,10,11.
$$
Since $\Grit(\vartheta_{D_n})$ can have at worst poles of order $12-n$ along the components of $\cH_n$, one can obtain a more direct proof  that the quotient $\Grit(\vartheta_{D_n}) / \Borch(\Psi_{D_n})$ is holomorphic everywhere and therefore constant. Our first proof is more satisfying: for one thing, it does not require explicit computation of $\Psi_{D_n}$, and for another it shows that both $\Grit(\vartheta_{D_n})$ and $\Borch(\Psi_{D_n})$ lie in the space of modular forms of weight $12-n$ for $\widetilde{\Orth}^+(2U\oplus D_n)$ with poles on the arrangement $\cH_n$, and the equality is guaranteed by the divisor of the theta block and the Looijenga condition on $\cH_n$. We will see in the next section that this vector space is one-dimensional.  
\end{remark}

\begin{corollary}\label{cor:theta}
Let $\Theta_{\mathbf{s}}$ be a theta block of positive weight and $q$-order one associated to $L$. Then we have the identity
$$
\Grit(\Theta_{\mathbf{s}}) = \Borch\left( - \frac{\Theta_{\mathbf{s}}|T_{-}(2)}{\Theta_{\mathbf{s}}} \right).
$$
\end{corollary}
\begin{proof}
The definition of $\Theta_{\mathbf{s}}$ yields an embedding $\iota_\mathbf{s}$ of $L$ into $D_n$, where $n\geq \rk(L)$ is the number of elements of the set $\mathbf{s}$. Then $\Theta_{\mathbf{s}}$ is the pullback of $\vartheta_{D_n}$ along $\iota_{\mathbf{s}}$. Moreover, $\Grit(\Theta_{\mathbf{s}})$ is the pullback of $\Grit(\vartheta_{D_n})$ along the induced embedding $\iota_{\mathbf{s}}: 2U\oplus L \hookrightarrow 2U\oplus D_n$. It is also known that the pullback of a Borcherds product is again a Borcherds product (see \cite{Ma19}). By Theorem \ref{th:theta D_n}, we find that $\Grit(\Theta_{\mathbf{s}})$ is a Borcherds product. The expression $- \left(\Theta_{\mathbf{s}}|T_{-}(2)\right) / \Theta_{\mathbf{s}}$ for the input into the Borcherds product can be determined by comparing the Fourier--Jacobi expansions of singular additive lifts and Borcherds products in Theorem \ref{th:additive} and Theorem \ref{th:product}. 
\end{proof}

\begin{corollary}
The theta block conjecture \ref{conj:theta} is true.  
\end{corollary}
\begin{proof}
Every holomorphic theta block of weight $k$ and $q$-order one is of the form $\Theta_{\mathbf{s}}$ for a subset $\mathbf{s} \subseteq A_1(N)'$, for some $N \in \mathbb{N}$.
\end{proof}

The above two corollaries also imply the ``if" part of \cite[Conjecture 4.10]{GW20} which is a generalization of the theta block conjecture to holomorphic modular forms on higher-rank orthogonal groups. The ``only if" part remains open.

\section{Free algebras of meromorphic modular forms on type IV symmetric domains}\label{sec:type IV}
In this section we prove Theorem \ref{MTH2}. We first state a more detailed form of Theorem \ref{MTH2}, then prove the existence of the potential modular Jacobians, and finally construct the generators. 

\subsection{A precise statement of Theorem \ref{MTH2} and the outline of its proof}
\subsubsection{The construction of the hyperplane arrangement $\cH_L$}
Let $L$ be a lattice in the following three families of root lattices:
\begin{equation}\label{eq:lattices}
\begin{aligned}
&\text{$A$-type:}& & \left(\bigoplus_{j=1}^t A_{m_j}\right)\oplus A_m, \quad t\geq 0, \quad m\geq 1, \quad (m+1)+\sum_{j=1}^t (m_j+1) \leq 11; \\
&\text{$AD$-type:}& & \left(\bigoplus_{j=1}^t A_{m_j}\right)\oplus D_m,  \quad t\geq 0, \quad m\geq 4, \quad   m + \sum_{j=1}^t (m_j+1) \leq 11; \\
&\text{$AE$-type:}& & E_6, \quad A_1 \oplus E_6, \quad A_2\oplus E_6, \quad E_7,\quad  A_1\oplus E_7.    
\end{aligned}    
\end{equation}
Corresponding to the above decomposition, we write $L = L_0 \oplus L_1$, where $L_1$ is $A_m$, $D_m$ or $E_6$ or $E_7$. When $t = 0$, we take $L_0 = \{0\}$. 

\begin{remark}
Different decompositions  may yield equivalent lattices, e.g. $A_2 \oplus A_1 \cong A_1 \oplus A_2$; however, we will view them as distinct as long as the $L_1$ are distinct, because we associate different hyperplane arrangements to them. In this sense, there are $97$ $A$-type lattices and $45$ $AD$-lattices, and therefore $147$ lattices altogether. There are only two repeats in the sense that both the lattices and the hyperplane arrangements are the same (and even in these cases, the decomposition of $\cH_L$ into $\cH_{L, 0}$ and $\cH_{L, 1}$ and therefore the classification below of the generators into abelian and Jacobi type generators is distinct):
\begin{align*}
2U \oplus (A_1 \oplus A_3) \oplus D_4 &= 2U \oplus 3A_1 \oplus D_5,\\
2U \oplus A_3 \oplus D_6 &= 2U \oplus 2A_1 \oplus D_7.     
\end{align*}
\end{remark}

We list the lattices of $A$-type and $AD$-type in Table \ref{Tab:MTH2}.

\begin{table}[ht]
\caption{Lattices of $A$-type and $AD$-type in the families \eqref{eq:lattices}}\label{Tab:MTH2}
\renewcommand\arraystretch{1.0}
\noindent\[
\begin{array}{|c|c|}
\hline
& \\ [-4mm]
L_1 & L_0 \\ \hline
& \\ [-4mm]
A_1 & 0,\; A_1,\; 2A_1,\; 3A_1,\; 4A_1,\; A_1\oplus A_2,\; A_1\oplus A_3,\; A_1\oplus A_4,\; A_1\oplus A_5,\; A_1\oplus A_6, \\
& A_1\oplus 2A_2,\; A_1\oplus A_2\oplus A_3,\; 2A_1\oplus A_2,\; 2A_1\oplus A_3,\; 2A_1\oplus A_4,\; 3A_1\oplus A_2,\; A_2,\; 2A_2,\\
& 3A_2,\; A_2\oplus A_3,\; A_2\oplus A_4,\; A_2\oplus A_5,\; A_3,\; 2A_3,\; A_3\oplus A_4,\; A_4,\; A_5,\; A_6,\; A_7,\; A_8 \\ \hline
& \\ [-4mm]
A_2 & 0,\; A_1,\; 2A_1,\; 3A_1,\; 4A_1,\; A_1\oplus A_2,\; A_1\oplus A_3,\; A_1\oplus A_4,\; A_1\oplus A_5,\; A_1\oplus 2A_2,\\
& 2A_1\oplus A_2,\; 2A_1\oplus A_3,\; A_2,\; 2A_2,\; A_2\oplus A_3,\; A_2\oplus A_4,\; A_3,\; 2A_3,\; A_4,\; A_5,\; A_6,\; A_7 \\ \hline
& \\ [-4mm]
A_3,\; D_4 & 0,\; A_1,\; 2A_1,\; 3A_1,\; A_1\oplus A_2,\; A_1\oplus A_3,\; A_1\oplus A_4,\\
& 2A_1\oplus A_2,\; A_2,\; 2A_2,\; A_2\oplus A_3,\; A_3,\; A_4,\; A_5,\; A_6 \\ \hline
& \\ [-4mm]
A_4,\; D_5 & 0,\; A_1,\; 2A_1,\; 3A_1,\; A_1\oplus A_2,\; A_1\oplus A_3,\; A_2,\; 2A_2,\; A_3,\; A_4,\; A_5\\ \hline
& \\ [-4mm]
A_5,\; D_6 & 0,\; A_1,\; 2A_1,\; A_1\oplus A_2,\; A_2,\; A_3,\; A_4 \\ \hline
& \\ [-4mm]
A_6,\; D_7 & 0,\; A_1,\; 2A_1,\; A_2,\; A_3 \\ \hline
& \\ [-4mm]
A_7,\; D_8 & 0,\; A_1,\; A_2 \\ \hline
& \\ [-4mm]
A_8,\; D_9 & 0,\; A_1 \\ \hline
& \\ [-4mm]
A_9,\; D_{10} & 0 \\ \hline
& \\ [-4mm] 
A_{10},\; D_{11} & 0 \\
\hline
\end{array}
\]
\end{table}

We will construct a hyperplane arrangement $\cH_L$ for every $L$. View $A_n$ as a sublattice of $D_{n+1}$:
$$
A_n = \left\{ (x_i)_{i=1}^{n+1} \in \ZZ^{n+1}: \sum_{i=1}^{n+1} x_i =0 \right\}. 
$$
As before, let $\varepsilon_1,..., \varepsilon_{n+1}$ be the standard basis of $\RR^{n+1}$. We fix the following coordinates
\begin{equation}\label{model of A_n}
e_j=\varepsilon_{j+1} - \varepsilon_j, \quad 1\leq j \leq n; \quad \mathfrak{z}=\sum_{j=1}^n z_j e_j \in A_n\otimes \CC.     
\end{equation}
Up to sign, the nonzero vectors of minimal norm in $A_n'$ are
\begin{equation}\label{minimal vectors}
 u_j:=\varepsilon_j - \frac{1}{n+1}\sum_{i=1}^{n+1} \varepsilon_i, \quad 1\leq j\leq n+1.    
\end{equation}
These vectors have norm $\latt{u_j,u_j}=n/(n+1)$ and they lie in the same coset of $A_n'/A_n$.  

The hyperplane arrangement $\cH_L$ consists of two parts. The first part is the union of Heegner divisors associated to the minimal vectors in the dual of every irreducible component $A_{m_j}$ of the direct summand $L_0=\oplus_{j=1}^t A_{m_j}$:
\begin{equation}\label{eq:H_0}
\cH_{L,0} = \bigcup_{j=1}^t H\left( \frac{m_j}{2(m_j+1)}, v_j \right),  \quad \text{$v_j$ is a minimal norm vector in $A_{m_j}'$.} 
\end{equation}
By Remark \ref{rem:divisor} (1), the above Heegner divisor associated to $v_j$ is the $\widetilde{\Orth}^+(2U\oplus L)$-orbit of the hyperplane $\cD_{(0,0,v_j,1,0)}$. 
By convention, if $t=0$ and therefore $L_0 = \{0\}$, then $\cH_{L,0}$ is empty.

We introduce the second part of $\cH_L$. For a class $\gamma$ of $L'/L$, we define the minimal norm of $\gamma$ as
\begin{equation}\label{eq:delta_gamma}
\delta_\gamma  = \min\{ \latt{v,v} : v \in L + \gamma \}.     
\end{equation}
As before we define the minimal norm of $L$ as
\begin{equation}\label{eq:delta_L}
\delta_L = \max\{ \delta_\gamma: \gamma \in L'/L \}.    
\end{equation}

\begin{lemma}\label{lem:norm of L}
All lattices $L$ in the three families \eqref{eq:lattices} satisfy $\delta_L < 3$. 
\end{lemma}
\begin{proof} Direct computation.
\end{proof}

The second part of $\cH_L$ is defined as the following finite union of Heegner divisors:
\begin{equation}\label{eq:H_1}
\cH_{L,1} = \bigcup_{\substack{ \gamma \in L'/L\\ \delta_\gamma > 2}}  H\left( \frac{\delta_\gamma}{2}-1, \gamma \right).  
\end{equation}
By Remark \ref{rem:divisor} (1), the above Heegner divisor associated to $\gamma$ is the $\widetilde{\Orth}^+(2U\oplus L)$-orbit of the hyperplane $\cD_{(0,1,v,1,0)}$, where $v\in L'$ is a vector of minimal norm $\delta_\gamma$ in the class $\gamma + L$.
In particular, if $\delta_L \leq 2$ then $\cH_{L,1}$ is empty.

The $\widetilde{\Orth}^+(2U\oplus L)$-invariant hyperplane arrangement on $\cD(2U\oplus L)$ is then defined as 
\begin{equation}
\cH_L=\cH_{L,0}\cup \cH_{L,1}.    
\end{equation}
When $L$ is an irreducible root lattice (i.e. $L_0=\{0\}$) satisfying $\delta_L \leq 2$, then the arrangement $\cH_L$ is empty. The lattices satisfying this are precisely $A_m$ for $m\leq 7$, $D_m$ for $4\leq m\leq 8$, $E_6$ and $E_7$.

\begin{lemma}
For all lattices $L$ in the three families \eqref{eq:lattices}, the arrangement $\cH_L$ satisfies the Looijenga condition \ref{condition}.
\end{lemma}
\begin{proof}
We use induction on the rank of $L$; when $L$ has rank one, it is the $A_1$ root lattice and $\cH_L$ is empty.
In general, by Lemma \ref{lem:norm of L}, the arrangement $\cH_L$ is a finite union of Heegner divisors of the form $H(a,-)$ with $a< 1/2$. Every $n$-dimensional intersection of hyperplanes including at least one in $\cH_{L,0}$ is a symmetric domain attached to the lattice $2U \oplus M$ for another $M$ appearing in \eqref{eq:lattices}, and that every $(n-1)$-dimensional intersection of hyperplanes contained in it is contained in the arrangement $\cH_M$ associated to $M$, so in this case the claim follows by induction. 

Therefore, we only need to consider intersections of hyperplanes from $\cH_{L,1}$. Enumerating the minimal norm vectors in each $L'$ shows that the divisors $\cH_{L, 1}$ satisfy Corollary \ref{cor:intersection} and therefore the Looijenga condition \ref{condition}.
\end{proof}

\subsubsection{A more precise version of Theorem \ref{MTH2}}
\begin{theorem}\label{th:2precise}
Let $L=L_0\oplus L_1$ be one of the $147$ lattices in \eqref{eq:lattices}. Then the ring of modular forms for $\widetilde{\Orth}^+(2U\oplus L)$ with poles supported on $\cH_L$ is the polynomial algebra on $\rk(L)+3$ generators. The generators fall naturally into three groups:
\begin{enumerate}
    \item Eisenstein type: there are two generators of weights $4$ and $6$ whose zeroth Fourier--Jacobi coefficients are respectively the $\SL_2(\ZZ)$ Eisenstein series $E_4$ and $E_6$. 
    \item Abelian type: for each component $A_{m_j}$ of $L_0=\bigoplus_{j=1}^t A_{m_j}$, there are $m_j$ generators of weights
    $$
    m_j + 1, \quad m_j+1-i \quad \text{for $2\leq i\leq m_j$}. 
    $$
    \item Jacobi type: there are $\rk(L_1)+1$ generators associated to $L_1$ whose weights are given by
    $$
    k_i + t_i\left( 12 - \sum_{j=1}^t (m_j+1) \right), \quad 0\leq i \leq \rk(L_1)
    $$
    where the pairs $(k_i,t_i)$ are the weights and indices of the free generators of $W(L_1)$-invariant weak Jacobi forms described in \S \ref{subsec:Weyl}. 
\end{enumerate}
\end{theorem}

The grouping is based on the leading Fourier--Jacobi coefficient, which is either an Eisenstein series, an abelian function associated to $L_0$, or a Jacobi form associated to $L_1$. In more detail:

(1) The zeroth Fourier--Jacobi coefficient of a modular form that is holomorphic away from $\cH_{L, 1}$ is a (holomorphic) modular form in $\CC[E_4, E_6]$. The generators of Eisenstein type will turn out to have their poles supported on $\cH_{L, 1}$.

(2) The generators of abelian type are related to theta blocks associated to root systems of type $A$.  It follows from Wirthm\"uller's theorem \cite{Wir92} that the $\CC[E_4, E_6]$-module of weak Jacobi forms of index $1$ for $A_n$ is generated by forms of weights $0$, $-2$, $-3$, ..., $-(n+1)$. The minimal weight generator can be constructed as a theta block:
\begin{equation}\label{eq:theta-blockA_n}
\phi_{-(n+1),A_n,1}(\tau, \mathfrak{z})= \prod_{s=1}^{n+1} \frac{\vartheta(\tau, \latt{u_s, \mathfrak{z}})}{\eta(\tau)^3} = -\frac{\vartheta(\tau,z_1)\vartheta(\tau,z_n)}{\eta(\tau)^6} \prod_{s=1}^{n-1} \frac{\vartheta(\tau, z_s - z_{s+1})}{\eta(\tau)^3}, 
\end{equation}
where $u_s$ are minimal norm vectors of $A_n'$ defined in \eqref{minimal vectors} and we use the model of $A_n$ fixed in \eqref{model of A_n}. Let $\Delta=\eta^{24}$.  Then 
$$
\vartheta_{A_n}:=\Delta \phi_{-(n+1),A_n,1} 
$$
is a theta block of weight $11-n$ and $q$-order one associated to $A_n$, and it is the pullback of the theta block $\vartheta_{D_{n+1}}$. Corollary \ref{cor:theta} implies that its singular additive lift is also a Borcherds product:

\begin{lemma}\label{lem:theta A_n}
When $n\leq 10$, as meromorphic modular forms of weight $11-n$ for $\widetilde{\Orth}^+(2U\oplus A_n)$,
$$
\Grit(\vartheta_{A_n}) = \Borch\left( - \frac{\vartheta_{A_n}|T_{-}(2)}{\vartheta_{A_n}} \right).
$$
\end{lemma}

For each component $A_{m_j}$ of $L_0$, there will be $m_j$ generators of abelian type whose zeroth Fourier--Jacobi coefficients are the abelian functions
\begin{equation}\label{eq:generators of type (2)}
    \phi_{-i,A_{m_j},1} / \phi_{-(m_j+1),A_{m_j},1}, \quad i=0, 2,3, ..., m_j.
\end{equation}
These can be understood as meromorphic Jacobi forms of weight $(m_j + 1) - i$ and index $0$ associated to $L$ which are constant in the components of $L$ other than $A_{m_j}$.

(3) The generators of Jacobi type are motivated by the results of \cite{WW20a}. For the $14$ lattices $L$ with $\cH_L=\emptyset$, we proved in \cite{WW20a} that the ring of holomorphic modular forms for $\widetilde{\Orth}^+(2U\oplus L)$ is freely generated by $\rk(L)+3$ modular forms whose leading Fourier--Jacobi coefficients are respectively 
$$
E_4,\quad E_6,\quad \Delta^{s} f_{L,s} \cdot \xi^{s}, \quad \text{where} \; \xi=e^{2\pi i \omega},
$$
where $f_{L,s}$ runs through the generators of index $s$ of Wirthm\"uller's ring of $W(L)$-invariant weak Jacobi forms. In the present case, we expect to find $\rk(L_1)+1$ generators associated to the direct summand $L_1$ for which the leading terms in the Fourier--Jacobi expansions are
\begin{equation}\label{eq:generators of type (3)}
\left(\bigotimes_{j=1}^t \phi_{-(m_j+1),A_{m_j},1}\right)^{t_i} \otimes (\Delta^{t_i}\phi_{k_i,L_1,t_i})\cdot \xi^{t_i}  
\end{equation}
and, if $L_1 = D_m$, $m \ge 4$, the additional form beginning with
\begin{equation}\label{eq:generators Dn(1)}
\left(\bigotimes_{j=1}^t \phi_{-(m_j+1),A_{m_j},1}\right)\otimes (\Delta \psi_{-m,D_m,1})\cdot \xi.
\end{equation}
We refer to Notation \ref{notation} for the description of these generators of Weyl invariant weak Jacobi forms.
This coefficient is a Jacobi form of weight $k_i+t_i(12-\sum_{j=1}^t(m_j+1))$ and index $t_i$ associated to $L=L_0\oplus L_1$. 

Note that the expected generators are all of positive weight. Every generator of index $t_i$ and weight $k_i$ of the ring of $W(L_1)$-invariant weak Jacobi forms satisfies $k_i / t_i \geq -K_{L_1}$, where $K_{L_1}$ is $m+1$ for $L_1=A_m$, $m$ for $L_1=D_m$, and $5$ for $L_1=E_6, E_7$.  Thus the generators of Jacobi type have positive weight, because of the constraints $$(m+1) + \sum_{j=1}^t (m_j + 1) \le 11$$ for lattices of $A$-type and $$m + \sum_{j=1}^t (m_j + 1) \le 11$$ for lattices of $AD$-type.

\begin{remark}
If a lattice $L$ in the family of $A$-type has $$K:=m+1+\sum_{j=1}^t (m_j+1)\geq 12,$$ then the hyperplane arrangement $\cH_L$ does \emph{not} satisfy the Looijenga condition. Indeed, if $\theta_L$ is the theta block of $q$-order one defined as $$\theta_L := \Delta \cdot \left( \bigotimes_{j=1}^t \phi_{-(m_j + 1), A_{m_j}, 1} \right) \otimes \phi_{-(m+1), A_m, 1},$$ then the Borcherds lift of $-2(\theta_L | T_{-}(2)) / \theta_L$ is a meromorphic modular form of non-positive weight $24-2K$ which is holomorphic away from $\cH_{L, 1}$, violating Koecher's principle.
Therefore, Theorem \ref{th:2precise} cannot be extended to the lattices satisfying $K\geq 12$. A similar result holds for the family of $AD$-type lattices. 
\end{remark}

\subsubsection{The sketch of the proof}
To prove Theorem \ref{th:2precise}, we apply the modular Jacobian criterion of Theorem \ref{th:Jacobiancriterion}. We sketch the main steps of the proof below: 

\vspace{3mm}

\begin{itemize}
    \item[(I)] We construct a Borcherds product $\Phi_L$ for $\widetilde{\Orth}^+(2U\oplus L)$ satisfying the following conditions:
    \begin{itemize}
        \item[(i)] The weight of $\Phi_L$ equals the sum of the weights of the $\rk(L)+3$ generators prescribed in Theorem \ref{th:2precise} plus $\rk(L)+2$. 
        \item[(ii)] The function $\Phi_L$ vanishes with multiplicity one on hyperplanes associated to vectors of norm $2$ in $2U\oplus L$ which are not contained in the arrangement $\cH_L$.
        \item[(iii)] All other zeros and poles of $\Phi_L$ are contained in $\cH_L$. 
    \end{itemize}
    \item[(II)] We construct the $\rk(L)+3$ generators which have weights given in Theorem \ref{th:2precise} and the leading Fourier--Jacobi coefficients described above. By the discussion above, we conclude from the algebraic independence of generators of Jacobi forms that the leading Fourier--Jacobi coefficients of these generators of orthogonal type are algebraically independent over $\CC$. Therefore, these generators of orthogonal type are algebraically independent, which implies that their Jacobian $J_L$ is not identically zero. 
    \item[(III)] The quotient $J_L / \Phi_L$ defines a modular form of weight $0$ which is holomorphic away from $\cH_L$, and is therefore constant by Lemma \ref{lem:Koecher}. 
    \item[(IV)] Theorem \ref{th:2precise} follows by applying Theorem \ref{th:Jacobiancriterion}. 
\end{itemize}

\vspace{3mm}

Step (I) is Theorem \ref{th:existence}. For Step (II), we construct the generators of Eisenstein type in Lemma \ref{lem: type (1)}, the generators of abelian type in Lemma \ref{lem: type (2) AD} and Lemma \ref{lem: type (2) AE}, and the generators of Jacobi type in \S\ref{subsubsec: generators of type (3)} for the families of $A$-type and $AD$-type. Finally, we find the generators of Jacobi type for the family of $AE$-type in \S \ref{sec:non-free}. 
Steps (III) and (IV) then follow immediately.

\subsection{Preimage of the modular Jacobian under the Borcherds lift} For each lattice $L$ in \eqref{eq:lattices}, we will prove the existence of a Borcherds product $\Phi_L$ which will turn out to be (up to a scalar) the Jacobian of any set of generators of the algebra of meromorphic modular forms.

\begin{theorem}\label{th:existence}
For each lattice $L$ in the families \eqref{eq:lattices}, there exists a Borcherds product $\Phi_L$ satisfying conditions $(i)$, $(ii)$, $(iii)$ of Step $(I)$. More precisely, $\Phi_L$ is the Borcherds multiplicative lift of a nearly holomorphic Jacobi form of weight $0$ and index $1$ associated to $L$ whose Fourier expansion has the form
\begin{equation}\label{eq:Jacobian q^0}
    \phi_L(\tau, \mathfrak{z}) = q^{-1} + 2k + \sum_{\substack{ r\in L\\ \latt{r,r}=2}} \zeta^r + \sum_{j=1}^t c_{m_j} \sum_{u\in \mathcal{U}_j}(\zeta^u + \zeta^{-u}) + O(q),
\end{equation}
where $\mathcal{U}_j$ is the set of minimal norm vectors in the dual of the component $A_{m_j}$ of $L_0$ defined in \eqref{minimal vectors}. The multiplicities $c_{m_j}$ are determined by the formulae
\begin{equation}
    c_1 = 2(h-2), \quad c_{n} = h - (n+1) \quad \text{if $n>1$},
\end{equation}
where $h$ is the Coxeter number of the root lattice $L_1$. 
The weight $k$ of $\Phi_L$ is determined by
\begin{equation}
    k = 12 (h+1) - \frac{1}{2} \rk(L_1)h - \sum_{j=1}^t \left(h-1 - \frac{1}{2}m_j \right)(m_j+1). 
\end{equation}
\end{theorem}

The proof of the theorem follows essentially from the following Lemma.

\begin{lemma}\label{lem:Jacobi}
Let $L$ be an even positive definite lattice with bilinear form $\latt{-,-}$.
For every nonzero $\gamma \in L'/L$, we define the basic orbit associated to $\gamma$ as 
$$
\orb(\gamma)= \sum_{\substack{v\in L+\gamma \\ \latt{v,v}=\delta_\gamma }} (\zeta^v + \zeta^{-v}), \quad \zeta^v=e^{2\pi i\latt{v,\mathfrak{z}}}, \quad \mathfrak{z} \in L\otimes\CC,
$$
where as before $\delta_\gamma = \min\{ \latt{v,v} : v\in L + \gamma \}$. The number of distinct basic orbits is denoted $N$.
\begin{enumerate}
    \item The space $J_{2*,L,1}^{\w}$ of weak Jacobi forms of even weight and index $1$ associated to $L$ is a free module over $\CC[E_4, E_6]$ on $N+1$ generators. The $q^0$-term of any weak Jacobi form in $J_{2*,L,1}^{\w}$ is a $\CC$-linear combination of the basic orbits and the constant term, and the $q^0$-terms of any basis of the $\CC[E_4, E_6]$-module are linearly independent over $\CC$.
    \item Let $L=\oplus_{i=1}^n L_i$. For every $L_i$, assume that
    \begin{itemize}
        \item[(i)]  the free module $J_{*,L_i,1}^{\w}=\oplus_{k\in\ZZ}J_{k,L_i,1}^{\w}$ is generated by forms of non-positive weight;
        \item[(ii)] $J_{*,L_i,1}^{\w}$ has no generator of weight $-1$;
        \item[(iii)]$J_{*,L_i,1}^{\w}$ has exactly one generator each of weight $-2$ and weight $0$.
    \end{itemize}
    For $1\leq i\leq n-1$, choose any basic orbit $\orb(\gamma_i)$ for $L_i$. Then there exists a nearly holomorphic Jacobi form of weight $0$ and index $1$ associated to $L$ with rational Fourier expansion
    $$
    \phi_L = q^{-1} + \sum_{\substack{r\in L \\ \latt{r,r}=2}} \zeta^r + \sum_{j=1}^{n-1} c_j \orb(\gamma_j) + c_0 + O(q), 
    $$
    and where all $c_j$ are rational numbers.
\end{enumerate}
\end{lemma}
For general $L$, $\phi_L$ will not be unique.
\begin{proof}
(1) This is standard result. (See also \cite[Proposition 2.2]{WW21c}.) Briefly, following the argument of \cite[Theorem 8.4]{EZ85}, $J_{2*,L,1}^{\w}$ can be shown to be a free module over $M_*(\SL_2(\ZZ))$ using only the fact that the weight of weak Jacobi forms is bounded from below. If we view Jacobi forms as modular forms for $\rho_L$, then Borcherds' obstruction principle \cite{Bor99} shows that there exist weak Jacobi forms of sufficiently large weight whose $q^0$-term is a constant or any single basic orbit. This implies that there are $N+1$ generators. The $q^0$-terms of any basis are linearly independent, as otherwise some $\CC[E_4, E_6]$-linear combination of the generators would be a weak Jacobi form whose $q^0$-term is zero and therefore is a product of $\Delta$ and a $\CC[E_4,E_6]$-linear combination of generators of lower weight, leading to a contradiction.

(2) Define two rational vector spaces:
\begin{align*}
    V=& \QQ \oplus \bigoplus_{\substack{\gamma \in L' / L \\ \gamma \neq 0}} \QQ \cdot \orb(\gamma); \\
    V_0=& \{ \text{$q^0$-terms of weak Jacobi forms of weight $0$ and index $1$ for $L$ with rational coefficients} \}.
\end{align*}
The space $V_0$ is naturally a subspace of $V$. By \cite{McG03}, all spaces of vector-valued modular forms for the representations $\rho_L$ have bases with rational Fourier coefficients, so the generators of $J^{\w}_{*, L_i, 1}$ can all be chosen to be rational. By assumption (i), the free module $J_{2*,L,1}^{\w}$ is generated by forms of non-positive weight, and therefore the codimension of $V_0$ in $V$ is the number of generators of weight $-2$ of $J_{2*,L,1}^{\w}$ which is in fact $n$ by assumptions (ii) and (iii). We claim that $\{1, \orb(\gamma_i), 1\leq i\leq n-1 \}$ is a basis of the quotient space $V/V_0$; in other words, that there is no nonzero weak Jacobi form of weight $0$ and index $1$ for $L$ whose $q^0$-term lies in the span of this basis. This follows from the vector system relation \eqref{eq:vectorsystem} satisfied by the $q^0$-term of any weak Jacobi form of weight $0$:
$$
\sum_{\ell \in L'} f(0,\ell) \latt{\ell, \mathfrak{z}}^2 = 2C \latt{\mathfrak{z},\mathfrak{z}}, \quad \mathfrak{z}\in L\otimes\CC,
$$
which is not satisfied by any combination of $1$ and $\orb(\gamma_i)$ (since these $\gamma_i$ does not span the whole space $L\otimes\RR$). Therefore, for any basic orbit $\orb(\gamma)$ related to $L$, there exists a weak Jacobi form of weight $0$ and index $1$ for $L$ with Fourier expansion 
$$
\phi_{L,\gamma} = \orb(\gamma) + \sum_{i=1}^{n-1} a_i \orb(\gamma_i) + a_0,
$$
where $a_i \in \QQ$ for $0\leq i\leq n-1$. Since $J_{2*,L,1}^{\w}$ is generated in non-positive weights, part (1) implies that there exists a weak Jacobi form of weight $12$ whose $q^0$-term is $1$ and whose other Fourier coefficients are rational. We divide it by $\Delta$ and obtain a nearly holomorphic Jacobi form of weight $0$ and index $1$ for $L$ whose Fourier expansion begins
$$
q^{-1}+\sum_{r\in L, \latt{r,r}=2}\zeta^r + C(\zeta) + O(q), \; \; C(\zeta) \in V.
$$
By modifying it with the above $\phi_{L,\gamma}$, we obtain the desired Jacobi form $\phi_L$. 
\end{proof}

\begin{proof}[Proof of Theorem \ref{th:existence}]
The root lattice $L$ satisfies all assumptions in the above lemma by \S \ref{subsec:Weyl}, so we obtain the existence of a nearly holomorphic Jacobi form $\phi_L$ of weight $0$ and index $1$ for $L$ whose Fourier expansion has the form \eqref{eq:Jacobian q^0}. It remains to determine $k$ and $c_{m_j}$. The multiplicities $c_{m_j}$ can be computed by applying the identity \eqref{eq:vectorsystem} to every irreducible component of $L_0$,  and $k$ is determined by the identity \eqref{eq:q^0-term}.
\end{proof}

\begin{remark}
For the lattices in \eqref{eq:lattices}, the Jacobi form $\phi_L$ from \eqref{eq:Jacobian q^0} is unique because $J_{-12,L,1}^{\w} = \{0\}$.  Lemma \ref{lem:norm of L} and Remark \ref{rem:divisor} (4) show that the representative singular Fourier coefficients of $\phi_L$ only appear in its $q^{-1}$-term, $q^0$-term and $q^1$-term. The representative singular Fourier coefficients in $q^{-1}$-term correspond to mirrors of reflections in $\widetilde{\Orth}^+(2U\oplus L)$. The representative singular Fourier coefficients in $q^{0}$-term correspond to the first part $\cH_{L,0}$ of the arrangement $\cH_L$. The representative singular Fourier coefficients in $q^{1}$-term correspond to the second part $\cH_{L,1}$ of  $\cH_L$.

In Theorem \ref{th:existence} we only showed that $\phi_L$ has rational Fourier coefficients. This is enough for our purposes: suppose $N \in \mathbb{N}$ is such that $N \phi_L$ has integral Fourier coefficients and is therefore a valid input in Borcherds' lift. We will simply replace step (III) of the proof of Theorem \ref{th:2precise} by showing that $J_L^N / \Borch(N\phi_L)$ is constant. From the divisor of $J_L^N$ we see \emph{a posteriori} that the singular coefficients of $\phi_L$ (in particular those that appear in the $q^1$-term) were in fact integers.
\end{remark}

\begin{remark}
Let us verify that the weight $k$ in Theorem \ref{th:existence} equals the weight of the Jacobian of generators predicted by Theorem \ref{th:2precise}. Wirthm\"uller \cite{Wir92} determined the weights and indices of generators of $W(R)$-invariant weak Jacobi forms in terms of invariants of $R$ (we refer to \cite[Theorem 2.2]{Wan21b} for a clear description). In particular, the sums of the weights $k_i$ and indices $t_i$ satisfy
$$
\sum k_i = -\left( 1 + \frac{h}{2} \right)\rk(R), \quad \sum t_i = h,
$$
where $h$ is the Coxeter number of $R$. Using these identities, it follows that $k - (\mathrm{rk}(R) + 2)$ is the sum of the weights of the generators in Theorem \ref{th:2precise}.
\end{remark}

\subsection{The construction of generators}\label{subsec:generators}
In this subsection we construct the generators required in Theorem \ref{th:2precise}.

\subsubsection{Generators of Eisenstein type}

\begin{lemma}\label{lem:E4E6}
Assume that the free $M_*(\SL_2(\ZZ))$-module $J_{2*,L,1}^{\w}$ is generated by forms of non-positive weight. Then there exist weak Jacobi forms with the following Fourier expansions:
\begin{align*}
    E_{4,L} &= 1 +O(q) \in J_{4,L,1}^{\w},\\
    E_{6,L} &= 1 +O(q) \in J_{6,L,1}^{\w}. 
\end{align*}
Their singular additive lifts define meromorphic modular forms of weight $4$ and $6$ for $\widetilde{\Orth}^+(2U\oplus L)$ with Fourier--Jacobi expansions
\begin{align*}
 \cE_{4,L} &:= 240 \Grit(E_{4,L}) = E_4 + O(\xi),\\
 \cE_{6,L} &:= -504 \Grit(E_{6,L}) = E_6 + O(\xi),
\end{align*}
 where $\xi=e^{2\pi i \omega}$ as before. 
\end{lemma}
\begin{proof}
By Lemma \ref{lem:Jacobi}, the $q^0$-terms of generators of $J_{2*,L,1}^{\w}$ are linearly independent, and the number of generators equals the number of basic orbits plus one. Therefore the forms $E_{4, L}$ and $E_{6, L}$ can be constructed as suitable $\CC[E_4,E_6]$-linear combinations of the generators. The second part of the lemma follows from Theorem \ref{th:additive}.  
\end{proof}

The above form $E_{4,L}$ (resp. $E_{6,L}$) is unique if and only if $J_{-8,L,1}^{\w}$ (resp. $J_{-6,L,1}^{\w}$) is trivial. 

The following gives the construction of generators of Eisenstein type. 

\begin{lemma}\label{lem: type (1)}
For every lattice $L$ in \eqref{eq:lattices}, the possible poles of $\cE_{4,L}$ and $\cE_{6,L}$ are contained in $\cH_{L,1}$.   
\end{lemma}
\begin{proof}
All singular Fourier coefficients of $E_{4,L}$ and $E_{6,L}$ must appear in their $q^1$-terms, and these coefficients correspond to divisors in $\cH_{L,1}$ under the singular additive lift.
\end{proof}

\subsubsection{Generators of abelian type}
The generators of abelian type will be constructed as quotients of singular additive lifts. For $n\geq 4$, the minimal weight generator of index $1$ for $J_{*,D_n,*}^{\w,W(D_n)}$ which is anti-invariant under changing an odd number of signs (e.g. the map $z_1 \mapsto -z_1$) can be constructed as the theta block
\begin{equation}
    \psi_{-n, D_n,1}(\tau, \mathfrak{z}) = \frac{\vartheta_{D_n}(\tau, \mathfrak{z})}{\eta(\tau)^{24}} = \prod_{j=1}^n \frac{\vartheta(\tau,z_j)}{\eta(\tau)^3} \in J_{-n,D_n,1}^{\w, W(D_n)}.
\end{equation}
Here we have fixed the model of $D_n$ in \eqref{model of D_n} and $\vartheta_{D_n}$ is defined in \eqref{eq:theta block D_n}.
We recall the basic Jacobi forms fixed in Notation \ref{notation}. The generators of $W(A_n)$-invariant weak Jacobi forms are labelled $\phi_{-k,A_n,1}$ for $k=0$ or $2\leq k \leq n+1$. Note that $\phi_{-(n+1),A_n,1}$ is the theta block \eqref{eq:theta-blockA_n}. 

\begin{lemma}\label{lem: type (2) AD}
Let $L=L_0\oplus L_1$ be a lattice in the family of $A$-type or $AD$-type with $L_0\neq \{0\}$. For any component $A_{m_s}$ of $L_0=\oplus_{j=1}^t A_{m_j}$, the associated $m_s$ generators can be constructed by
\begin{equation}
\begin{aligned}
    &\frac{\Grit\left[\Delta \left(\bigotimes_{j=1}^{s-1}\phi_{-(m_j+1),A_{m_j},1}\right)\otimes \phi_{-k,A_{m_s},1} \otimes \left(\bigotimes_{j=s+1}^{t}\phi_{-(m_j+1),A_{m_j},1}\right)\otimes f_{L_1} \right]}{\Grit\left[\Delta\left(\bigotimes_{j=1}^{t}\phi_{-(m_j+1),A_{m_j},1}\right)\otimes f_{L_1} \right]} \\
    =& \frac{\phi_{-k,A_{m_s},1}}{\phi_{-(m_s+1),A_{m_s},1}} + O(\xi), \quad \text{for $k=0$ or $2\leq k \leq m_s$}
\end{aligned}
\end{equation}
where $f_{L_1}=\phi_{-(m+1),A_m,1}$ if $L_1=A_m$, or $f_{L_1}=\psi_{-m, D_m,1}$ if $L_1=D_m$. 
\end{lemma}
\begin{proof}
We write the above quotient as $\Grit(f)/\Grit(g)$ for convenience. By definition, $g$ is a theta block of positive weight and $q$-order one. By Corollary \ref{cor:theta}, $\Grit(g)$ is a Borcherds product. Moreover, on the complement $\cD(2U\oplus L) - \cH_{L} $ the additive lift $\Grit(g)$ vanishes precisely with multiplicity one along the Heegner divisor corresponding to the theta block $f_{L_1}$, which is either $H(m/2(m+1),u)$ for a minimal norm vector $u$ in the dual of $L_1=A_m$, or $H(1/2,\varepsilon_1)$ for the minimal vector $\varepsilon_1$ in the dual of $L_1=D_m$. By Remark \ref{rem:divisor} (5), $\Grit(f)$ also vanishes on this Heegner divisor. In addition, all poles of $\Grit(f)$ and $\Grit(g)$ are contained in $\cH_{L,1}$. Therefore, the above quotient is holomorphic away from $\cH_{L}$. Its leading Fourier--Jacobi coefficient is the quotient of the leading Fourier--Jacobi coefficients of $\Grit(f)$ and $\Grit(g)$.  
\end{proof}

The generators of abelian type for the three $AE$-lattices with $L_0 \neq \{0\}$ must be constructed by a different argument. In this case, $\delta_L \leq 2$, so $\cH_{L, 1}$ is empty.

\begin{lemma}\label{lem: type (2) AE}
For the three lattices of $AE$-type with $L_0\neq \{0\}$, the generators of abelian type exist:
\begin{align*}
G_{2, A_1\oplus E_6} &= \Grit(\phi_{-2,A_1,1}\otimes E_{4,E_6}) = c_1 \frac{\phi_{0,A_1,1}}{\phi_{-2,A_1,1}} + O(\xi),\\
G_{2, A_1\oplus E_7} &= \Grit(\phi_{-2,A_1,1}\otimes E_{4,E_7}) = c_2 \frac{\phi_{0,A_1,1}}{\phi_{-2,A_1,1}} + O(\xi),\\
G_{1, A_2\oplus E_6} &= \Grit(\phi_{-3,A_2,1}\otimes E_{4,E_6}) = c_3 \frac{\phi_{-2,A_2,1}}{\phi_{-3,A_2,1}} + O(\xi),\\
G_{3, A_2\oplus E_6} &= \Grit(\phi_{-3,A_2,1}\otimes E_{6,E_6}) - G_{1, A_2\oplus  E_6}^3 = c_4 \frac{\phi_{0, A_2,1}}{\phi_{-3,A_2,1}} + O(\xi),
\end{align*}
where $c_i$ are nonzero constants, and where $E_{4,E_6}$, $E_{6,E_6}$ and $E_{4,E_7}$ are defined in Lemma \ref{lem:E4E6}.
\end{lemma}
\begin{proof}
The above meromorphic modular forms are well defined and their poles are supported on $\cH_{L,0}$. We only need to show that they have the claimed zeroth Fourier--Jacobi coefficients. This follows from Theorem \ref{th:additive}, together with the identity (\cite{EZ85}, Theorem 3.6) $$\frac{\phi_{0, A_1, 1}}{\phi_{-2, A_1, 1}} = -\frac{3}{\pi^2} \wp$$ in the first two cases, and the identity (\cite{WW21c}, Lemma 4.1) \begin{align*} \frac{\phi_{-2, A_2, 1}(\tau, z_1, z_2)}{\phi_{-3, A_2, 1}(\tau, z_1, z_2)} &= \frac{1}{2\pi i} \left( -\frac{\vartheta'(\tau, z_1)}{\vartheta(\tau, z_1)} + \frac{\vartheta'(\tau, z_1-z_2)}{\vartheta(\tau, z_1-z_2)} + \frac{\vartheta'(\tau, z_2)}{\vartheta(\tau, z_2)} \right) \\ &= \frac{1}{2\pi i} \left( -\zeta(\tau, z_1) + \zeta(\tau, z_1-z_2) + \zeta(\tau, z_2) \right) \end{align*} in the third case. As for $G_{3, A_2 \oplus E_6}$, we can see that if the (unique) allowed pole in $\cH_{L, 0}$ is cut out locally by $z = 0$ then $G_{1, A_2 \oplus E_6}$ and $\Grit(\phi_{-3,A_2,1}\otimes E_{6,E_6})$ have Taylor expansions \begin{align*} G_{1, A_2 \oplus E_6}(\tau, \mathfrak{z}, \omega) &= (2\pi i z)^{-1} + O(z) \\ \Grit(\phi_{-3,A_2,1}\otimes E_{6,E_6})(\tau, \mathfrak{z}, \omega) &= (2\pi i z)^{-3} + O(z) \end{align*} and therefore that $G_{3, A_2 \oplus E_6}$ has at worst a simple pole. Its weight is also less than the singular weight for $2U \oplus A_2 \oplus E_6$, so to see that it indeed has a simple pole on its only allowed singularity it is enough to check that it does not vanish identically. This is clear from the leading Fourier--Jacobi coefficient.
\end{proof}

\subsubsection{Generators of Jacobi type}\label{subsubsec: generators of type (3)}

Finally, we construct the generators of Jacobi type. The construction is easiest for the family of $A$-type lattices:

\begin{lemma}
Let $L=L_0\oplus A_m$ be a lattice in the family of $A$-type. The $m+1$ generators of Jacobi type can be constructed as singular additive lifts:
\begin{equation}
    \Grit\left(\Delta  \left( \otimes_{j=1}^t \phi_{-(m_j+1),A_{m_j},1} \right) \otimes \phi_{-k,A_m,1} \right), \quad \text{for $k=0$ or $2\leq k \leq m+1$}.
\end{equation}
\end{lemma}
\begin{proof}
It follows immediately from Theorem \ref{th:additive} that these lifts are holomorphic away from $\cH_{L, 1}$ and have the necessary leading Fourier--Jacobi coefficients.
\end{proof}

Now we consider the family of $AD$-type lattices. For every index-one generator $f_{D_m}$ of the ring of $W(D_m)$-invariant weak Jacobi forms, the associated meromorphic generator is again a singular theta lift:
\begin{equation}
    \Grit\left(\Delta  \left( \otimes_{j=1}^t \phi_{-(m_j+1),A_{m_j},1} \right) \otimes f_{D_m}  \right).
\end{equation}
Unlike the $L_1 = A_m$ case, the ring $J_{*,D_m,*}^{\w,W(D_m)}$ has generators of index $2$. The corresponding meromorphic generators can be constructed as polynomials in singular theta lifts, but to do this directly in all but the simplest cases apparently requires intricate identities among elliptic functions and Jacobi forms. We will construct only the lowest weight generators explicitly and prove the existence of the other generators indirectly from our work on modular forms for the lattice $2U\oplus D_8$ \cite{WW20a}.

\begin{lemma}\label{lem:L1}
Let $L=L_0\oplus D_m$ be a lattice in the family of $AD$-type. There exist weak Jacobi forms of weights $2$ and $4$ and lattice index $L$ with the following $q^0$-terms:
\begin{align}
    f_{2,L} &= \sum_{j=1}^m (e^{2\pi i z_j} + e^{-2\pi i z_j}) -2m + O(q) \in J_{2,L,1}^{\w},\\
    f_{4,L} &= \sum_{j=1}^m (e^{2\pi i z_j} + e^{-2\pi i z_j}) + O(q) \in J_{4,L,1}^{\w},
\end{align}
where $(z_1,...,z_m)$ is the coordinates of $D_m\otimes \CC$ fixed in \eqref{model of D_n}.  
\end{lemma}
\begin{proof}
As in Notation \ref{notation}, let $\phi_{0,D_m,1}$, $\phi_{-2,D_m,1}$ and $\phi_{-4,D_m,1}$ be the index-one generators of the ring of $W(C_m)$-invariant Jacobi forms. Let $E_{4,L_0}$ and $E_{6,L_0}$ be the weak Jacobi forms defined in Lemma \ref{lem:E4E6} (which reduce to the usual Eisenstein series $E_4$ and $E_6$ when $L_0 = \{0\}$).  Note that the $q^0$-terms of $\phi_{k,D_m,1}$ are linearly independent and involve only the constant term and the two basic orbits
$$
\sum_{j=1}^m (e^{2\pi i z_j} + e^{-2\pi i z_j}) \quad \text{and} \quad \sum_{v \in \{\pm 1/2\}^m}  \prod_{j=1}^m e^{2\pi i v_j z_j},
$$
Therefore, there are $\CC$-linear combinations $f_{2, L}$ of $E_{4,L_0}\phi_{-2,D_m,1}$ and $E_{6,L_0}\phi_{-4,D_m,1}$ and $f_{4, L}$ of $E_{4,L_0}\phi_{0,D_m,1}$, $E_{6,L_0}\phi_{-2,D_m,1}$ and $E_4E_{4,L_0}\phi_{-4,D_m,1}$ with the desired $q^0$-terms.
\end{proof}

\begin{lemma}\label{lem:L2}
Let $L=L_0\oplus D_m$ be a lattice in the family of $AD$-type. There exist meromorphic modular forms of weights $2$ and $4$ for $\widetilde{\Orth}^+(2U\oplus L)$, with double poles on $H(1/2,\varepsilon_m)$ and whose other poles are contained in $\cH_{L,1}$, and whose Fourier--Jacobi expansions are
\begin{align}
    \Psi_{2,L}(\tau, \mathfrak{z}, \omega) &=-\frac{3}{\pi^2} \sum_{j=1}^m \wp(\tau,z_j) + O(\xi),\\
    \Psi_{4,L}(\tau, \mathfrak{z}, \omega) &= \frac{9}{\pi^4}\sum_{1\leq j_1<j_2\leq m} \wp(\tau,z_{j_1})\wp(\tau,z_{j_2})+ O(\xi).
\end{align}

\end{lemma}
\begin{proof}
The singular additive lifts of $f_{2,L}$ and $f_{4,L}$ have Fourier--Jacobi expansions beginning
\begin{align*}
    F_{2,L} &= \frac{1}{(2\pi i)^2}\sum_{j=1}^m \wp(\tau,z_j) + O(\xi),\\
    F_{4,L} &= \frac{1}{(2\pi i)^4} \sum_{j=1}^m \wp''(\tau,z_j) + O(\xi),
\end{align*}
and their poles are supported on $\cH_{L,1}\cup H(1/2,\varepsilon_m)$ with multiplicity $2$ and $4$ respectively. Since $\wp(\tau, z) = z^{-2} + O(z^2)$, the Taylor expansions of $F_{2, L}$ and $F_{4, L}$ about the divisor $z_m = 0$ which represents $H(1/2, \varepsilon_m)$ begin
\begin{align*} F_{2, L} &= (2\pi i z_m)^{-2} + f_{2, 2} + O(z_m^2), \\
F_{4, L} &= 6 \cdot (2\pi i z_m)^{-4} + f_{4, 4} + O(z_m^2), 
\end{align*}
for some functions $f_{2, 2}$ and $f_{4, 4}$ which are holomorphic near $z_m = 0$.
Clearly we can take $\Psi_{2, L} = 12 F_{2, L}$. The construction \begin{align*}
 \Psi_{4, L} = 72 F_{2,L}^2 - 12 F_{4,L} - 120 \Grit(E_{4,L}) = -\frac{36}{\pi^2}f_{2,2} z_m^{-2} + O(z_m^0)\end{align*} therefore has at most a double pole on $z_m = 0$. Using the Weierstrass differential equation $$\wp''(\tau, z) = 6 \wp(\tau, z)^2 - \frac{2}{3}\pi^4 E_4(\tau)$$ we obtain the first Fourier--Jacobi coefficient of $\Psi_{4, L}$.
\end{proof}

\begin{lemma}\label{lem:L3}
Let $L=L_0\oplus D_m$ be a lattice of $AD$-type and let $\kappa = 2(12-\sum_{j=1}^t (m_j+1))$. The generators of Jacobi type of weights $\kappa - 2m + 2$ and $\kappa - 2m + 4$ corresponding to $\phi_{-2(m-1),D_m,2}$ and $\phi_{-2(m-2),D_m,2}$ can be constructed as
\begin{align*}
    \Phi_{\kappa-2m+2,L} &:= B_{\kappa/2 -m, L}^2 \Psi_{2,L} = \Delta^2 (\otimes_{j=1}^t \phi_{-(m_j+1),A_{m_j},1}^2)\otimes \phi_{-2(m-1),D_m,2}\cdot\xi^2 + O(\xi^3),\\
    \Phi_{\kappa-2m+4,L} &:= B_{\kappa/2 -m, L}^2 \Psi_{4,L} = \Delta^2 (\otimes_{j=1}^t \phi_{-(m_j+1),A_{m_j},1}^2)\otimes \phi_{-2(m-2),D_m,2}\cdot\xi^2 + O(\xi^3),
\end{align*}
where
$$
B_{\kappa/2-m,L} = \Grit\left[\Delta \cdot \left(\bigotimes_{j=1}^t \phi_{-(m_j+1),A_{m_j},1}\right)\otimes \psi_{-m,D_m,1}\right]. 
$$
\end{lemma}
\begin{proof}
Recall from \cite{AG20} that the index two generators $\phi_{-2(m-1),D_m,2}$ and $\phi_{-2(m-2),D_m,2}$ of the ring of $W(D_m)$-invariant Jacobi forms can be constructed using the pullback from $mA_1$ to $D_m(2)$ as the symmetric sums
\begin{align*}
    \phi_{-2(m-1),D_m,2} = \sum_{\text{sym}} \phi_{-2,A_1,1}^{\otimes^{m-1}} \otimes \phi_{0,A_1,1},\\
    \phi_{-2(m-2),D_m,2} = \sum_{\text{sym}} \phi_{-2,A_1,1}^{\otimes^{m-2}} \otimes \phi_{0,A_1,1}^{\otimes^2},
\end{align*}
where $\sum_{\text{sym}} f(z_1,...,z_n) = \frac{1}{n!} \sum_{\sigma \in S_n} f(z_{\sigma(1)},...,z_{\sigma(n)}).$ Note also that the input form in the additive lift to $B_{\kappa/2-m,L}$ is a theta block of $q$-order one, and that $B_{\kappa/2-m,L}$ has simple zeros along $H(1/2,\varepsilon_m)$, cancelling the poles of $\Psi_{2, L}$ and $\Psi_{4, L}$ there. It follows that the forms $\Phi_{-,L}$ are holomorphic away from $\cH_{L,1}$. Their leading Fourier--Jacobi coefficients can be determined using
\[
\psi_{-m,D_m,1}^2(\tau,\mathfrak{z}) = \prod_{j=1}^m \phi_{-2,A_1,1}(\tau,z_j) \quad \text{and} \quad -\frac{3}{\pi^2}\wp(\tau,z) = \frac{\phi_{0,A_1,1}(\tau,z)}{\phi_{-2,A_1,1}(\tau,z)}. 
\qedhere \]
\end{proof}

To construct the other generators of Jacobi type, we need the following result for $2U\oplus D_8$:

\begin{lemma}[see Theorem 5.2 and Corollary 4.3 of \cite{WW20a}]\label{lem:D8}
The ring of holomorphic modular forms for $\Orth^+(2U\oplus D_8)$ is freely generated by $11$ additive lifts of Jacobi Eisenstein series. Moreover, for each index two generator $\phi_{k,D_8,2}$ of the ring of $W(C_8)$-invariant Jacobi forms, there exists a modular form for $\Orth^+(2U\oplus D_8)$  of weight $24+k$ with Fourier--Jacobi expansion
$$
G_{24+k,D_8} = \Delta^2 \phi_{k, D_8,2}\cdot \xi^2 + O(\xi^3), \quad k=-6, -8, ..., -12, -14, -16,
$$
\end{lemma}

The $11$ Jacobi Eisenstein series can all be expressed as $\CC[E_4, E_6]$-linear combinations of the index-one generators $\phi_{0,D_8,1}$, $\phi_{-2,D_8,1}$ and $\phi_{-4,D_8,1}$ of the ring of $W(C_8)$-invariant Jacobi forms.

We first consider the cases $L = D_m$ with $9 \leq m \leq 11$. (In particular, $L_0 = \{0\}$.)

\begin{lemma}\label{lem:D11}
For $m=9,10,11$, there exist modular forms, holomorphic away from $\cH_{L,1}$, with Fourier--Jacobi expansion beginning
$$
G_{24+k, D_m} = \Delta^2 \phi_{k, D_m,2}\cdot \xi^2 + O(\xi^3), \quad k=-6, -8, ...,-2(m-2), -2(m-1). 
$$
\end{lemma}
\begin{proof}
It is enough to prove the lemma for $L=D_{11}$, as the forms $G_{24+k, D_{10}}$ and $G_{24+k, D_{9}}$ can be constructed by restricting $G_{24+k, D_{11}}$. By Lemma \ref{lem:L3} we can construct $G_{4,D_{11}}=\Phi_{4,D_{11}}$ and $G_{6,D_{11}}=\Phi_{6,D_{11}}$. 
For $k\geq -16$, we can write the form $G_{24+k,D_8}$ in Lemma \ref{lem:D8} uniquely as a polynomial in the additive lifts of $\CC[E_4, E_6]$-linear combinations of $\phi_{0, D_8, 1}$, $\phi_{-2, D_8, 1}$, $\phi_{-4, D_8, 1}$: $$G_{24 + k, D_8} = P \left( \mathrm{Grit}( E_4^{a_0} E_6^{b_0} \phi_{0, D_8, 1} + E_4^{a_2} E_6^{b_2} \phi_{-2, D_8, 1} + E_4^{a_4} E_6^{b_4} \phi_{-4, D_8, 1}), \; 4a_i + 6b_i -i = 24+k \right).$$ We lift this to a form $\widehat{G}_{24+k, D_{11}}$ by replacing all instances of $\phi_{-k, D_8, 1}$ with $\phi_{-k, D_{11}, 1}$:
$$\widehat{G}_{24 + k, D_8} = P \left( \mathrm{Grit}( E_4^{a_0} E_6^{b_0} \phi_{0, D_{11}, 1} + E_4^{a_2} E_6^{b_2} \phi_{-2, D_{11}, 1} + E_4^{a_4} E_6^{b_4} \phi_{-4, D_{11}, 1}), \; 4a_i + 6b_i -i = 24+k \right).$$
This has at worst poles on $\cH_{L, 1}$ because that is true for all of the Gritsenko lifts appearing in $P$, and its Fourier--Jacobi expansion begins
$$
\widehat{G}_{24+k,D_{11}} = \Delta^2 \varphi_{k, D_{11},2} \cdot \xi^2 + O(\xi^3),
$$
where $\varphi_{k,D_{11},2} \in J_{k,D_{11},2}^{\w,W(C_{11})}$ has pullback $\phi_{k, D_8, 2}$ to $D_8$. Therefore, $\varphi_{k,D_{11},2} - \phi_{k,D_{11},2}$ is a $\CC[E_4, E_6]$-linear combination of $\phi_{-18,D_{11},2}$, $\phi_{-20,D_{11},2}$ and $\psi_{-11,D_{11},1}^2$. By modifying $\widehat{G}_{24+k, D_8}$ by a polynomial expression in $\cE_{4,D_{11}}, \cE_{6,D_{11}}$ and the forms $\Phi_{4, D_{11}}$, $\Phi_{6, D_{11}}$ and $B_{1, D_{11}}^2$, we obtain the desired generator $G_{24+k,D_{11}}$. 
\end{proof}

By now we have constructed all of the generators for the lattices in the $D_n$-tower
$$
D_4 \hookrightarrow D_5 \hookrightarrow D_6 \hookrightarrow D_7 \hookrightarrow D_8 \hookrightarrow D_9 \hookrightarrow D_{10} \hookrightarrow D_{11}. 
$$
We now consider the lattices of $AD$-type with $L_0\neq \{0\}$. Recall that $L=L_0\oplus D_m$.

\textbf{Case I.} When $m=4$ or $5$, there are at most two generators of Jacobi type, corresponding to the index-two generators of $J_{*,D_m,2}^{\w,W(D_m)}$, and they have been constructed in Lemma \ref{lem:L3}.

\textbf{Case II.} When $m\geq 6$, we need only consider the following towers:
\begin{align*}
    &A_1\oplus D_6 \hookrightarrow A_1\oplus D_7 \hookrightarrow  A_1\oplus D_8 \hookrightarrow A_1\oplus D_9, \\
    &A_2\oplus D_5 \hookrightarrow A_2\oplus D_6 \hookrightarrow  A_2\oplus D_7 \hookrightarrow A_2\oplus D_8, \\
    &A_3\oplus D_4 \hookrightarrow A_3\oplus D_5 \hookrightarrow  A_3\oplus D_6 \hookrightarrow A_3\oplus D_7, \\
    &2A_1\oplus D_4 \hookrightarrow 2A_1\oplus D_5 \hookrightarrow  2A_1\oplus D_6 \hookrightarrow 2A_1\oplus D_7, \\
    &A_1\oplus A_2 \oplus D_3 \hookrightarrow A_1\oplus A_2 \oplus D_4 \hookrightarrow A_1\oplus A_2 \oplus D_5 \hookrightarrow A_1\oplus A_2 \oplus D_6,\\
    &A_4\oplus D_3 \hookrightarrow A_4\oplus D_4 \hookrightarrow A_4\oplus D_5 \hookrightarrow A_4\oplus D_6.
\end{align*}
For each of the above towers, the leftmost lattice can be embedded into $D_8$, and therefore we obtain the Jacobi type generators corresponding to index-two Jacobi forms (i.e. $\phi_{k,D_m,2}$ for $k=-6, -8, ..., -2m$) as pullbacks of generators for $2U\oplus D_8$. These forms can again be expressed as polynomials in additive lifts whose inputs are $\CC[E_4,E_6]$-linear combinations of the basic weak Jacobi forms of type $L_0 \oplus D_m$. We can construct the desired generators for the other lattices in the tower by the same argument as Lemma \ref{lem:D11}. We remind the reader that the pullbacks of generators of type $D_8$ are not exactly the desired generators for the first lattice of the tower, but we can obtain the generators with prescribed leading Fourier--Jacobi coefficients by modifying these pullbacks in a simple way;  for example, the pullback of $G_{10,D_8}$ to $A_1\oplus D_6$ yields the form 
$$
\widehat{G}_{10,A_1\oplus D_6} = \Delta^2 (\phi_{-2,A_1,1}^2\otimes \phi_{-10,D_6,2} + 2 (\phi_{-2,A_1,1}\phi_{0,A_1,1})\otimes \psi_{-6,D_6,1}^2)\cdot \xi^2 + O(\xi^3),
$$
and the desired generator $G_{10, A_1 \oplus D_6}$ is then constructed as the modification
\begin{align*}
    G_{10,A_1\oplus D_6} & = \widehat{G}_{10,A_1\oplus D_6} - 2\Grit(\Delta\phi_{-2,A_1,1}\otimes \psi_{-6,D_6,1})\Grit(\Delta\phi_{0,A_1,1}\otimes\psi_{-6,D_6,1})\\
    & = (\Delta^2 \phi_{-2,A_1,1}^2\otimes \phi_{-10,D_6,2})\cdot \xi^2 + O(\xi^3).
\end{align*}

Theorem \ref{th:2precise} has been proved for the families of $A$-type and $AD$-type. The existence of the generators of Jacobi type for the four $AE$-lattices is harder to prove directly; our argument appears in the next section.

\begin{remark}
If $D_n$ or $E_n$ appears as a component of $L_0$, then $\cH_L$ does not satisfy the Looijenga condition. All such lattices have a sublattice of type $D_4 \oplus A_m$, so we may assume $L_0 = D_4$ and $L_1 = A_m$. Define  $\cH_{L,0}=H(1/2,\varepsilon_1)$ and $\cH_L=\cH_{L,0}\cup \cH_{L,1}$. Then
the quotient
$$
\Grit(\Delta\phi_{-4,D_4,1}\otimes \phi_{-(m+1),A_m,1}) / \Grit(\Delta\psi_{-4,D_4,1}\otimes \phi_{-(m+1),A_m,1})
$$
is a non-constant modular form of weight $0$ with poles on $\cH_L$, violating Koecher's principle. 
\end{remark}

\begin{remark}
When $L_1$ is not irreducible, the ring of $W(L_1)$-invariant weak Jacobi forms is not a polynomial algebra, so the algebra of modular forms for $\widetilde{\Orth}^+(2U\oplus L)$ with poles on $\cH_L$ is also not free. 
\end{remark}

\begin{remark}
We know from \cite{Wan21a} that the Jacobian of generators of a free algebra of holomorphic modular forms is a cusp form. However, the Jacobian of generators of a free algebra of meromorphic modular forms can be non-holomorphic. For example, when $L_0=A_2$ and $L_1=A_1$, by Theorem \ref{th:existence} the corresponding Jacobian has poles of order $1$ along $\cH_{L,0}$. 
\end{remark}

\subsection{A free algebra of meromorphic modular forms for the full orthogonal group}
Theorem \ref{th:2precise} shows that the algebra of modular forms for $\widetilde{\Orth}^+(2U\oplus 5A_1)$ with poles on $\cH_{5A_1}=\cH_{5A_1,0}\cup\cH_{5A_1,1}$ is freely generated in weights $2$, $2$, $2$, $2$, $2$, $4$, $4$, $6$. The arrangement $\cH_{5A_1,0}$ is not invariant under $\Orth^+(2U\oplus 5A_1)$; however, $\cH_{5A_1,1}$ is invariant under $\Orth^+(2U\oplus 5A_1)$. 

We can prove that the ring of meromorphic modular forms for $\Orth^+(2U\oplus 5A_1)$ with poles on the hyperplane arrangement $\cH_{5A_1,1}$ is a polynomial algebra without relations.  This construction is motivated by the free algebras of holomorphic modular forms for the full orthogonal groups of lattices in the tower
\begin{align}\label{eq:nA1}
2U\oplus A_1 \hookrightarrow 2U\oplus 2A_1 \hookrightarrow 2U\oplus 3A_1 \hookrightarrow 2U\oplus 4A_1,
\end{align}
considered by Woitalla in \cite{Woi18}, who proved that the algebra $M_*(\Orth^+(2U\oplus 4A_1))$ at the top is freely generated in weights $4$, $4$, $6$, $6$, $8$, $10$, $12$. The bottom of the tower is the famous Igusa algebra \cite{Igu62} of Siegel modular forms of degree two and even weight which is freely generated in weights $4$, $6$, $10$, $12$. A different interpretation of this tower was given in \cite{WW20a} by associating it to the $B_n$ root system.

Let $5A_1$ be the lattice $\ZZ^5$ with the diagonal Gram matrix $\mathrm{diag}(2,2,2,2,2)$. Then $\cH_{5A_1,1}$ is the Heegner divisor $H(\frac{1}{4}, (\frac{1}{2}, \frac{1}{2}, \frac{1}{2}, \frac{1}{2}, \frac{1}{2}))$. There is a Borcherds product $\widehat{\Phi}_{5A_1}$ of weight $59$ for the full orthogonal group $\Orth^+(2U\oplus 5A_1)$ whose divisor is
$$
H(1,0) + \sum_{\substack{ v \in 5A_1'/5A_1 \\ v^2 = \frac{1}{2}} } H\left(\frac{1}{4}, v\right) + \sum_{\substack{ u \in 5A_1' /5A_1 \\ u^2 = 1} } H\left(\frac{1}{2}, u\right) + 21 H\left( \frac{1}{4}, \left(\frac{1}{2}, \frac{1}{2}, \frac{1}{2}, \frac{1}{2}, \frac{1}{2}\right) \right).
$$
The first three types of divisors correspond to reflections in $\Orth^+(2U\oplus 5A_1)$. Let $\phi_{k,5A_1,1}$ be the generators of $J_{*,5A_1,*}^{\w, \Orth(nA_1)}$, where $k=0, -2, -4, -6, -8, -10$. The additive lifts of $\Delta\phi_{k,5A_1,1}$ together with the forms of weight $4$ and $6$ related to the Eisenstein series $E_4$ and $E_6$ give the generators. Their Jacobian $J$ defines a nonzero modular form of weight $59$. By comparing the divisors we find that $J/\widehat{\Phi}_{59}$ is constant. Then Theorem \ref{th:Jacobiancriterion} implies the following:

\begin{theorem}
The algebra of modular forms for $\Orth^+(2U\oplus 5A_1)$ which are holomorphic away from the Heegner divisor $H(\frac{1}{4}, (\frac{1}{2}, \frac{1}{2}, \frac{1}{2}, \frac{1}{2}, \frac{1}{2}))$ is freely generated by singular additive lifts of weights $2$, $4$, $4$, $6$, $6$, $8$, $10$, $12$. \end{theorem}

In particular, the $2U \oplus nA_1$-tower \eqref{eq:nA1} of modular forms extends naturally to $n = 5$. Similarly, we obtain free rings of meromorphic modular forms for $\Orth^+(2U\oplus D_m)$ with poles on $\cH_{D_m,1}$ for $m=9$, $10$, $11$. In this case, we need only square the generator of weight $12-m$ related to the theta block $\psi_{-m,D_m,1}$.

\section{Algebras of holomorphic modular forms on reducible root lattices}\label{sec:non-free}

In \cite{WW20a} we used Wirthm\"uller's theorem and the existence of modular forms with certain leading Fourier--Jacobi coefficients to determine the ring structure of some modular forms on irreducible root lattices. The argument proves in particular that all formal Fourier--Jacobi series satisfying a suitable symmetry condition actually define modular forms, and was first used by Aoki \cite{Aok00} to give a new proof of Igusa's structure theorem for $\mathrm{Sp}_4(\mathbb{Z})$, corresponding to the root lattice $A_1$.

We will use a similar approach to determine the algebras of modular forms on the discriminant kernel of some reducible root lattices. Unlike \cite{WW20a}, the algebras appearing here are not free.

Let $M=2U\oplus L$ and $G$ be a subgroup of $\Orth(L)$ containing $\widetilde{\Orth}(L)$. We denote by $\Gamma$ the subgroup generated by $\widetilde{\Orth}^+(M)$ and $G$. For any modular form $F$ of weight $k$ and trivial character for $\Gamma$, we write its Fourier and Fourier--Jacobi expansions as follows:
\begin{align*}
F(\tau,\mathfrak{z},\omega)=\sum_{\substack{n,m\in \NN, \ell\in L' \\2nm-\latt{\ell,\ell}\geq 0}}f(n,\ell,m)q^n\zeta^\ell\xi^m
=\sum_{m=0}^{\infty}\phi_m(\tau,\mathfrak{z})\xi^m,
\end{align*}
where as before $q=\exp(2\pi i\tau)$, $\zeta^\ell=\exp(2\pi i \latt{\ell, \mathfrak{z}})$, $\xi=\exp(2\pi i \omega)$.
The coefficients $\phi_m$ are $G$-invariant holomorphic Jacobi forms of weight $k$ and index $m$ associated to $L$, i.e. $\phi_m\in J_{k,L,m}^{G}$. The invariance of $F$ under the involution $(\tau,\mathfrak{z},\omega) \mapsto (\omega, \mathfrak{z}, \tau)$ yields the symmetry relation
$$
f(n,\ell,m)=f(m,\ell,n),\quad \text{for all} \; (n,\ell,m)\in \NN \oplus L' \oplus \NN,
$$
and further implies the bound (see \cite[\S 3]{WW20a} for details)
\begin{equation}\label{eq:MF-JF}
\dim M_k(\Gamma) \leq  \sum_{t=0}^{\infty} \dim J_{k-12t,L,t}^{\w ,G}.
\end{equation}
If there exists a positive constant $\delta<12$ such that 
\begin{equation}
J_{k,L,m}^{\w,G}=\{0\} \quad \text{if} \quad k<-\delta m,
\end{equation}
then \eqref{eq:MF-JF} can be improved to the finite upper bound
\begin{equation}
\dim M_k(\Gamma) \leq \sum_{t=0}^{[\frac{k}{12-\delta}]} \dim J_{k-12t,L,t}^{\w ,G},
\end{equation}
where $[x]$ is the integer part of $x$.

\begin{definition}
The subgroup $\Gamma$ will be called \emph{nice} if the inequality \eqref{eq:MF-JF} degenerates into an equality for every weight $k\in\NN$. 
\end{definition}

\begin{lemma}\label{lem:nice}
The group $\Gamma$ is nice if and only if, for any weak Jacobi form $\phi_m\in J_{k,L,m}^{\w,G}$, there exists a modular form of weight $k+12m$ for $\Gamma$ whose leading Fourier--Jacobi coefficient is $\Delta^m\phi_m \cdot \xi^m$.
\end{lemma}
This implies in particular that $\Delta^m \phi_m$ is holomorphic for all weak Jacobi forms $\phi_m \in J_{k, L, m}^{\w, G}$.
\begin{proof}
For $r\geq 0$ we define
$$
M_k(\Gamma)(\xi^r)=\{F\in M_k(\Gamma) : F=O(\xi^r)\} \quad \text{and} \quad J_{k,L,m}^{G}(q^r)=\{\phi \in J_{k,L,m}^{G}: \phi=O(q^r) \}
$$
For any $r\geq 0$ we have the following exact sequences 
\begin{align*}
&0\longrightarrow M_k(\Gamma)(\xi^{r+1})\longrightarrow  M_k(\Gamma)(\xi^r)\stackrel{P_r}\longrightarrow J_{k,L,r}^{G}(q^r), \\
&0\longrightarrow J_{k,L,r}^{G}(q^r)\stackrel{Q_r}\longrightarrow J_{k-12r,L,r}^{\w, G},    
\end{align*}
where $P_r$ sends $F$ to its $r^{\text{th}}$ Fourier--Jacobi coefficient $\phi_r$, and $Q_r$ maps $\phi_r$ to $\phi_r / \Delta^r$. Taking dimensions shows that $\Gamma$ is nice if and only if, for any $r\geq 0$, the extended sequences 
\begin{align*}
&0\longrightarrow M_k(\Gamma)(\xi^{r+1})\longrightarrow  M_k(\Gamma)(\xi^r)\stackrel{P_r}\longrightarrow J_{k,L,r}^{G}(q^r)\longrightarrow 0 , \\
&0\longrightarrow J_{k,L,r}^{G}(q^r)\stackrel{Q_r}\longrightarrow J_{k-12r,L,r}^{\w, G} \longrightarrow 0,    
\end{align*}
are exact, and the proof follows.
\end{proof}

A formal series of holomorphic  Jacobi forms 
$$
\Psi(Z)=\sum_{m=0}^{\infty} \psi_m \xi^m \in \prod_{m=0}^\infty J_{k,L,m}^{G}
$$
is called a \textit{formal Fourier--Jacobi series} of weight $k$ if it satisfies the symmetry condition
$$
f_m(n,\ell)=f_n(m,\ell), \quad \text{for all}  \quad m,n\in \NN, \ell \in L',
$$
where $f_m(n,\ell)$ are Fourier coefficients of $\psi_m$. We label the space of formal FJ-series $FM_k(\Gamma)$.

\begin{lemma}\label{lem:modularity}
Assume that $\Gamma$ is nice.
\begin{enumerate}
    \item Every formal Fourier--Jacobi series converges on the tube domain $\cH(L)$ and defines a modular form for $\Gamma$. In other words, $$FM_k(\Gamma)=M_k(\Gamma), \quad \text{for all $k\in \NN$}.$$ 
    \item The lattice $L$ satisfies the $\mathrm{Norm}_2$ condition:
     $$
     \delta_L := \max\{ \min\{\latt{y,y}: y\in L + x \} : x \in L' \} \leq 2.
     $$
\end{enumerate}
\end{lemma}
\begin{proof}
(1) The map
$$
M_k(\Gamma) \to FM_k(\Gamma), \quad F\mapsto \text{Fourier--Jacobi expansion of $F$}
$$
is injective. The assumption that $\Gamma$ is nice is precisely the claim that both sides of this map have the same dimension.

(2) Suppose that $\delta_L > 2$. Then there exists a weak Jacobi form $\phi$ whose $q^0$-term contains the nonzero coefficients $\zeta^\ell$ with $\latt{\ell,\ell}>2$. In this case, $\Delta \phi$ is not a holomorphic Jacobi form, which contradicts Lemma \ref{lem:nice}.
\end{proof}

If enough Jacobi forms occur as the leading Fourier--Jacobi coeffcients of modular forms for $\Gamma$, then $\Gamma$ is automatically nice:

\begin{assumption}\label{assum}
We make the following assumptions:
\begin{enumerate}
    \item The bigraded ring $J_{*,L,*}^{\w,G}$ is minimally generated over  $\CC[E_4,E_6]$ by $N$ generators
$$
\phi_i\in J_{k_i,L,m_i}^{\w,G}, \;1\leq i \leq N.
$$
\item There exist modular forms $\cE_4 \in M_4(\Gamma)$ and $\cE_6\in M_6(\Gamma)$ whose zeroth Fourier--Jacobi coefficients are respectively the Eisenstein series $E_4$ and $E_6$.
\item For each $1\leq i\leq N$, there exists a modular form $\Phi_i\in M_{k_i+12m_i}(\Gamma)$ whose Fourier--Jacobi expansion has the form
$$\Phi_i = (\Delta^{m_i} \phi_i) \xi^{m_i} + O(\xi^{m_i+1}).$$
\end{enumerate}
\end{assumption}

\begin{theorem}\label{th:criterion}
If Assumption \ref{assum} holds, then $\Gamma$ is nice. Moreover, $M_*(\Gamma)$ is minimally generated over $\CC$ by $\cE_4$, $\cE_6$ and $\Phi_i$, $1\leq i \leq N$.
\end{theorem}
\begin{proof}
The proof is similar to that of \cite[Proposition 4.2]{WW20a}. Condition (3) immediately implies that $\Gamma$ is nice, so we only need to prove that it is generated by the claimed forms. 
Let $r\geq 1$ and $F_r=\sum_{m=r}^{\infty} f_{m,r} \xi^m \in M_k(\Gamma)(\xi^r)$. For any $m\geq r$, $$f_{m,r}\in J_{k,L,m}^{G}(q^r)$$ implies $f_{r,r}/ \Delta^r\in J_{k-12r,L,r}^{\w,G}$. By Assumption \ref{assum} there exists a polynomial $$P_r\in \CC[E_4,E_6, \phi_i, 1\leq i \leq N]$$ such that $f_{r,r}=\Delta^r P_r(E_4,E_6,\phi_i)$. It follows that
\begin{align*}
&F_{r+1}:=F_r-P_r(\cE_4,\cE_6,\Phi_i)=\sum_{m=r+1}^{\infty} f_{m,r+1} \xi^m \in M_k(\Gamma)(\xi^{r+1}),\\
&f_{m,r+1}\in J_{k,L,m}^G(q^{r+1}), \quad \text{for all $m\geq r+1$}.
\end{align*}
The proof follows by induction over $r$, since $M_k(\Gamma)(\xi^r) = \{0\}$ for $r$ sufficiently large.

We now explain why the set of generators is minimal. Suppose $P$ is a polynomial expression in the generators whose leading Fourier--Jacobi coefficients have index less than some $t \in \mathbb{N}$, and that $P$ itself has leading Fourier--Jacobi coefficient of index $t$.  Then this leading coefficient factors as a product of Jacobi forms of index less than $t$ (indeed, all Fourier--Jacobi coefficients of $P$ do). By minimality of the generators of $J_{*, L, *}^{\w, G}$, this leading coefficient cannot be a $\mathbb{C}[E_4,E_6]$-linear combination of the index $t$ weak Jacobi form generators. This implies that none of the generators of $M_*(\Gamma)$ can be written as a polynomial in the others, so the system of generators is minimal. 
\end{proof}

Theorem \ref{th:criterion} allows us to determine algebras of holomorphic modular forms for nice groups $\Gamma$. We focus on the most interesting case when $L$ is a root lattice and $\Gamma=\widetilde{\Orth}^+(2U\oplus L)$ (i.e. $G=\widetilde{\Orth}(L)$). We first classify all root lattices such that $\widetilde{\Orth}^+(2U\oplus L)$ are nice. 
\begin{enumerate}
    \item The existence of $\cE_4\in M_4(\Gamma)$ forces the singular weight to be at most $4$, and therefore $\rk(L)\leq 8$. 
    \item By Lemma \ref{lem:modularity} (2), the lattice $L$ satisfies $\delta_L \leq 2$.
    \item We know from \cite{HU14} that $M_*(\widetilde{\Orth}^+(2U\oplus E_8))$ is a free algebra. However, by \cite[\S 6]{Wan21b} the ring $J_{*,E_8,*}^{\w,W(E_8)}$ is not a polynomial algebra, and the group $\widetilde{\Orth}^+(2U\oplus E_8)$ is not nice because 
    $$
    \dim M_{34}(\widetilde{\Orth}^+(2U\oplus E_8))=12 \quad \text{but} \quad \sum_{t=0}^\infty \dim J_{34-12t,E_8,t}^{\w,W(E_8)} = 13. 
    $$
\end{enumerate}

The values of $\delta_L$ for irreducible ADE root lattices of rank at most $8$ are listed in Table \ref{tab:delta}.

\begin{remark} Modularity of symmetric formal Fourier--Jacobi series, on the other hand, does not imply that $\Gamma$ is nice; indeed $\Gamma = \widetilde{\Orth}^+(2U\oplus E_8)$ is a counterexample. A symmetric formal Fourier--Jacobi series defines a modular form if and only if it converges, and in particular this condition is preserved by passage to a sublattice; and the convergence of symmetric formal Fourier--Jacobi series for the subgroup $\widetilde{\Orth}^+(2U\oplus D_8)$ was shown in \cite{WW20a}. 
\end{remark}

\begin{table}[ht]
\caption{The values of $\delta_L$}\label{tab:delta}
\renewcommand\arraystretch{1.3}
\noindent\[
\begin{array}{|c|c|c|c|c|c|c|c|c|c|c|c|c|c|c|c|c|}
\hline 
L & A_1 & A_2 & A_3 & A_4 & A_5 & A_6 & A_7 & A_8 & D_4 & D_5 & D_6 & D_7 & D_8 & E_6 & E_7 & E_8 \\ 
\hline 
\delta_L & \frac{1}{2} & \frac{2}{3} & 1 & \frac{6}{5} & \frac{3}{2} & \frac{12}{7} & 2 & \frac{20}{9} & 1 & \frac{5}{4} & \frac{3}{2} & \frac{7}{4} & 2 & \frac{4}{3} & \frac{3}{2} & 0 \\ 
\hline
\end{array} 
\]
\end{table}

\begin{lemma}\label{lem:lattices}
There are exactly $40$ root lattices satisfying the $\mathrm{Norm}_2$ condition whose irreducible components do not contain $E_8$. They are listed in Table \ref{tab:Norm_2 lattices}.
\end{lemma}

\begin{table}[ht]
\caption{Lattices not of type $E_8$ satisfying the $\mathrm{Norm}_2$ condition}\label{tab:Norm_2 lattices}
\renewcommand\arraystretch{1.3}
\noindent\[
\begin{array}{|c|c|}
\hline
\text{Type} & \text{Lattice} \\ \hline
A & A_1,\; 2A_1,\; 3A_1,\; 4A_1,\; A_2,\; A_3,\; A_4,\; A_5,\; A_6,\; A_7,\; 2A_1\oplus A_2,\; 2A_1\oplus A_3,\; A_1\oplus A_2,\\
&  A_1\oplus 2A_2,\; A_1\oplus A_3,\; A_1\oplus A_4,\; A_1\oplus A_5,\; 2A_2,\; 3A_2,\; A_2\oplus A_3,\; A_2\oplus A_4,\; 2A_3 \\ \hline
D & D_4,\; D_5,\; D_6,\; D_7,\; D_8,\; 2D_4 \\ \hline
AD & A_1\oplus D_4,\; A_1\oplus D_5,\; A_1\oplus D_6,\; 2A_1\oplus D_4,\; A_2\oplus D_4,\; A_2\oplus D_5,\; A_3\oplus D_4 \\ \hline
E & E_6,\; E_7 \\ \hline
AE & A_1\oplus E_6,\; A_2\oplus E_6,\; A_1\oplus E_7 \\
\hline
\end{array}
\]
\end{table}

We can now derive a classification of nice groups of type $\widetilde{\Orth}^+(2U\oplus L)$ for root lattices $L$.

\begin{theorem}\label{th:non-free}
Let $L$ be a root lattice. Then $\widetilde{\Orth}^+(2U\oplus L)$ is nice if and only if it is one of the $40$ lattices in Table \ref{tab:Norm_2 lattices}.  For any $L$ in Table \ref{tab:Norm_2 lattices}, the ring of modular forms for $\widetilde{\Orth}^+(2U\oplus L)$ is minimally generated in weights $4$, $6$ and $12m+k$, where the pairs $(k, m)$ are the weights and indices of the generators of the ring of $\widetilde{\Orth}(L)$-invariant weak Jacobi forms. 
\end{theorem}

In the $14$ cases where $L$ is irreducible, this was proved in \cite{WW20a} and the corresponding ring of modular forms is a free algebra. The algebras in the remaining $26$ cases are not freely generated.

\begin{proof}
By the discussions above, if $\widetilde{\Orth}^+(2U\oplus L)$ is nice then $L$ is a lattice in Lemma \ref{lem:lattices}. To prove the theorem it is enough to verify that every lattice in Lemma \ref{lem:lattices} satisfies Assumption \ref{assum}. Let $L=\oplus_{j=1}^n L_j$ be the decomposition into irreducible root lattices. By Wirthm\"uller's theorem \cite{Wir92}, the ring of $J_{*,L_j,*}^{\w,W(L_j)}$ is a polynomial algebra for $1\leq j\leq n$. Then Remark \ref{rem:sum of Jacobi forms} and the fact that
$$
\widetilde{\Orth}(L) = \bigotimes_{j=1}^n \widetilde{\Orth}(L_j) = \bigotimes_{j=1}^n W(L_j). 
$$
together imply that $J_{*,L,*}^{\w, \widetilde{\Orth}(L)}$ is finitely generated by tensor products of generators of $J_{*,L_j,*}^{\w,W(L_j)}$. Moreover, the set of these generators is minimal. Therefore condition (1) in Assumption \ref{assum} is satisfied. Condition (2) follows from Lemma \ref{lem:E4E6}, because $\delta_L\leq 2$ so $\cE_4$ and $\cE_6$ can be constructed as the additive lifts of holomorphic Jacobi forms $E_{4,L}$ and $E_{6,L}$. Only the last condition remains to be verified. When $L$ is of type $A$, as in the previous section, the generators $\phi_i$ of Jacobi forms all have index one, so the generators other than $\cE_4$ and $\cE_6$ can be constructed as the additive lifts $\Delta\phi_i$. For the other cases, the construction is given in the lemmas below. 
\end{proof}

We first construct the generators for lattices of type $AD$. 

\begin{lemma}
Let $L$ be a lattice of type $AD$ in Table \ref{tab:Norm_2 lattices}. Then Condition $(3)$ in Assumption \ref{assum} holds. 
\end{lemma}
\begin{proof}
Let us write $L=A_n\oplus D_m$. 
We first construct the orthogonal generators corresponding to index-one generators of Jacobi forms (i.e. $\phi_{k,A_n,1}\otimes \phi_{l,D_m,1}$) using additive lifts.  When $m \le 5$, we construct the orthogonal generators corresponding to index-two generators of Jacobi forms (i.e. $(\phi_{k_1,A_n,1}\phi_{k_2,A_n,1})\otimes \phi_{l,D_m,2}$) as the generators of Jacobi type in Lemma \ref{lem:L3}. More precisely, the generators corresponding to $(\phi_{k_1,A_n,1}\phi_{k_2,A_n,1})\otimes \phi_{-2(m-1),D_m,2}$ for $m=4,5$  can be constructed as
$$
\Grit(\Delta\phi_{k_1,A_n,1}\otimes \psi_{-m,D_m,1}) \Grit(\Delta\phi_{k_2,A_n,1}\otimes \psi_{-m,D_m,1}) \Psi_{2,L},
$$
and the generators corresponding to $(\phi_{k_1,A_n,1}\phi_{k_2,A_n,1})\otimes \phi_{-6,D_5,2}$ can be constructed as
$$
\Grit(\Delta\phi_{k_1,A_n,1}\otimes \psi_{-5,D_5,1}) \Grit(\Delta\phi_{k_2,A_n,1}\otimes \psi_{-5,D_5,1}) \Psi_{4,L}.
$$
Recall that $\psi_{-m,D_m,1}$ is the theta block \eqref{eq:theta block D_n}, and $\Psi_{2,L}$ and $\Psi_{4,L}$ are meromorphic modular forms constructed in Lemma \ref{lem:L2}. This construction also works for $L=2A_1\oplus D_4$. 

Note that these lattices $L$ can be embedded into $D_8$, so the orthogonal generators corresponding to index-two generators of Jacobi forms can be constructed as pullbacks of generators of type $D_8$. This also yields a construction for $A_1\oplus D_6$. These forms are necessarily holomorphic because $\delta_L\leq 2$ implies $\cH_{L,1}=\emptyset$. 
\end{proof}

Among lattices of type $D$, the only new example is $2D_4$. Since $2D_4$ is a sublattice of $D_8$, we can construct the generators similarly to the above lemma. We omit the details. 

\begin{lemma}
Let $L$ be a lattice of type $AE$ in Table \ref{tab:Norm_2 lattices}. Then Condition $(3)$ in Assumption \ref{assum} holds.
\end{lemma}
\begin{proof} We may assume that $L = A_2 \oplus E_6$ or $L = A_1 \oplus E_7$, as any other lattices of type $AE$ can be embedded into one of them and the generators in Condition $(3)$ can be obtained by restriction.

 It is enough to prove $$\mathrm{dim}\, M_k(\Gamma) \ge \sum_{t=0}^{\infty} \mathrm{dim}\, J_{k - 12t, L, t}^{\w, G}$$ for all $k \le k_0 := \max\{12t_i - k_i\}$, where $(-k_i, t_i)$ are the weights and indices of the generators of $J_{*, L, *}^{\w, G}$, as the reverse inequality is automatically true. We have $k_0 = 24$ if $L = A_2 \oplus E_6$ and $k_0 = 30$ if $L = A_1 \oplus E_7$.

To estimate $\mathrm{dim}\, M_k(\Gamma)$, we use exactly the argument of sections 5.2 and 5.3 of \cite{WW20a}. For any lattice vector $v \in L \backslash \{0\}$, there is a natural ring homomorphism $$P_v : M_k(\Gamma) \longrightarrow M_k(\widetilde{\mathrm{O}}^+(2U\oplus A_1(v^2/2))) \cong M_k(K(v^2/2)),$$ the pullback map, to the space of Siegel paramodular forms of level $v^2/2$ (cf. Section \ref{sec:theta block conjecture}), which is compatible with the Gritsenko lift in the sense that $$P_v(\mathrm{Grit}(\phi)) = \mathrm{Grit}( P_v(\phi))$$ where $P_v$ also denotes the natural pullback of Jacobi forms.

For any modular forms $f_1,...,f_n \in M_k(\Gamma)$, and any vectors $v_1,...,v_n \in L \backslash \{0\}$, we have $$\{\text{linear relations among}\; f_1,...,f_n\} \subseteq \bigcap_{i=1}^n \{\text{linear relations among} \; P_{v_i}(f_1),...,P_{v_i}(f_n)\},$$ so we obtain a lower bound for $\mathrm{dim}\, M_k(\Gamma)$ by computing the Fourier series of paramodular Gritsenko lifts to sufficiently high precision. This is more efficient than working with $M_k(\Gamma)$ directly.

The computation was carried out in SageMath \cite{sagemath} using the package ``WeilRep'' available from \cite{Wil20}. The source code (including our choices of lattice vectors $v_i$) and the results of this computation are available as ancillary material on arXiv.
\end{proof}

\begin{corollary}\label{cor: AE jacobi_type}
The generators of Jacobi type required by Theorem \ref{th:2precise} exist for the lattices of $AE$-type.
\end{corollary}
\begin{proof} Since $\Gamma$ is nice, the existence follows from Lemma \ref{lem:nice}.
\end{proof}

\begin{remark} The Hilbert--Poincar\'e series of $M_*(\widetilde{\Orth}^+(2U\oplus L))$ is $$\sum_{k=0}^{\infty} \mathrm{dim}\, M_k(\widetilde{\Orth}^+(2U\oplus L)) x^k = \sum_{t=0}^{\infty} \sum_{k=0}^{\infty} \mathrm{dim}\, J_{k - 12t, L, t}^{\w, \widetilde{\Orth}(L)} x^k = \sum_{t=0}^{\infty} \sum_{k \in \mathbb{Z}} \mathrm{dim}\, J_{k,L,t}^{\w, \widetilde{\Orth}(L)} x^{k + 12t}.$$ Write $L = L_1 \oplus \cdots \oplus L_n$ where each $L_i$ is irreducible. Then each $L_i$ yields the Hilbert--Poincar\'e series $$F_{L_i}(x, y) = \sum_{t=0}^{\infty} \sum_{k \in \mathbb{Z}} \mathrm{dim}\, J_{k, L_i, t}^{\w, \widetilde{\Orth}(L)} x^k y^t = \frac{1}{(1 - x^4)(1 - x^6) \prod_j (1 - x^{k_j} y^{m_j})}$$ where $(k_j, m_j)$ are the weights and indices of the generators of $J_{*,L_i,*}^{\w,W(L_i)}$ in Table \ref{Tab:Jacobi}. In view of the isomorphism of $\mathbb{C}[E_4,E_6]$-modules $$J_{*, L, t}^{\w, \widetilde{\Orth}(L)} = \bigotimes_{i=1}^n J_{*, L_i,t}^{\w, W(L_i)}$$ it follows that, if we write $$F_{L_i}(x, y) = \frac{1}{(1 - x^4)(1 - x^6)} \sum_{t=0}^{\infty} f_{L_i, t}(x) y^t, \; \; f_{L_i, t} \in \mathbb{C}[x, x^{-1}]$$ then $$F_L(x, y) := \sum_{t=0}^{\infty} \sum_{k \in \mathbb{Z}} \mathrm{dim}\, J_{k, L, t}^{\w,\widetilde{\Orth}(L)} x^k y^t = \frac{1}{(1 - x^4)(1 - x^6)} \sum_{t=0}^{\infty} \Big( \prod_{i=1}^n f_{L_i, t}(x) \Big) y^t$$ and therefore $$
\sum_{k=0}^{\infty} \mathrm{dim}\, M_k(\widetilde{\Orth}^+(2U\oplus L)) x^k = F_L(x, x^{12}).
$$ 
These series are computed explicitly in \S \ref{sec:tables_hol}.
\end{remark}

We can also consider rescalings of the root lattices. 
\begin{theorem}
Let $L$ be a direct sum of rescaled irreducible root lattices, none of which are $E_8$ or $E_8(2)$, and suppose $L$ satisfies the $\mathrm{Norm}_2$ condition. Let $W(L)$ be the direct product of the Weyl groups of all components of $L$. Then the group $\Gamma$ generated by $\widetilde{\Orth}^+(2U\oplus L)$ and $W(L)$ is nice. Moreover, the ring $M_*(\Gamma)$ is minimally generated in weights $4$, $6$ and $12m+k$, where the pairs $(k, m)$ are the weights and indices of the generators of $J_{*,L,*}^{\w, W(L)}$. 
\end{theorem}
This applies to the following lattices which do not appear in Table \ref{tab:Norm_2 lattices}:
\begin{align*}
    &A_1(2)& &A_1(3)& &A_1(4)& &A_2(2)& \\
    &A_2(3)& &A_3(2)& &D_4(2)& &2A_1(2)& \\
    &A_1\oplus A_1(2)& &A_2\oplus A_1(2)& &A_3\oplus A_1(2)& &D_4\oplus A_1(2)&\\
    &2A_1\oplus A_1(2)& &A_1\oplus A_1(3)& &A_1\oplus A_2(2)& &A_2\oplus A_2(2)&
\end{align*}

The proof is essentially the same as Theorem \ref{th:non-free} so we omit it. 

Finally, we consider two examples related to Siegel paramodular forms, which are not covered by the above theorem (because we remove the Weyl invariance). Let $\phi_{0,1}$, $\phi_{-2,1}$ and $\phi_{-1,2}$ be the generators of $J_{*,A_1,*}^{\w}$ as determined by Eichler--Zagier \cite{EZ85}.

\begin{example} 
Let $L=A_1(2)$ and $\Gamma=\widetilde{\Orth}(L)$. We have 
$$
J_{*,A_1(2),*}^{\w}= \CC[E_4,E_6][\phi_{0,1}^2, \phi_{0,1}\phi_{-2,1}, \phi_{-2,1}^2, \phi_{-1,2}],
$$
where all generators are of index one when viewed as Jacobi forms for $A_1(2)$. We conclude from Theorem \ref{th:criterion} that $M_*(\Gamma)$ is generated by the weight $4, 6, 8, 10, 11, 12$ Gritsenko lifts of the Eisenstein series $E_{4, 2}, E_{6, 2}$ and $\Delta \phi$, $\phi \in \{\phi_{0,1}^2, \phi_{0,1}\phi_{-2,1}, \phi_{-2,1}^2, \phi_{-1,2}\}$. This corresponds to the ring of paramodular forms of degree $2$ and level $2$ which was determined in \cite{IO97}.
\end{example}

\begin{example}
Let $L=A_1(3)$ and $\Gamma=\widetilde{\Orth}(L)$. We have 
$$
J_{*,A_1(3),*}^{\w}= \CC[E_4,E_6][\phi_{0,1}^3, \phi_{0,1}^2\phi_{-2,1}, \phi_{0,1}\phi_{-2,1}^2, \phi_{-2,1}^3, \phi_{0,1}\phi_{-1,2}, \phi_{-2,1}\phi_{-1,2}].
$$
All generators are of index one as Jacobi forms for $A_1(3)$. We conclude from Theorem \ref{th:criterion} that $M_*(\Gamma)$ is generated by the weight $4, 6, 6, 8, 9, 10, 11, 12$ Gritsenko lifts of the Eisenstein series $E_{4, 3}, E_{6, 3}$ and $\Delta \phi$, where $\phi$ is one of the above generators of $J_{*, A_1(3), *}^{\w}$. This corresponds to the ring of paramodular forms of degree $2$ and level $3$ which was determined in \cite{Der02}.
\end{example}

\section{Meromorphic modular forms on complex balls}\label{sec:ball quotients}
In this section we establish the modular Jacobian approach for meromorphic modular forms on complex balls attached to unitary groups of signature $(l,1)$ whose poles are supported on hyperplane arrangements. We also apply this criterion to root lattices with complex multiplication to construct free algebras of unitary modular forms with poles on hyperplane arrangements. 

\subsection{The Looijenga compactification of ball quotients}
Following \cite[\S 2]{WW21a} we define modular forms with respect to unitary groups of Hermitian lattices.  Let $d$ be a square-free negative integer.  Then $\F=\QQ(\sqrt{d})$ is an imaginary quadratic field with ring of integers $\mathcal{O}_\F$.  Let $D_\F$ denote the discriminant of $\F$, such that $$
\mathcal{O}_\F=\ZZ+\ZZ\cdot \zeta, \; \text{where} \; \zeta=(D_{\F}+\sqrt{D_\F})/2.
$$ 
Let $\cD_\F^{-1}$ denote the inverse different
$$
\cD_\F^{-1} = \{x \in \F: \; \mathrm{Tr}_{\F / \QQ}(xy) \in \mathbb{Z} \; \text{for all} \; y \in \mathcal{O}_\F\} = \frac{1}{\sqrt{D_{\F}}} \mathcal{O}_\F.
$$
A \emph{Hermitian lattice} $M$ is a free $\mathcal{O}_{\F}$-module equipped with a non-degenerate form 
$$
h(-,-): M \times M \longrightarrow \F
$$ 
which is linear in the first component and conjugate-linear in the second. 
We call $M$ \emph{even} if $h(x,x)\in \ZZ$ for all $x\in M$. We define the dual of $M$ by 
$$
M' = \{x \in M \otimes_{\mathcal{O}_{\F}} \F: \; h(x,y) \in \cD_\F^{-1} \;\text{for all} \; y \in M\}.
$$
We view a Hermitian lattice $M$ as a usual $\ZZ$-lattice (denoted $M_\ZZ$) equipped with the following bilinear form induced by $h(-,-)$
$$
(-, -) := \mathrm{Tr}_{\F/\QQ} h(-,-) : M \otimes M \longrightarrow \mathbb{Q}.
$$
Under this identification, $M$ is even if and only if $M_\ZZ$ is even, and the $\mathcal{O}_\F$-dual $M'$ coincides with the $\ZZ$-dual $M_\ZZ'$. 

Now assume that $M$ is an even Hermitian lattice of signature $(l, 1)$ with $l\geq 2$.  We equip the complex vector space $V_{\U}(M) := M \otimes_{\mathcal{O}_{\F}} \CC$ with the $\QQ$-structure defined by $M$ and the sesquilinear form induced by $h(-,-)$. Then the unitary group $\U(V_{\U}(M))\cong \U(l,1)$ is an algebraic group defined over $\QQ$. The Hermitian symmetric domain attached to $M$  is the Grassmannian of negative-definite lines 
$$
\cD_{\U}(M) = \{[z] \in \PP(V_{\U}(M)): \; h(z, z) < 0\},
$$ on which $\U(l,1)$ acts by multiplication, and it is biholomorphic to the complex ball of dimension $l$. The quotients $\cD_{\U}(M) / \Gamma$ by arithmetic subgroups $\Gamma \le \U(l,1)$ are usually called \emph{ball quotients}. 

Let $\cA_{\U}(M)$ be the principal $\CC^\times$-bundle 
$$
\cA_{\U}(M) = \{z \in V_{\U}(M): \; [z] \in \cD_{\U}(M)\}.
$$
Modular groups $\Gamma$ will be finite-index subgroups of 
$$
\U(M) := \{\gamma \in \U(l,1):\; \gamma M = M\}.
$$ 
The most important example is the \emph{discriminant kernel},
$$
\widetilde{\U}(M) = \{\gamma \in \U(M): \; \gamma x - x \in M \; \text{for all} \; x \in M'\}.
$$

For any vector $v\in M'$ satisfying $h(v,v)>0$, the associated hyperplane $v^\perp$ is the set of vectors of $\cD_{\U}(M)$ orthogonal to $v$ and it defines a complex ball of dimension $l-1$. Following Looijenga's work \cite{Loo03a}, an arrangement $\cH_{\U}$ of hyperplanes is a finite union of $\Gamma$-orbits of hyperplanes in $\cD_{\U}(M)$. We say that $\cH_{\U}$ satisfies the \emph{Looijenga condition} if every one-dimensional intersection in $V_{\U}(M)$ of hyperplanes from $\cH_{\U}$ is positive definite. This condition guarantees that the analogue of Koecher's principle holds for the meromorphic modular forms defined below. 

\begin{definition}
Let $k\in\ZZ$ and $\Gamma$ be a finite-index subgroup of $\U(M)$.  A modular form for $\Gamma$ of weight $k$ and character $\chi$ with poles on $\cH_{\U}$ is a meromorphic function $F: \cA_{\U} \to \CC$ which is holomorphic away from $\cH_{\U}$ and satisfies
\begin{align*}
F(tz)&=t^{-k}F(z),  \quad \text{for all $t\in \CC^\times$},\\
F(\gamma z)&=\chi(\gamma) F(z),  \quad \text{for all $\gamma\in \Gamma$.}
\end{align*}
\end{definition}

Assume that $\cH_{\U}$ satisfies the Looijenga condition. By \cite[Corollary 5.8]{Loo03a}, the ring $M_*^!(\Gamma)$ of modular forms for $\Gamma$ of integral weight and trivial character with poles on $\cH_{\U}$ is finitely generated over $\CC$ by forms of positive weight. Moreover, the $\mathrm{Proj}$ of $M_*^!(\Gamma)$ gives the Looijenga compactification of the complement of $\cH_{\U}$ in $\cD_{\U}(M)$ with respect to $\Gamma$. It is similar to the orthogonal case that the boundary components of the Looijenga compactification have codimension at least two. When the arrangement $\cH_{\U}$ is empty, the Looijenga compactification coincides with the Baily--Borel compactification of $\cD_{\U}(M) / \Gamma$ which is obtained by adding finitely many $0$-dimensional cusps.

\subsection{The Jacobian of unitary modular forms}
Reflections are automorphisms of finite order whose fixed point set has codimension one.  For any primitive vector $r\in M$ satisfying $h(r,r)> 0$ and $\alpha \in \mathcal{O}_\F^\times \setminus \{1\}$, the reflection associated to $r$ and  $\alpha$ is
$$
\sigma_{r, \alpha}: x \mapsto x -(1-\alpha) \frac{ h(x, r)}{h(r,r)} r .
$$
The hyperplane $r^\perp$ is the fixed point set of $\sigma_{r,\alpha}$ and is called the \textit{mirror} of $\sigma_{r, \alpha}$. Note that $\sigma_{r, \alpha} \sigma_{r, \beta}$ equals $\sigma_{r, \alpha \beta}$ or the identity when $\alpha \beta = 1$, and therefore $$
\mathrm{ord}(\sigma_{r, \alpha}) = \mathrm{ord}(\alpha) \in \{2, 3, 4, 6\},
$$ 
where $\mathrm{ord}(\alpha)$ is the order of $\alpha$ in $\mathcal{O}_\F^{\times}$. We refer to \cite[\S 2.3]{WW21a} for a full description of the reflections contained in $\U(M)$ and $\widetilde{\U}(M)$. 

For $1 \leq j \leq l+1$,  let  $F_j$ be a modular form for $\Gamma$ of weight $k_j$ and trivial character with poles on $\cH_{\U}$.   View $F_j$ as meromorphic functions defined on the affine cone $\cA_{\U}(M) \subseteq \CC^{l, 1}$.  With respect to the natural coordinates $(z_1,...,z_{l+1})$ on $\CC^{l, 1}$, the Jacobian determinant of $(F_1,...,F_{l+1})$ is defined in the usual way:
$$
J_{\U}:=J_{\U}(F_1, ...,F_{l+1})=\frac{\partial (F_1,  F_2, ..., F_{l+1})}{\partial (z_1, z_2, ..., z_{l+1})}. 
$$
Similar to \cite[Theorem 3.1]{WW21a}, the Jacobian $J_{\U}$ satisfies the following properties.
\begin{enumerate}
\item $J_{\U}$ is a modular form for $\Gamma$ of weight $l+1 + \sum_{j=1}^{l+1}k_j$ and character $\det^{-1}$ with poles on $\cH_{\U}$, where $\det$ is the determinant character. 
\item $J_{\U}$ is not identically zero if and only if these $F_j$ are algebraically independent over $\CC$.
\item Let $r\in M$ and $\alpha \in \mathcal{O}_\F^\times \setminus \{1\}$.  If $\sigma_{r,\alpha}\in \Gamma$ and $r^\perp$ is not contained in $\cH_{\U}$, then the vanishing order of $J_{\U}$ on $r^\perp$ satisfies
$$
\mathrm{ord}(J_{\U}, r^\perp) \equiv -1 \, \mathrm{mod}\, \ord(\alpha).
$$
\item If we view $F_j$ as functions of $\tau, z_1,...,z_{l-1}$ on the Siegel domain $\cH_{\U}$ attached to a zero-dimensional cusp (i.e. a rational isotropic line), then the Jacobian takes the form
$$
\left\lvert \begin{array}{cccc}
k_1F_1 & k_2F_2 & \cdots & k_{l+1}F_{l+1} \\ 
\partial_\tau F_1 & \partial_\tau F_2 & \cdots & \partial_\tau F_{l+1} \\
\partial_{z_1} F_1 & \partial_{z_1} F_2 & \cdots & \partial_{z_1} F_{l+1} \\ 
\vdots & \vdots & \ddots & \vdots \\ 
\partial_{z_{l-1}} F_{1} & \partial_{z_{l-1}} F_{2} & \cdots & \partial_{z_{l-1}} F_{l+1}
\end{array}   \right\rvert.
$$ 
\end{enumerate}

As analogues of Theorem \ref{th:freeJacobian}, Theorem \ref{th:Jacobiancriterion} and \cite[Theorem 3.3, Theorem 3.4]{WW21a}, we will describe the link between the Jacobian and free algebras of meromorphic modular forms.

\begin{theorem}\label{th:unitaryfreeJacobian}
Let $\cH_{\U}$ be an arrangement of hyperplanes satisfying the Looijenga condition. Suppose that the ring of unitary modular forms for $\Gamma$ with poles on $\cH_{\U}$ is freely generated by forms $F_j$ of weight $k_j$ for $1\leq j \leq l+1$. Then:
\begin{enumerate}
\item $\Gamma$ is generated by reflections.
\item $J_{\U}=J_{\U}(F_1,...,F_{l+1})$ is a nonzero modular form for $\Gamma$ of weight $l+1 + \sum_{j=1}^{l+1}k_j$ and character $\det^{-1}$ which satisfies:
\begin{enumerate}
        \item If a mirror $r^\perp$ of reflections in $\Gamma$ is not contained in $\cH_{\U}$, then $J_{\U}$ vanishes on $r^\perp$ with multiplicity 
        $$
        \mathrm{ord}(J_{\U}, r^{\perp}) = -1 + \max\{\mathrm{ord}(\alpha): \; \sigma_{r, \alpha} \in \Gamma\}.
        $$ 
        \item The other zeros and poles of $J_{\U}$ are contained in $\cH_{\U}$. 
\end{enumerate}
\end{enumerate}
\end{theorem}
\begin{proof}
This theorem can be proved in a similar way to Theorem \ref{th:freeJacobian} and \cite[Theorem 3.3]{WW21a}. The essential fact is that the boundary components of the Loojenga compactification of the arrangement complement $(\cD_{\U}(M)- \cH_{\U}) / \Gamma$ have codimension at least two. 
\end{proof}

\begin{theorem}\label{th:unitaryJacobiancriterion}
Let $\cH_{\U}$ be an arrangement of hyperplanes satisfying the Looijenga condition. Suppose that there exist $l+1$ algebraically independent modular forms of trivial character for $\Gamma$ with poles on $\cH_{\U}$, and suppose the restriction of their Jacobian to the complement $\cD_{\U}(M) - \cH_{\U}$ vanishes precisely on all mirrors of reflections in $\Gamma$ with multiplicity $m-1$,  where $m$ is the maximal order of the reflections in $\Gamma$ through that mirror.  Then these forms freely generate the ring of unitary modular forms for $\Gamma$ with poles on $\cH_{\U}$. In particular, $\Gamma$ is generated by reflections.  
\end{theorem} 
\begin{proof}
The proof is an adjustment to the proofs of Theorem \ref{th:Jacobiancriterion} and \cite[Theorem 3.4]{WW21a}.
\end{proof}

\subsection{Free algebras of meromorphic modular forms on complex balls}
Let $M$ be an even Hermitian lattice of signature $(l, 1)$. There are natural embeddings $$\U(M) \hookrightarrow \Orth^+(M_{\ZZ}) \quad \text{and} \quad \widetilde{\U}(M) \hookrightarrow \widetilde{\Orth}^+(M_{\ZZ}),$$ and modular forms for $\Orth^+(M_{\ZZ})$ can be pulled back to modular forms for $\U(M)$ by restricting to those lines in $\cD(M)$ which are preserved by the complex structure on $M$; for more detail see \cite[\S 2.4]{WW21a}. In \cite[\S 4]{WW21a} we used the relationship between the Jacobians of orthogonal and unitary modular forms to show that free algebras of modular forms for $\Gamma \le \Orth^+(M_{\ZZ})$ often yield free algebras of modular forms for $\Gamma \cap \U(M)$. This result naturally extends to modular forms with poles:

\begin{theorem}\label{th:twins}
Let $M$ be an even Hermitian lattice of signature $(l,1)$ over $\F=\QQ(\sqrt{-1})$ or $\QQ(\sqrt{-3})$ with $l\geq 2$. Let $\cH$ be a hyperplane arrangement on $\cD(M_{\ZZ})$ satisfying the Looijenga condition. Let $\cH_{\U}$ denote the hyperplane arrangement obtained as the restriction of $\cH$ to $\cD_{\U}(M)$, and assume that $\cH_{\U}$ also satisfies the Looijenga condition. 
Suppose that the ring of modular forms for $\widetilde{\Orth}^+(M_{\ZZ})$ with poles on $\cH$ is freely generated by $2l+1$ orthogonal modular forms, $l$ of whose restrictions to $\cD_{\U}(M)$ are identically zero.  Then the ring of modular forms for $\widetilde{\U}(M)$ with poles on $\cH_{\U}$ is freely generated by the restrictions to $\cD_{\U}(M)$ of the remaining $l+1$ generators of orthogonal type.
\end{theorem}
\begin{proof}
For simplicity we label $\cD^\circ=\cD(M_{\ZZ}) - \cH$ and $\cD_{\U}^\circ=\cD_{\U}(M) - \cH_{\U}$.
Let $F_j$, $0 \le j \le 2l$, be the $2l+1$ generators for $M_*(\widetilde{\Orth}^+(M_{\ZZ}))$, and let $f_j$ be their restrictions to $\cD_{\U}(M)$. Without loss of generality, suppose $f_j = 0$ for $j \geq l+1$. Let $J_{\Orth}$ be the Jacobian of $F_0,...,F_{2l}$, and let $J_{\U}$ be the Jacobian of $f_0,...,f_l$. Note that $J_{\Orth}$ has its zero locus supported on Heegner divisors, all of which intersect $\cD_{\U}(M)$ transversally, so its restriction $\widehat{J}_{\Orth}$ to $\cD_{\U}(M)$ does not vanish identically. 

By \cite[Proposition 4.1]{WW21a} and its natural analogue to modular forms with poles, there is a meromorphic unitary modular form $g$, holomorphic away from $\cH_{\U}$, such that $\widehat{J}_{\Orth}=g J_{\U}$. It follows that $J_{\U}$ is also not identically zero.  According to \cite[Lemma 2.2]{WW21a}, the restrictions of mirrors of reflections in $\widetilde{\Orth}^+(M_{\ZZ})$ to $\cD_{\U}(M)$ are exactly the mirrors of reflections in $\widetilde{\U}(M)$. By Theorem \ref{th:2precise}, on the complement $\cD^\circ$, the orthogonal Jacobian $J_{\Orth}$ vanishes precisely with multiplicity one on mirrors of reflections in $\widetilde{\Orth}^+(M_{\ZZ})$. As a factor of $\widehat{J}_{\Orth}$, the unitary Jacobian $J_{\U}$ also vanishes only on mirrors of reflections in $\widetilde{\U}(M)$ on $\cD_{\U}^\circ$. 

By considering its character, we see that $J_{\U}$ must vanish on every mirror of a reflection in $\widetilde{\U}(M)$ which is not contained in $\cH_{\U}$. Therefore $J_{\U}$ vanishes precisely on the mirrors of reflections in $\widetilde{\U}(M)$ in $\cD_{\U}^\circ$. To apply Theorem \ref{th:unitaryJacobiancriterion}, we must show that the order of vanishing of $J_{\U}$ on any mirror is exactly $1$ when $d=-1$ and $2$ when $d=-3$. This is true because the restriction $\widehat{J}_{\Orth}$ of $J_{\Orth}$ vanishes with multiplicity at most $$N := |\mathcal{O}_\F^\times / \{\pm 1\}| = \begin{cases}2 : & d = -1; \\ 3: & d = -3; \end{cases},$$ so its factor $J_{\U}$ vanishes to at most that multiplicity; and because $\ord(J_{\U}) \equiv -1$ mod $N$.
\end{proof}

Note that the Looijenga condition for $\cH$ does not imply the Looijenga condition for $\cH_{\U}$ in general. Two examples where this fails are $L = A_1 \oplus A_1$ and $L = 3A_1 \oplus A_1$, viewed as Gaussian lattices, where $\cH_L$ is defined in Theorem \ref{th:2precise}.

The above proof also shows that, up to constant multiple, $J_{\U} = \widehat{J}_{\Orth}^{1/2}$ if $d=-1$ and $J_{\U} = \widehat{J}_{\Orth}^{2/3}$ if $d=-3$.

As in \cite[Theorem 4.2]{WW21a}, the above theorem holds when we replace $\widetilde{\Orth}^+(M_{\ZZ})$ and $\widetilde{\U}(M)$ with $\Orth(M_{\ZZ})$ and $\U(M)$ and the proof is similar. We cannot extend this theorem to other discriminants, because when $d \notin \{-1, -3\}$ the restriction of a reflection in $\Orth^+(M_{\ZZ})$ is not necessarily a reflection in $\U(M)$.

The following lemma is a convenient way to prove that certain modular forms vanish along $\cD_{\U}(M)$.
\begin{lemma}\label{lem:zero}
Let $M=H\oplus L$ be an even Hermitian lattice of signature $(l,1)$ over $\F = \QQ(\sqrt{d})$ split by a complex line $H$ satisfying $H_{\ZZ}=U$. We write $L=\oplus_{j=1}^n L_j$.  Let $\Gamma$ denote the subgroup generated by $\widetilde{\Orth}^+(M_{\ZZ})$ and these $\Orth((L_j)_{\ZZ})$.
Suppose that $F$ is a nonzero meromorphic modular form of integral weight $k$ and trivial character for $\Gamma$.  Let $f$ denote the restriction of $F$ to $\cD_{\U}(M)$.
\begin{enumerate}
    \item When $d=-1$, if $k$ is not a multiple of $4$, then $f$ is identically zero.  
    \item When $d=-3$, if $k$ is not a multiple of $6$, then $f$ is identically zero.
\end{enumerate}
\end{lemma}
\begin{proof}
For any unit $\alpha \in \mathcal{O}_\F^\times$, the map $\alpha_{L_j}: z \mapsto \alpha\cdot z$ lies in $\U(L_j)$, and $\alpha_H : z \mapsto \alpha\cdot z$ lies in $\U(H) = \widetilde{\U}(H)$. Therefore $\Gamma \cap \U(M)$ contains the automorphism $\alpha_M : z \mapsto \alpha\cdot z$, which can be viewed as the composition of $\alpha_H$ and these $\alpha_{L_j}$. Suppose $\alpha$ has order $a$. Then
$$
f(z)=f(\alpha \cdot z)=\alpha^{-k}f(z),
$$
which implies that $a|k$ if $f$ is nonzero.
\end{proof}

Some of the root lattices in Theorem \ref{th:2precise} have complex multiplication over $\QQ(\sqrt{-1})$ or $\QQ(\sqrt{-3})$.  We obtain the following result by applying Theorem \ref{th:twins} to them.

\begin{theorem}\label{th:algebras-unitary}
\noindent
 Let $M$ be a Gaussian lattice whose associated $\ZZ$-lattice is $M_{\ZZ}=2U\oplus L$, where $L$ is  $2A_1\oplus D_4$, $2A_1\oplus D_6$ or $D_{10}$, or an Eisenstein lattice whose associated $\ZZ$-lattice is $M_{\ZZ}=2U\oplus L$, where $L$ is $A_2\oplus A_2$, $2A_2\oplus A_2$, $A_2\oplus D_4$, $2A_2\oplus D_4$ or $A_2\oplus E_6$. Let $\cH_L$ be the hyperplane arrangement for $M_{\ZZ}$ defined in Theorem \ref{th:2precise} and let $\cH_{\U}$ be the restriction of $\cH_L$ to $\cD_{\U}(M)$. Then $\cH_{\U}$ satisfies the Looijenga condition and the ring of modular forms for $\widetilde{\U}(M)$ with poles on $\cH_{\U}$ is freely generated.
\end{theorem}
The weights of the generators of the twin free algebras $M_*^!(\widetilde{\Orth}^+(2U \oplus L))$ and $M_*^!(\widetilde{\U}(M))$ are listed in Tables \ref{tab:-1} and \ref{tab:-3} below.

\begin{proof}
It is easy to verify that the Looijenga condition is satisfied for all $\cH_{\U}$ by Lemma \ref{lem:intersection} and its unitary analogue. To apply Theorem \ref{th:twins} we need to argue that there are $l$ generators of orthogonal type whose restrictions to $\cD_{\U}(M)$ are identically zero. (Recall that the signature of $M_{\ZZ}$ is $(2l,2)$, so $l = \rk(L) / 2 - 1$.)

We will prove the existence of these generators for two of the harder cases; the other cases are similar and we omit the details.

(1) $L=2A_1\oplus D_6$ as a Gaussian lattice. As explained in Notation \ref{notation}, we can choose the generators of the ring of $W(D_6)$-invariant weak Jacobi forms as follows: (i) $\phi_{k,D_6,2}$ are invariant under $\Orth(D_6)$ for $k=-6, -8, -10$; (ii) $\phi_{k,D_6,1}$ are invariant under $\Orth(D_6)$ for $k=0, -2, -4$. (iii) $\psi_{-6,D_6,1}$ is invariant under $W(D_6)$ and anti-invariant under the odd sign change. Therefore, we can choose the generators of the ring of modular forms for $\widetilde{\Orth}^+(M_{\ZZ})$ with poles on $\cH_L$ in the following way.
\begin{itemize}
    \item[(a)] The two generators of Eisenstein type are invariant under $\Orth(2A_1)\otimes\Orth(D_6)$. By Lemma \ref{lem:zero}, the generator of Eisenstein type of weight $6$ restricts to zero.
    \item[(b)] There are two generators of abelian type, both of weight $2$. Up to scalar there is a unique linear combination of them which is invariant under $\Orth(2A_1)\otimes\Orth(D_6)$, and whose restriction must be zero by Lemma \ref{lem:zero}.
    \item[(c)] There are six generators of Jacobi type related to the six basic $\Orth(D_6)$-invariant weak Jacobi forms. They have weights $8$, $6$, $4$, $10$, $8$, $6$. By Lemma \ref{lem:zero}, the restrictions of three generators of weight $6$, $10$ and $6$ are zero.
\end{itemize}
We have found 5 generators of $M^!_*(\widetilde{\Orth}^+(M_{\ZZ}))$ whose restrictions are identically zero, so we can apply Theorem \ref{th:twins}.

\vspace{3mm}

(2) $L=2A_2\oplus D_4$ as an Eisenstein lattice. This case is more subtle. Recall that the ring of $W(A_2)$-invariant weak Jacobi forms has generators $\phi_{0,A_2,1}$, $\phi_{-2,A_2,1}$ and $\phi_{-3,A_2,1}$. Note that $\Orth(A_2)$ is generated by $W(A_2)$ and the sign change $\mathfrak{z}\mapsto -\mathfrak{z}$, so the even weight generators are invariant under $\Orth(A_2)$, but the odd weight generator is anti-invariant under the sign change.  The $W(D_4)$-invariant weak Jacobi forms have five generators: $\phi_{0,D_4,1}$, $\phi_{-2,D_4,1}$, $\phi_{-4,D_4,1}$, $\psi_{-4,D_4,1}$ and  $\phi_{-6,D_4,2}$. The full orthogonal group of $D_4$ is the Weyl group of the $F_4$ root system and its ring of weak Jacobi forms is generated in weights and indices $(0,1)$, $(-2,1)$, $(-6,2)$, $(-8,2)$, $(-12,3)$; we can take $\phi_{0,D_4,1}$, $\phi_{-2,D_4,1}$ and $\phi_{-6,D_4,2}$ to be $\Orth(D_4)$-invariant. Similar to the above case, we can choose the generators of the ring of modular forms for $\widetilde{\Orth}^+(M_{\ZZ})$ with poles on $\cH_L$ in the following way.
\begin{itemize}
\item[(a)] The two generators of Eisenstein type are invariant under $\Orth(2A_2)\otimes\Orth(D_4)$. Lemma \ref{lem:zero} shows that the weight $4$ generator restricts to zero.
\item[(b)] There are four generators of abelian type of weights $1,1,3,3$.  The two generators of weight $1$ are labelled $F_1$ and $G_1$, and they are mapped into each other by the element $\sigma \in \Orth(L)$ which swaps the two copies of $A_2$. Then $F_1^2+G_1^2$ and $F_1^2G_1^2$ are meromorphic modular forms of weight $2$ and $4$, invariant under $\Orth(2A_2)\otimes\Orth(D_4)$, whose restrictions have to be zero by Lemma \ref{lem:zero}. It follows that the restrictions of $F_1$ and $G_1$ themselves are identically zero.
\item[(c)] There are five generators of Jacobi type, corresponding to the five Jacobi forms $\phi_{0,D_4,1}$, $\phi_{-2,D_4,1}$, $\phi_{-4,D_4,1}$, $\psi_{-4,D_4,1}$ and  $\phi_{-6,D_4,2}$, and we label them $F_{6}$, $F_{4}$, $F_{2}$, $G_{2}$ and $G_{6}$ respectively, where the subscript indicates the weight. Our choice of the generators guarantees that $F_6^2$, $F_4^2$ and $G_6$ are invariant under $\Orth(2A_2)\otimes\Orth(D_4)$, which implies that the restriction of $F_4$ is zero. By \cite[\S 4.1]{Adl20}, the basic $\Orth(D_4)$-invariant weak Jacobi form of weight $-8$ and index $2$ can be expressed as $\phi_{-4,D_4,1}^2 -\psi_{-4,D_4,1}^2$ if we choose $\phi_{-4,D_4,1}$ and $\psi_{-4,D_4,1}$ in a suitable way. Then $F_2^2 - G_2^2$ is a modular form of weight $4$, which is invariant under $\Orth(2A_2)\otimes\Orth(D_4)$, and therefore restricts to zero. It follows that the restriction of one of $F_2+G_2$ or $F_2-G_2$ is identically zero.
\end{itemize}
As before, having found $5$ generators of $M^!_*(\widetilde{\Orth}^+(M_{\ZZ}))$ whose restrictions are identically zero, we can now apply Theorem \ref{th:twins}.
\end{proof}

\begin{table}[ht]
\caption{Free algebras of modular forms for $\widetilde{\U}(M)$ with poles on $\cH_{\U}$ over $\QQ(\sqrt{-1})$}\label{tab:-1}
\renewcommand\arraystretch{1.3}
\noindent\[
\begin{array}{|c|c|c|}
\hline 
L & \text{weights of generators of $M_*^!(\widetilde{\Orth}^+(M_\ZZ))$} & \text{weights of generators of $M_*^!(\widetilde{\U}(M))$} \\ 
\hline 
2A_1\oplus D_4 & 2, 2, 4, 4, 4, 6, 6, 8, 10 & 2, 4, 4, 4, 8 \\
\hline
2A_1\oplus D_6 & 2, 2, 2, 4, 4, 6, 6, 6, 8, 8, 10 & 2, 2, 4, 4, 8, 8\\
\hline
D_{10} & 2, 4, 6, 6, 8, 8, 10, 10, 12, 12, 14, 16, 18 & 2, 4, 8, 8, 12, 12, 16\\
\hline
\end{array} 
\]
\end{table}

\begin{table}[ht]
\caption{Free algebras of modular forms for $\widetilde{\U}(M)$ with poles on $\cH_{\U}$ over $\QQ(\sqrt{-3})$}\label{tab:-3}
\renewcommand\arraystretch{1.3}
\noindent\[
\begin{array}{|c|c|c|}
\hline 
L & \text{weights of generators of $M_*^!(\widetilde{\Orth}^+(M_\ZZ))$} & \text{weights of generators of $M_*^!(\widetilde{\U}(M))$} \\ 
\hline 
A_2\oplus A_2 & 1, 3, 4, 6, 6, 7, 9 & 3, 6, 6, 9 \\ 
\hline
2A_2\oplus A_2 & 1, 1, 3, 3, 3, 4, 4, 6, 6 & 3, 3, 3, 6, 6 \\
\hline
A_2\oplus D_4 & 1, 3, 4, 5, 5, 6, 7, 9, 12 & 3, 5, 6, 9, 12 \\
\hline
2A_2\oplus D_4 & 1, 1, 2, 2, 3, 3, 4, 4, 6, 6, 6 & 2, 3, 3, 6, 6, 6 \\
\hline
A_2\oplus E_6 & 1, 3, 4, 4, 6, 7, 9, 9, 10, 12, 15 & 3, 6, 9, 9, 12, 15\\
\hline
\end{array} 
\]
\end{table}

\begin{remark}
As a direct consequence of Theorem \ref{th:algebras-unitary}, the modular groups $\widetilde{\U}(M)$ related to these free algebras are generated by reflections. Thus they provide explicit examples of finite-covolume reflection groups acting on complex hyperbolic spaces (see \cite{All00a}). 
\end{remark}

\bigskip

\noindent
\textbf{Acknowledgements} 
H. Wang thanks Max Planck Institute for Mathematics (MPIM Bonn) for its hospitality where part of the work was done, and thanks Zhiwei Zheng for helpful discussions on the Looijenga compactification. H. Wang was supported by the Institute for Basic Science (IBS-R003-D1). B. Williams thanks Cris Poor and David Yuen for interesting discussions, especially related to Section \ref{sec:non-free}.

\addtocontents{toc}{\setcounter{tocdepth}{1}} 
\section{Appendix: Tables}\label{appendix}
\subsection{Weights of generators of free algebras of meromorphic modular forms}
On the following pages we list the weights of generators for the $147$ free algebras of meromorphic modular forms on the discriminant kernels of $2U \oplus L$, which $L=L_0\oplus L_1$ are root lattices in the three families \eqref{eq:lattices}. The allowed poles lie on the hyperplane arrangement $\cH_L = \cH_{L, 0} \cup \cH_{L, 1}$ defined in Section \ref{sec:type IV}. Generators of abelian type are never holomorphic, because they have poles on $\cH_{L, 0}$; while generators of Eisenstein and Jacobi type may be holomorphic or have poles only on the part $\cH_{L, 1}$. The tables list the weights of generators of Eisenstein, abelian, and Jacobi type, and the weight of their Jacobian.

We also list in Table \ref{tab:predict} some algebras we predict to be freely generated as well as the weights, but for which we have been unable to construct appropriate generators. The ring structure of $M_*(\Orth^+(2U \oplus E_8))$ was shown by \cite{HU14} to be the free algebra on generators of weights $4$, $10$, $12$, $16$, $18$, $22$, $24$, $28$, $30$, $36$, $42$ and it is very exceptional. It does not fit into the pattern of generators for the other root lattices so we omit it and other potential interesting algebras of meromorphic modular forms on root lattices containing it as a component.

\vspace{5mm}

\noindent\[
\renewcommand{\arraystretch}{1.0}
\begin{array}{|c|c|c|c|c|c|}
\hline
L_0 & L_1 & \text{Eisenstein} & \text{Abelian} & \text{Jacobi} & \text{wt}\, J \\
\hline
0 & A_1 & 4, 6 & - & 10, 12 & 35\\
A_1 & A_1 & 4, 6 & 2 & 8, 10  & 34 \\
2A_1 & A_1 & 4, 6 & 2, 2 & 6, 8 & 33 \\
3A_1 & A_1 & 4, 6 & 2, 2, 2 & 4, 6 & 32 \\
4A_1 & A_1 & 4, 6 & 2, 2, 2, 2 & 2, 4 & 31\\
A_1 \oplus A_2 & A_1 & 4, 6 & 1, 2, 3 & 5, 7 & 34 \\
A_1 \oplus A_3 & A_1 & 4, 6 & 1, 2, 2, 4 & 4, 6 & 36  \\
A_1 \oplus A_4 & A_1 & 4, 6 & 1, 2, 2, 3, 5 & 3, 5 & 39  \\
\hline
\end{array}
\]

\clearpage

\noindent\[
\renewcommand{\arraystretch}{1.0}
\begin{array}{|c|c|c|c|c|c|}
\hline
L_0 & L_1 & \text{Eisenstein} & \text{Abelian} & \text{Jacobi} & \text{wt}\, J \\
\hline
A_1 \oplus A_5 & A_1 & 4, 6 & 1, 2, 2, 3, 4, 6 & 2, 4 & 43 \\
A_1\oplus A_6 & A_1 & 4, 6 & 1, 2, 2, 3, 4, 5, 7 & 1, 3 & 48  \\
A_1 \oplus 2A_2 & A_1 & 4, 6 & 1, 1, 2, 3, 3 & 2, 4 & 34  \\
A_1 \oplus A_2 \oplus A_3 & A_1 & 4, 6 & 1, 1, 2, 2, 3, 4 & 1, 3 & 36  \\
2A_1 \oplus A_2 & A_1 & 4, 6 & 1, 2, 2, 3 & 3, 5 & 33 \\
2A_1 \oplus A_3 & A_1 & 4, 6 & 1, 2, 2, 2, 4 & 2, 4 & 35 \\
2A_1 \oplus A_4 & A_1 & 4, 6 & 1, 2, 2, 2, 3, 5 & 1, 3 & 38 \\
3A_1 \oplus A_2 & A_1 & 4, 6 & 1, 2, 2, 2, 3 & 1, 3 & 32 \\
A_2 & A_1 & 4, 6 & 1, 3 & 7, 9 & 35 \\
2A_2 & A_1 & 4, 6 & 1, 1, 3, 3 & 4, 6 & 35 \\
3A_2 & A_1 & 4, 6 & 1, 1, 1, 3, 3, 3 & 1, 3 & 35 \\
A_2 \oplus A_3 & A_1 & 4, 6 & 1, 1, 2, 3, 4 & 3, 5 & 37 \\
A_2 \oplus A_4 & A_1 & 4, 6 & 1, 1, 2, 3, 3, 5 & 2, 4 & 40 \\
A_2 \oplus A_5 & A_1 & 4, 6 & 1, 1, 2, 3, 3, 4, 6 & 1, 3 & 44 \\
A_3 & A_1 & 4, 6 & 1, 2, 4 & 6, 8 & 37 \\
2A_3 & A_1 & 4, 6 & 1, 1, 2, 2, 4, 4 & 2, 4 & 39 \\
A_3 \oplus A_4 & A_1 & 4, 6 & 1, 1, 2, 2, 3, 4, 5 & 1, 3 & 42 \\
A_4 & A_1 & 4, 6 & 1, 2, 3, 5 & 5, 7 & 40 \\
A_5 & A_1 & 4, 6 & 1, 2, 3, 4, 6 & 4, 6 & 44 \\
A_6 & A_1 & 4, 6 & 1, 2, 3, 4, 5, 7 & 3, 5 & 49 \\
A_7 & A_1 & 4, 6 & 1, 2, 3, 4, 5, 6, 8 & 2, 4 & 55 \\
A_8 & A_1 & 4, 6 & 1, 2, 3, 4, 5, 6, 7, 9 & 1, 3 & 62 \\
0 & A_2 & 4, 6 & - & 9, 10, 12 & 45\\
A_1 & A_2 & 4, 6 & 2 & 7, 8, 10 & 42 \\
2A_1 & A_2 & 4, 6 & 2, 2 & 5, 6, 8 & 39 \\
3A_1 & A_2 & 4, 6 & 2, 2, 2 & 3, 4, 6 & 36 \\
4A_1 & A_2 & 4, 6 & 2, 2, 2, 2 & 1, 2, 4 & 33 \\
A_1 \oplus A_2 & A_2 & 4, 6 & 1, 2, 3 & 4, 5, 7 & 39 \\
A_1 \oplus A_3 & A_2 & 4, 6 & 1, 2, 2, 4 & 3, 4, 6 & 40 \\
A_1 \oplus A_4 & A_2 & 4, 6 & 1, 2, 2, 3, 5 & 2, 3, 5 & 42 \\
A_1 \oplus A_5 & A_2 & 4, 6 & 1, 2, 2, 3, 4, 6 & 1, 2, 4 & 45 \\
A_1 \oplus 2A_2 & A_2 & 4, 6 & 1, 1, 2, 3, 3 & 1, 2, 4 & 36 \\
2A_1 \oplus A_2 & A_2 & 4, 6 & 1, 2, 2, 3 & 2, 3, 5 & 36 \\
2A_1 \oplus A_3 & A_2 & 4, 6 & 1, 2, 2, 2, 4 & 1, 2, 4 & 37 \\
A_2 & A_2 & 4, 6 & 1, 3 & 6, 7, 9 & 42 \\
2A_2 & A_2 & 4, 6 & 1, 1, 3, 3 & 3, 4, 6 & 39 \\
A_2 \oplus A_3 & A_2 & 4, 6 & 1, 1, 2, 3, 4 & 2, 3, 5 & 40 \\
A_2 \oplus A_4 & A_2 & 4, 6 & 1, 1, 2, 3, 3, 5 & 1, 2, 4 & 42 \\
A_3 & A_2 & 4, 6 & 1, 2, 4 & 5, 6, 8 & 43 \\
2A_3 & A_2 & 4, 6 & 1, 1, 2, 2, 4, 4 & 1, 2, 4 & 41 \\
A_4 & A_2 & 4, 6 & 1, 2, 3, 5 & 4, 5, 7 & 45 \\
A_5 & A_2 & 4, 6 & 1, 2, 3, 4, 6 & 3, 4, 6 & 48 \\
A_6 & A_2 & 4, 6 & 1, 2, 3, 4, 5, 7 & 2, 3, 5 & 52 \\
A_7 & A_2 & 4, 6 & 1, 2, 3, 4, 5, 6, 8 & 1, 2, 4 & 57 \\
0 & A_3 & 4, 6 & - & 8, 9, 10, 12 & 54 \\
A_1 & A_3 & 4, 6 & 2 & 6, 7, 8, 10 & 49 \\
2A_1 & A_3 & 4, 6 & 2, 2 & 4, 5, 6, 8 & 44 \\
\hline
\end{array}
\]

\clearpage

\noindent\[
\renewcommand{\arraystretch}{1.0}
\begin{array}{|c|c|c|c|c|c|}
\hline
L_0 & L_1 & \text{Eisenstein} & \text{Abelian} & \text{Jacobi} & \text{wt}\, J \\
\hline
3A_1 & A_3 & 4, 6 & 2, 2, 2 & 2, 3, 4, 6 & 39 \\
A_1 \oplus A_2 & A_3 & 4, 6 & 1, 2, 3 & 3, 4, 5, 7 & 43 \\
A_1 \oplus A_3 & A_3 & 4, 6 & 1, 2, 2, 4 & 2, 3, 4, 6 & 43 \\
A_1 \oplus A_4 & A_3 & 4, 6 & 1, 2, 2, 3, 5 & 1, 2, 3, 5 & 44 \\
2A_1 \oplus A_2 & A_3 & 4, 6 & 1, 2, 2, 3 & 1, 2, 3, 5 & 38 \\
A_2 & A_3 & 4, 6 & 1, 3 & 5, 6, 7, 9 & 48 \\
2A_2 & A_3 & 4, 6 & 1, 1, 3, 3 & 2, 3, 4, 6 & 42 \\
A_2 \oplus A_3 & A_3 & 4, 6 & 1, 1, 2, 3, 4 & 1, 2, 3, 5 & 42 \\
A_3 & A_3 & 4, 6 & 1, 2, 4 & 4, 5, 6, 8 & 48 \\
A_4 & A_3 & 4, 6 & 1, 2, 3, 5 & 3, 4, 5, 7 & 49 \\
A_5 & A_3 & 4, 6 & 1, 2, 3, 4, 6 & 2, 3, 4, 6 & 51 \\
A_6 & A_3 & 4, 6 & 1, 2, 3, 4, 5, 7 & 1, 2, 3, 5 & 54 \\
0 & A_4 & 4, 6 & - & 7, 8, 9, 10, 12 & 62 \\
A_1 & A_4 & 4,6 & 2 & 5, 6, 7, 8, 10 & 55 \\
2A_1 & A_4 & 4, 6 & 2, 2 & 3, 4, 5, 6, 8 & 48 \\
3A_1 & A_4 & 4,6 & 2, 2, 2 & 1, 2, 3, 4, 6 & 41 \\
A_1 \oplus A_2 & A_4 & 4, 6 & 1, 2, 3 & 2, 3, 4, 5, 7 & 46 \\
A_1 \oplus A_3 & A_4 & 4, 6 & 1, 2, 2, 4 & 1, 2, 3, 4, 6 & 45 \\
A_2 & A_4 & 4, 6 & 1, 3 & 4, 5, 6, 7, 9 & 53 \\
2A_2 & A_4 & 4, 6 & 1, 1, 3, 3 & 1, 2, 3, 4, 6 & 44 \\
A_3 & A_4 & 4, 6 & 1, 2, 4 & 3, 4, 5, 6, 8 & 52 \\
A_4 & A_4 & 4, 6 & 1, 2, 3, 5 & 2, 3, 4, 5, 7 & 52 \\
A_5 & A_4 & 4, 6 & 1, 2, 3, 4, 6 & 1, 2, 3, 4, 6 & 53 \\
0 & A_5 & 4, 6 & - & 6, 7, 8, 9, 10, 12 & 69 \\
A_1 & A_5 & 4, 6 & 2 & 4, 5, 6, 7, 8, 10 & 60 \\
2A_1 & A_5 & 4, 6 & 2, 2 & 2, 3, 4, 5, 6, 8 & 51 \\
A_1 \oplus A_2 & A_5 & 4, 6 & 1, 2, 3 & 1, 2, 3, 4, 5, 7 & 48 \\
A_2 & A_5 & 4, 6 & 1, 3 & 3, 4, 5, 6, 7, 9 & 57 \\
A_3 & A_5 & 4, 6 & 1, 2, 4 & 2, 3, 4, 5, 6, 8 & 55 \\
A_4 & A_5 & 4, 6 & 1, 2, 3, 5 & 1, 2, 3, 4, 5, 7 & 54 \\
0 & A_6 & 4, 6 & - & 5, 6, 7, 8, 9, 10, 12 & 75 \\
A_1 & A_6 & 4, 6 & 2 & 3, 4, 5, 6, 7, 8, 10 & 64 \\
2A_1 & A_6 & 4, 6 & 2, 2 & 1, 2, 3, 4, 5, 6, 8 & 53 \\
A_2 & A_6 & 4, 6 & 1, 3 & 2, 3, 4, 5, 6, 7, 9 & 60 \\
A_3 & A_6 & 4, 6 & 1, 2, 4 & 1, 2, 3, 4, 5, 6, 8 & 57 \\
0 & A_7 & 4,6 & - & 4, 5, 6, 7, 8, 9, 10, 12 & 80 \\
A_1 & A_7 & 4, 6 & 2 & 2, 3, 4, 5, 6, 7, 8, 10 & 67 \\
A_2 & A_7 & 4,6 & 1, 3 & 1, 2, 3, 4, 5, 6, 7, 9 & 62 \\
0 & A_8 & 4, 6 & - & 3, 4, 5, 6, 7, 8, 9, 10, 12 & 84 \\
A_1 & A_8 & 4, 6 & 2 & 1, 2, 3, 4, 5, 6, 7, 8, 10 & 69 \\
0 & A_9 & 4, 6 & - & 2, 3, 4, 5, 6, 7, 8, 9, 10, 12 & 87 \\
0 & A_{10} & 4,6 & - & 1, 2, 3, 4, 5, 6, 7, 8, 9, 10, 12 & 89 \\
\hline
\hline
0 & D_4 & 4, 6 & - & 8, 8, 10, 12, 18 & 72 \\
A_1 &D_4 & 4, 6 & 2 & 6, 6, 8, 10, 14 & 63 \\
2A_1 & D_4 & 4, 6 & 2, 2 & 4, 4, 6, 8, 10 & 54 \\
3A_1 & D_4 & 4, 6 & 2, 2, 2 & 2, 2, 4, 6, 6 & 45 \\
\hline
\end{array}
\]

\clearpage

\noindent\[
\renewcommand{\arraystretch}{1.0}
\begin{array}{|c|c|c|c|c|c|}
\hline
L_0 & L_1 & \text{Eisenstein} & \text{Abelian} & \text{Jacobi} & \text{wt}\, J \\
\hline
A_1 \oplus A_2 & D_4 & 4, 6 & 1, 2, 3 & 3, 3, 5, 7, 8 & 51 \\
A_1 \oplus A_3 & D_4 & 4, 6 & 1, 2, 2, 4 & 2, 2, 4, 6, 6 & 49 \\
A_1 \oplus A_4 & D_4 & 4, 6 & 1, 2, 2, 3, 5 & 1, 1, 3, 4, 5 & 48 \\
2A_1\oplus A_2 & D_4 & 4, 6 & 1, 2, 2, 3 & 1, 1, 3, 4, 5 & 42 \\
A_2 & D_4 & 4, 6 & 1, 3 & 5, 5, 7, 9, 12 & 60 \\
2A_2 & D_4 & 4, 6 & 1, 1, 3, 3 & 2, 2, 4, 6, 6 & 48 \\
A_2 \oplus A_3 & D_4 & 4, 6 & 1, 1, 2, 3, 4 & 1, 1, 3, 4, 5 & 46 \\
A_3 & D_4 & 4, 6 & 1, 2, 4 & 4, 4, 6, 8, 10 & 58 \\
A_4 & D_4 & 4, 6 & 1, 2, 3, 5 & 3, 3, 5, 7, 8 & 57 \\
A_5 & D_4 & 4, 6 & 1, 2, 3, 4, 6 & 2, 2, 4, 6, 6 & 57 \\
A_6 & D_4 & 4, 6 & 1, 2, 3, 4, 5, 7 & 1, 1, 3, 4, 5 & 58 \\
0 & D_5 & 4, 6 & - & 7, 8, 10, 12, 16, 18 & 88 \\
A_1 & D_5 & 4, 6 & 2 & 5, 6, 8, 10, 12, 14 & 75 \\
2A_1& D_5 & 4, 6 & 2, 2 & 3, 4, 6, 8, 8, 10 & 62 \\
3A_1 & D_5 & 4, 6 & 2, 2, 2 & 1, 2, 4, 4, 6, 6 & 49 \\
A_1 \oplus A_2 & D_5 & 4, 6 & 1, 2, 3 & 2, 3, 5, 6, 7, 8 & 57 \\
A_1 \oplus A_3 & D_5 & 4, 6 & 1, 2, 2, 4 & 1, 2, 4, 4, 6, 6 & 53 \\
A_2 & D_5 & 4, 6 & 1, 3 & 4, 5, 7, 9, 10, 12 & 70 \\
2A_2 & D_5 & 4, 6 & 1, 1, 3, 3 & 1, 2, 4, 4, 6, 6 & 52 \\
A_3 & D_5 & 4, 6 & 1, 2, 4 & 3, 4, 6, 8, 8, 10 & 66 \\
A_4 & D_5 & 4, 6 & 1, 2, 3, 5 & 2, 3, 5, 6, 7, 8 & 63 \\
A_5 & D_5 & 4, 6 & 1, 2, 3, 4, 6 & 1, 2, 4, 4, 6, 6 & 61 \\
0 & D_6 & 4, 6 & - & 6, 8, 10, 12, 14, 16, 18 & 102 \\
A_1 & D_6 & 4, 6 & 2 & 4, 6, 8, 10, 10, 12, 14 & 85 \\
2A_1 & D_6 & 4, 6 & 2, 2 & 2, 4, 6, 6, 8, 8, 10 & 68 \\
A_1 \oplus A_2 & D_6 & 4, 6 & 1, 2, 3 & 1, 3, 4, 5, 5, 6, 8 & 59 \\
A_2 & D_6 & 4, 6 & 1, 3 & 3, 5, 7, 8, 9, 10, 12 & 78 \\
A_3 & D_6 & 4, 6 & 1, 2, 4 & 2, 4, 6, 6, 8, 8, 10 & 72 \\
A_4 & D_6 & 4, 6 & 1, 2, 3, 5 & 1, 3, 4, 5, 5, 6, 8 & 65 \\
0 & D_7 & 4, 6 & - & 5, 8, 10, 12, 12, 14, 16, 18 & 114 \\
A_1 & D_7 & 4, 6 & 2 & 3, 6, 8, 8, 10, 10, 12, 14 & 93 \\
2A_1 & D_7 & 4, 6 & 2, 2 & 1, 4, 4, 6, 6, 8, 8, 10 & 72 \\
A_2 & D_7 & 4, 6 & 1, 3 & 2, 5, 6, 7, 8, 9, 10, 12 &  84 \\
A_3 & D_7 & 4, 6 & 1, 2, 4 & 1, 4, 4, 6, 6, 8, 8, 10 & 68 \\
0 & D_8 & 4, 6 & - & 4, 8, 10, 10, 12, 12, 14, 16, 18 & 124 \\
A_1 & D_8 & 4, 6 & 2 & 2, 6, 6, 8, 8, 10, 10, 12, 14 & 99 \\
A_2 & D_8 & 4, 6 & 1, 3 & 1, 4, 5, 6, 7, 8, 9, 10, 12 &  88 \\
0 & D_9 & 4, 6 & - & 3, 8, 8, 10, 10, 12, 12, 14, 16, 18 & 132 \\
A_1 & D_9 & 4, 6 & 2 & 1, 4, 6, 6, 8, 8, 10, 10, 12, 14 & 103 \\
0 & D_{10} & 4, 6 & - & 2, 6, 8, 8, 10, 10, 12, 12, 14, 16, 18 & 138 \\
0 & D_{11}& 4, 6 & - & 1, 4, 6, 8, 8, 10, 10, 12, 12, 14, 16, 18 & 142 \\
\hline
\hline
0 & E_6 & 4, 6 & - & 7, 10, 12, 15, 16, 18, 24 & 120 \\
A_1 & E_6 & 4, 6 & 2 & 5, 8, 10, 11, 12, 14, 18 & 99 \\
A_2 & E_6 & 4, 6 & 1, 3 & 4, 7, 9, 9, 10, 12, 18 & 93 \\
0 & E_7 & 4, 6 & - & 10, 12, 14, 16, 18, 22, 24, 30 & 165 \\
A_1 & E_7 & 4, 6 & 2 & 8, 10, 10, 12, 14, 16, 18, 22 & 132 \\
\hline
\end{array}
\]

\begin{table}[ht]
\caption{\textbf{Predicted} free algebras of meromorphic modular forms for discriminant kernels of lattices $2U \oplus L_0 \oplus L_1$ of type $AE$.}\label{tab:predict}
\renewcommand\arraystretch{1.0}
\noindent\[
\begin{array}{|c|c|c|c|c|c|}
\hline
L_0 & L_1 & \text{Eisenstein} & \text{Abelian} & \text{Jacobi} & \text{wt}\, J \\
\hline
2A_1 & E_6 & 4, 6 & 2, 2 & 3, 6, 7, 8, 8, 10, 12 & 78 \\
3A_1 & E_6 & 4, 6 & 2, 2, 2 & 1, 3, 4, 4, 6, 6, 6 & 57 \\
A_1 \oplus A_2 & E_6 & 4, 6 & 1, 2, 3 & 2, 5, 5, 6, 7, 8, 9 & 69 \\
A_1 \oplus A_3 & E_6 & 4, 6 & 1, 2, 2, 4 & 1, 3, 4, 4, 6, 6, 6 & 61 \\
2A_2 & E_6 & 4, 6 & 1, 1, 3, 3 & 1, 3, 4, 4, 6, 6, 6 & 60 \\
A_3 & E_6 & 4, 6 & 1, 2, 4 & 3, 6, 7, 8, 8, 10, 12 & 82 \\
A_4 & E_6 & 4, 6 & 1, 2, 3, 5 & 2, 5, 5, 6, 7, 8, 9 & 75 \\
A_5 & E_6 & 4, 6 & 1, 2, 3, 4, 6 & 1, 3, 4, 4, 6, 6, 6 & 69 \\
2A_1 & E_7 & 4, 6 & 2, 2 & 6, 6, 8, 8, 10, 10, 12, 14 & 99 \\
3A_1 & E_7 & 4, 6 & 2, 2, 2 & 2, 4, 4, 4, 6, 6, 6, 6 & 66 \\
A_1 \oplus A_2 & E_7 & 4, 6 & 1, 2, 3 & 4, 5, 6, 7, 7, 8, 9, 10 & 84 \\
A_1 \oplus A_3 & E_7 & 4, 6 & 1, 2, 2, 4 & 2, 4, 4, 4, 6, 6, 6, 6 & 70 \\
A_2 & E_7 & 4, 6 & 1, 3 & 7, 8, 9, 10, 12, 13, 15, 18 & 117 \\
2 A_2 & E_7 & 4, 6 & 1, 1, 3, 3 & 2, 4, 4, 4, 6, 6, 6, 6 & 69 \\
A_3 & E_7 & 4, 6 & 1, 2, 4 & 6, 6, 8, 8, 10, 10, 12, 14 & 103 \\
A_4 & E_7 & 4, 6 & 1, 2, 3, 5 & 4, 5, 6, 7, 7, 8, 9, 10 & 90 \\
A_5 & E_7 & 4, 6 & 1, 2, 3, 4, 6 & 2, 4, 4, 4, 6, 6, 6, 6 & 78 \\
\hline
\end{array}
\]
\end{table}

\subsection{Non-free algebras of holomorphic modular forms related to reducible root lattices}\label{sec:tables_hol}

In the following we describe the $26$ non-free algebras of holomorphic modular forms determined in \S \ref{sec:non-free}, associated to the discriminant kernel of $2U+R$ for a reducible root lattice $R$, by describing the weights of a minimal system of generators and the dimensions of the spaces of modular forms. The notation $k^n$ in the weights of the generators means that there are $n$ generators of weight $k$. The rational function associated to each root system $R$ is the Hilbert--Poincar\'e series $$\mathrm{Hilb}\, M_*(\widetilde{\Orth}^+(2U \oplus R)) = \sum_{k=0}^{\infty} \mathrm{dim}\, M_k(\widetilde{\Orth}^+(2U \oplus R)) t^k,$$ from which the dimensions can be extracted easily.

\vspace{3mm}

\begin{small}
\begin{enumerate}
    \item $R=2A_1$: weights of generators: 4, 6, 8, 10, 10, 12;
    $$
    \frac{1 + t^{10}}{(1 - t^4)(1-t^6)(1-t^8)(1-t^{10})(1-t^{12})}
    $$
    \item $R=3A_1$: weights of generators: $4, 6^2, 8^3, 10^3, 12$;
    $$
    \frac{1 + 2t^8 + 2t^{10} + t^{18}}{(1 - t^4)(1-t^6)^2(1-t^8)(1-t^{10})(1-t^{12})}
    $$
    \item $R=4A_1$: weights of generators: $4^2, 6^5, 8^6, 10^4, 12$;
    $$
    \frac{1 + 3t^6 + 5t^8 + 3t^{10} + 3t^{14} + 5t^{16} + 3t^{18} + t^{24}}{(1 - t^4)^2(1-t^6)^2 (1-t^8)(1-t^{10})(1-t^{12})}
    $$
    \item $R=A_1\oplus A_2$: weights of generators: $4, 6, 7, 8, 9, 10, 10, 12$;
    $$
    \frac{1+ t^{10} - t^{17} - t^{19}}{(1 - t^4)(1-t^6)(1-t^7)(1-t^8)(1-t^9)(1-t^{10})(1-t^{12})}
    $$
    \item $R=2A_1\oplus A_2$: weights of generators: $4, 5, 6^2, 7^2, 8^3, 9, 10^3, 12$;
    $$
    \frac{1 + t^5 + t^7 + 2t^8 + 3t^{10} + t^{12} - 2t^{17} + t^{18} - 2t^{19} - t^{22} - t^{24} - t^{25} - 2t^{27} - t^{29}}{(1-t^4)(1 - t^6)^2(1-t^7)(1-t^8)(1-t^9)(1-t^{10})(1-t^{12})}
    $$
    \item $R=2A_2$: weights of generators: $4, 6^2, 7^2, 8, 9^2, 10^2, 12$;
    $$
    \frac{1 + t^4 + t^7 + t^8 + t^9 + t^{10} + t^{11} + t^{12} + t^{13} + t^{14} + t^{15} +t^{18} + t^{22}}{(1-t^6)^2(1-t^7)(1-t^8)(1-t^9)(1-t^{10})(1-t^{12})}
    $$
    \item $R=3A_2$: weights of generators: $3, 4^4, 5^3, 6^5, 7^6, 8^3, 9^3, 10^3, 12$;
    \begin{align*}
    &( 1 + 3t^4 + 3t^5 + 3t^6 + 5t^7 + 5t^8 + 8t^9 + 8t^{10} + 9t^{11} + 12t^{12} + 12t^{13}\\
    &+ 15t^{14} + 16t^{15} + 16t^{16} +14t^{17} + 16t^{18} + 16t^{19} + 15t^{20} + 12t^{21}\\
    &+ 12t^{22} + 9t^{23} + 8t^{24} + 8t^{25} + 5t^{26} + 5t^{27} + 3t^{28} + 3t^{29} + 3t^{30} + t^{34} ) \\ 
    &/ (1 - t^3)(1-t^4)(1-t^6)^2(1-t^7)(1-t^8)(1-t^9)(1-t^{10})(1-t^{12})    
    \end{align*}
    \item $R=A_1\oplus A_3$: weights of generators: $4, 6, 6, 7, 8, 8, 9, 10, 10, 12$;
    $$
    \frac{1 - t^2 + t^6 - t^{15}}{(1-t^2)(1-t^4)(1 - t^6)(1-t^7)(1-t^8)(1-t^9)(1-t^{10})(1-t^{12})}
    $$
    \item $R=2A_1\oplus A_3$: weights of generators: $4^2, 5, 6^4, 7^2, 8^4, 9, 10^3, 12$;
    \begin{align*}
     &(1 + t^5 + 2t^6 + t^7 + 3t^8 + 3t^{10} + t^{12} +t^{14} - t^{15} + 2t^{16} - 2t^{17}\\
     &+ t^{18} - 2t^{19} - t^{21} -t^{22} - 2t^{23} - t^{24} - 3t^{25} - 2t^{27} - t^{29}) \\ 
     &/(1-t^4)^2(1-t^6)^2(1-t^7)(1-t^8)(1-t^9)(1-t^{10})(1-t^{12})
    \end{align*}
    \item $R=A_2\oplus A_3$: weights of generators: $4, 5, 6^3, 7^2, 8^2, 9^2, 10^2, 12$;
    \begin{align*}
      \frac{1 + t^5 + t^6 + t^7 + t^8 + t^9 + 2t^{10} + t^{11} + t^{12} + t^{13} + t^{14} + t^{15} + t^{16} + t^{18} + t^{20} +t^{21}}{(1 - t^6)(1-t^7)(1-t^8)(1-t^9)(1-t^{10})(1-t^{12})}
    \end{align*}
    \item $R=2A_3$: weights of generators: $4^2, 5^2, 6^4, 7^2, 8^3, 9^2, 10^2, 12$;
    \begin{align*}
    &(1 + 2t^5 + 2t^6 + t^7 + 2t^8 + t^9 + 3t^{10} + 2t^{11} +2t^{12}  + 2t^{13} \\
     &+ 2t^{14} + 2t^{15} + 3t^{16} + t^{17} + 2t^{18} + t^{19} + 2t^{20} + 2t^{21}  +t^{26}) \\ 
     &/(1-t^4)^2(1-t^6)^2(1-t^7)(1-t^8)(1-t^9)(1-t^{10})(1-t^{12})
    \end{align*}
    \item $R=A_1\oplus A_4$: weights of generators: $4, 5, 6, 6, 7, 7, 8, 8, 9, 10, 12$;
    $$
    \frac{1+t^7+t^8+t^{10}-t^{13}-t^{15}-t^{17}-t^{19}}{(1-t^4)(1 - t^5)(1-t^6)^2(1-t^7)(1-t^8)(1-t^9)(1-t^{10})(1-t^{12})}
    $$
    \item $R=A_2\oplus A_4$: weights of generators: $4^2, 5^2, 6^3, 7^3, 8^2, 9^2, 10^2, 12$;
    \begin{align*}
    &(1 + t^4 + t^5 + t^6 + 2t^7 + 2t^8 + 2t^9 + 2t^{10} + 2t^{11} + 2t^{12} + 2t^{13}\\
    &+ 2t^{14} + 2t^{15} + 2t^{16} + t^{17} + 2t^{18} + t^{19} + t^{20} + t^{21} + t^{22}) \\ 
    &/ (1-t^4)(1-t^5)(1-t^6)^2(1-t^7)(1-t^8)(1-t^9)(1-t^{10})(1-t^{12})
    \end{align*}
    \item $R=A_1\oplus A_5$: weights of generators: $4^2, 5, 6^3, 7^2, 8^2, 9, 10^2, 12$;
    \begin{align*}
    \frac{1 + t^6 + t^7 + t^8 + t^{10} - t^{11} - t^{13} -t^{15} - t^{17} - t^{19}}{(1 - t^4)^2(1-t^5)(1-t^6)^2(1-t^7)(1-t^8)(1-t^9)(1-t^{10})(1-t^{12})}
    \end{align*}
    \item $R=A_1\oplus 2A_2$: weights of generators: $4^2, 5^2, 6^3, 7^4, 8^3, 9^2, 10^3, 12$;
    \begin{align*}
    &(1 + t^{4} + 2t^{5} + t^{6} + 3t^{7} + 3t^{8} + 3t^{9} + 5t^{10} + 3t^{11} + 5t^{12} + 4t^{13} + 5t^{14} + 6t^{15} + 5t^{16}\\
    & + 4t^{17} + 6t^{18} + 4t^{19} + 5t^{20} + 4t^{21} + 3t^{22} + 3t^{23} + t^{24} + 2t^{25} + t^{26} + t^{28} + t^{30})\\
    & / (1 - t^4)(1 - t^6)^2(1 - t^7)(1 - t^8)(1 - t^9)(1 - t^{10})(1 - t^{12})
    \end{align*}
    \item $R=A_1\oplus D_4$: weights of generators: $4, 6^3, 8^3, 10^2, 12, 14, 16, 18$;
    \begin{align*}
    \frac{1 + 2t^6 + t^8 + t^{10} + 3t^{12} + t^{14} + 2t^{16} + 3t^{18} + t^{20} + t^{22} + t^{24}}{(1 - t^4)(1-t^6)(1-t^8)^2 (1-t^{10})(1-t^{12})(1-t^{14})(1-t^{18})}
    \end{align*}
    \item $R=2A_1\oplus D_4$: weights of generators: $4^3, 6^6, 8^5, 10^4, 12^3, 14^3, 16^2, 18$;
    \begin{align*}
    &(1 - t^{2} + 2t^{4} + 3t^{6} + t^{8} + 5t^{10} + 7t^{12} + 4t^{14} + 10t^{16} + 8t^{18} \\
    &+ 7t^{20} + 9t^{22} + 8t^{24} + 6t^{26} + 6t^{28} + 4t^{30} + 3t^{32} + t^{34} + t^{36}) \\ 
    &/(1-t^2)(1-t^4)(1-t^6)(1-t^8)^2(1-t^{10})(1-t^{12})(1-t^{14})(1-t^{18})
    \end{align*}
    \item $R=A_2\oplus D_4$: weights of generators: $4, 5^2, 6^3, 7, 8^3, 9, 10^2, 12^2, 13, 14, 15, 16, 18$;
    \begin{align*}
    &(1 + t^{5} + t^{6} + t^{7} + t^{8} + t^{9} + 2t^{10} + 3t^{12}  + 2t^{14} + t^{15} + 2t^{16} + t^{17} + 2t^{19} - t^{20} + t^{21} - t^{22}\\
    &- t^{23} - t^{24} - t^{25} - t^{27} - t^{29} - 2t^{30} - 2t^{32} - t^{33} - 2t^{34} - t^{35} - 2t^{36} - t^{37} - t^{38} - t^{39}) \\
    & / (1 - t^4)(1-t^5)(1-t^6)^2(1-t^8)^2(1-t^{10})(1-t^{12})(1-t^{14})(1-t^{18})
    \end{align*}
    \item $R=A_3\oplus D_4$: weights of generators: $4^3, 5^2, 6^4, 7, 8^4, 9, 10^3,11, 12^3, 13, 14^2, 15, 16, 18$;
    \begin{align*}
    &(1 - t^{3} + t^{4} + t^{5} + 3t^{6} + 2t^{8} - 2t^{9} + 4t^{10} - t^{11} + 8t^{12} - 3t^{13} + 6t^{14} - 6t^{15} + 8t^{16} - 4t^{17}\\ 
    &+ 9t^{18} - 6t^{19} + 7t^{20} - 7t^{21} + 7t^{22} - 7t^{23} + 9t^{24} - 7t^{25} + 7t^{26} - 9t^{27} + 5t^{28} - 9t^{29}\\
    &+ 3t^{30} - 7t^{31} + 3t^{32} - 5t^{33} + t^{34} - 5t^{35} + t^{36} - 2t^{37} + t^{38} - 2t^{39} - 2t^{41} - t^{42} - t^{43}) \\
    &/ (1-t^3)(1-t^4)^2(1-t^5)(1-t^6)(1-t^8)^2 (1-t^{10})(1-t^{12})(1-t^{14})(1-t^{18})
    \end{align*}
    \item $R=2D_4$: weights of generators: $4^5, 6^5, 8^5, 10^8, 12^6, 14^6, 16^2, 18^2$;
    \begin{align*}
    &(1 + 2t^{4} + 3t^{6} + 5t^{8} + 8t^{10} + 11t^{12} + 10t^{14} + 12t^{16} + 10t^{18} + 13t^{20} + 9t^{22}  + 10t^{24} + 4t^{26}  - 4t^{28} \\
    & - 10t^{30} - 9t^{32} - 13t^{34} - 10t^{36} - 12t^{38} - 10t^{40} - 11t^{42}  - 8t^{44} - 5t^{46} - 3t^{48} - 2t^{50} - t^{54}) \\ 
    & /(1-t^4)^3(1-t^6)^2 (1-t^8)^2 (1-t^{10})^2 (1 - t^{12})(1-t^{14})(1-t^{18})
    \end{align*}
    \item $R=A_1\oplus D_5$: weights of generators: $4, 5, 6^2, 7, 8^2, 10^2, 12^2, 14^2, 16^2, 18$;
    \begin{align*}
    \frac{1 + t^6 + t^7 + t^8 + t^{10} + 2t^{12} + 3t^{14} + 2t^{16} + 2t^{18} + 2t^{20} + t^{22} + t^{24} + t^{26}}{(1 - t^4)(1-t^5)(1-t^6)(1-t^8)(1-t^{10})(1-t^{12})(1-t^{14})(1-t^{16})(1-t^{18})}
    \end{align*}
    \item $R=A_2\oplus D_5$: weights of generators: $4^2, 5^2, 6^2, 7^2, 8^2, 9, 10^3, 11, 12^3, 13^2, 14^2, 15, 16^2, 18$;
    \begin{align*}
    &(1 - t^{3} + t^{4} + t^{5} + t^{6} + t^{7} + t^{8} + t^{9} + t^{10} + 2t^{11} + 3t^{12} + 2t^{13}  + 3t^{14} \\
    & + t^{15} + 3t^{16} + t^{17} + 4t^{18} + 2t^{19} + 4t^{20} + 2t^{21} + 4t^{22} + 2t^{23} + 4t^{24} \\ 
    &+ 2t^{25} + 3t^{26} + t^{27} + 2t^{28} + t^{29} + 2t^{30} + 2t^{31} + t^{32} + t^{33} + t^{34} + t^{35}) \\ & /(1-t^3)(1-t^4)(1-t^5)(1-t^6)(1-t^8)(1-t^{10})(1-t^{12})(1-t^{14})(1-t^{16})(1-t^{18})
    \end{align*}
    \item $R=A_1\oplus D_6$: weights of generators: $4^2, 6^3, 8^2, 10^3, 12^3, 14^3, 16^2, 18$;
    \begin{align*}
    \frac{1+t^4+t^6+2t^8+2t^{10}+4t^{12}+4t^{14}+4t^{16}+3t^{18}+3t^{20}+2t^{22}+2t^{24}+t^{26}+t^{28}}{(1-t^4)(1-t^6)^2(1-t^8)(1-t^{10})^2(1-t^{12})(1-t^{14})(1-t^{16})(1-t^{18})}
    \end{align*}
    \item $R=A_1\oplus E_6$: weights of generators: $4, 5, 6, 7, 8, 10^2, 11, 12^2, 13, 14^2, 15, 16^2, 18^2, 20, 22, 24$;
    \begin{align*}
    &(1 + t^{5} + t^{7} + t^{8} + 2t^{10} + 2t^{12} + 2t^{13} + 2t^{14}  + 2t^{15} + 2t^{16} + 2t^{17} + 3t^{18} + 2t^{19} + 4t^{20} + 2t^{21} + 4t^{22} \\
    & + 2t^{23} + 3t^{24} + 2t^{25} + 3t^{26} + 2t^{27} + 3t^{28} + t^{29} + 3t^{30}  + t^{31} + 2t^{32} + t^{33} + t^{34} + t^{35} + t^{36} + t^{38} + t^{40}) \\
    & / (1-t^4)(1-t^6)(1-t^{10})(1-t^{11})(1-t^{12})(1-t^{14})(1-t^{15})(1-t^{16})(1-t^{18})(1-t^{24})
    \end{align*}
    \item $R=A_2\oplus E_6$: weights of generators: $4^2$, 5, 6, $7^2$, 8, 9, $10^3$, 11, $12^3$, $13^2$, $14^2$, $15^2$, $16^3$, 17, $18^3$, 19, 20, 21, 22, 24;
    \begin{align*}
    &(1 + t^{4} + t^{5} + 2t^{7} + 2t^{8} + 2t^{9} + 4t^{10} + 3t^{11} + 6t^{12} + 6t^{13} + 6t^{14} + 8t^{15} + 9t^{16} + 11t^{17}\\
    & + 12t^{18} + 13t^{19} + 16t^{20} + 16t^{21} + 19t^{22} + 18t^{23} + 21t^{24} + 23t^{25} + 23t^{26} + 24t^{27} \\
    & + 27t^{28} + 26t^{29}  + 29t^{30}+ 28t^{31} + 29t^{32} + 28t^{33} + 30t^{34} + 28t^{35} + 28t^{36} + 28t^{37} \\
    & + 28t^{38} + 25t^{39} + 26t^{40}  + 22t^{41} + 22t^{42} + 21t^{43} + 18t^{44} + 16t^{45} + 16t^{46} + 13t^{47} \\
    &+ 11t^{48} + 10t^{49} + 9t^{50} + 7t^{51} + 6t^{52} + 4t^{53} + 3t^{54} + 3t^{55} + 2t^{56} + t^{57} + t^{58} + t^{59}) \\
    & / (1-t^4)(1-t^6)(1-t^9)(1-t^{10})(1-t^{11})(1-t^{12})(1-t^{14})(1-t^{15})(1-t^{16})(1-t^{18})(1-t^{24}) 
    \end{align*}
    \item $R=A_1\oplus E_7$: weights of generators: $4, 6, 8, 10^3, 12^3, 14^3, 16^3, 18^3,20^2, 22^3, 24^2, 26, 28, 30$;
    \begin{align*}
    \frac{1 - t^2 + t^8 + t^{10} + t^{16} + t^{18} + t^{20} + t^{26} + t^{28} + t^{36}}{(1 - t^2)(1-t^4)(1-t^6)(1-t^{10})(1-t^{12})(1-t^{14})(1-t^{16})(1-t^{18})(1-t^{22})(1-t^{24})(1-t^{30})}
    \end{align*}
\end{enumerate}
\end{small}

\bibliographystyle{plainnat}
\bibliofont
\bibliography{refs}

\end{document}